\documentclass[11pt,reqno]{amsart}
\usepackage{amsmath,latexsym,amssymb,verbatim,amsbsy,amsthm,mathrsfs,tikz,times}
\usepackage[margin=1in]{geometry}

\title[Long time dynamics of forced critical SQG]{Long time dynamics of forced critical SQG}
\date{\today}

\author{Peter Constantin}
\address{Department of Mathematics, Princeton University, Princeton, NJ 08544}
\email{const@math.princeton.edu}

\author{Andrei Tarfulea}
\address{Department of Mathematics, Princeton University, Princeton, NJ 08544}
\email{tarfulea@math.princeton.edu}

\author{Vlad Vicol}
\address{Department of Mathematics, Princeton University, Princeton, NJ 08544}
\email{vvicol@math.princeton.edu}

\usepackage{color}
\usepackage[colorlinks=true, pdfstartview=FitV, linkcolor=blue, citecolor=blue, urlcolor=blue]{hyperref}

\theoremstyle{plain}
\newtheorem{theorem}{Theorem}[section]
\newtheorem{definition}[theorem]{Definition}
\newtheorem{lemma}[theorem]{Lemma}
\newtheorem{proposition}[theorem]{Proposition}
\newtheorem{corollary}[theorem]{Corollary}
\theoremstyle{definition}
\newtheorem{remark}[theorem]{Remark}

\def\tilde{\widetilde}
\numberwithin{equation}{section}

\renewcommand\hat{\widehat}
\def\ZZ{{\mathbb Z}}
\def\RR{{\mathbb R}}
\def\TT{{\mathbb T}}

\def\RSZ{{\mathcal R}}

\def\KK{{\mathcal K}}
\def\AA{{\mathcal A}}
\def\BB{{\mathcal B}}

\def\ra{\rangle}
\def\la{\langle}

\newcommand{\eps}{\varepsilon}
\renewcommand{\phi}{\varphi}

\def\intint{\int\!\!\!\int}

\begin{document}

\begin{abstract}
We prove the existence of a compact global attractor for the dynamics of the forced critical surface quasi-geostrophic equation (SQG) and prove that it has finite fractal (box-counting) dimension.  In order to do so we give a new proof of global regularity for critical SQG. The main ingredient is the nonlinear maximum principle in the form of a nonlinear lower bound on the fractional Laplacian, which is used to bootstrap the regularity directly from $L^\infty$ to $C^\alpha$, without the use of De Giorgi techniques. We prove that for large time, the norm of the solution measured in a sufficiently strong topology becomes bounded with bounds that depend solely on norms of the force, which is assumed to belong merely to $L^\infty \cap H^1$. Using the fact that the solution is bounded independently of the initial data after a transient time, in spaces conferring enough regularity, we prove the existence of a compact absorbing set for the dynamics in $H^1$, obtain the compactness of the linearization and the continuous differentiability of the solution map. We then prove exponential decay of high yet finite dimensional volume elements in $H^1$ along solution trajectories, and use this property to bound the dimension of the global attractor.  

\end{abstract}



\maketitle

\setcounter{tocdepth}{1}
\tableofcontents

\section{Introduction}\label{sec:intro}

Nonlinear forced dissipative partial differential equations can generate physical space patterns which evolve in a temporally complex manner. As parameters are varied, the dynamics may transition from simple to chaotic and ultimately to fully turbulent. Nevertheless, several forced nonlinear dissipative PDE of hydrodynamic origin have been shown to have finite dimensional long time behavior. This has been proved if certain minimal conditions are satisfied.  Chief among them is the property that linearizations about time evolving solutions are dominated in a certain sense by the linear dissipative part. This is the case for semilinear dissipative PDE such as the Navier-Stokes system in 2D, or subcritical quasilinear damped systems. In this paper we study the long time behavior of a critical quasilinear system, where this property is far from obvious.

The forced, critically dissipative surface quasi-geostrophic  (SQG) equation is
\begin{align}
&\partial_t \theta + u \cdot \nabla \theta + \kappa \Lambda \theta = f \label{eq:SQG:1}\\
&u = \RSZ^\perp \theta = (-\RSZ_2\theta,\RSZ_1\theta) \label{eq:SQG:2}\\
&\theta(\cdot,0) = \theta_0 \label{eq:SQG:3}
\end{align}
where $\Lambda=  (-\Delta)^{1/2}$, $\RSZ_j = \partial_j \Lambda^{-1}$ is the $j$th Riesz transform, and the equations are set on the periodic domain $\TT^2 = [-\pi,\pi]^2$.
Here $\kappa>0$ is a positive constant, $\theta_0$ is the initial condition, and $f=f(x)$ is a time-independent force. The force is assumed
to belong to  $L^\infty(\TT^2) \cap H^1(\TT^2)$, while the initial data is assumed to belong to $H^1(\TT^2)$. We consider forcing and initial data of zero average, that is
$
\int_{\TT^2} f(x) dx = \int_{\TT^2} \theta_0(x) dx = 0
$
which immediately implies that a solution $\theta$ of \eqref{eq:SQG:1}--\eqref{eq:SQG:3} obeys
\[
\int_{\TT^2} \theta(x,t) dx = 0
\]
for any $t\geq 0$. Throughout this manuscript we consider mean-zero (zero average) solutions.

The SQG equation describes the evolution of a surface temperature field $\theta$ in a rapidly rotating, stably stratified fluid with potential vorticity~\cite{ConstantinMajdaTabak94,HeldPierrehumbertGarnerSwanson95}. From the mathematical point of view, the non-dissipative SQG equations (the system \eqref{eq:SQG:1}--\eqref{eq:SQG:3} with $\kappa = 0$) have properties that are similar to those of the 3D Euler equations in vorticity form~\cite{ConstantinMajdaTabak94}, and yet one may for instance prove the global existence of finite energy weak solutions~\cite{Resnick95}, albeit for completely different reasons than for 2D Euler. Initial numerical simulations~\cite{ConstantinMajdaTabak94,OhkitaniYamada97,ConstantinNieSchorghofer98} have furthermore exhibited highly nonlinear features in the inviscid evolution such as the formation of sharp fronts. The latter issue has been studied analytically~\cite{Cordoba98,CordobaFefferman01,CordobaFefferman02,Chae08,FeffermanRodrigo11} and also in more recent numerical simulations~\cite{CordobaFontelosManchoRodrigo05,DengHouLiYu06,OhkitaniSakajo10,ConstantinLaiSharmaTsengWu12}. Whether solutions of the inviscid SQG equations can develop singularities in finite time remains an open problem.

If in the three-dimensional quasi-geostrophic equations~\cite{Pedlosky82} we take into account damping (Ekman pumping at the boundary) and external sources we arrive at the system the system \eqref{eq:SQG:1}--\eqref{eq:SQG:3}. The square root of the Laplacian $\Lambda = (-\Delta)^{1/2}$ naturally arises as the Dirichlet to Neumann map for the surface temperature $\theta$. See e.g.~\cite{Constantin02} and references therein. 

Among the dissipative operators $\Lambda^\gamma$ with $0 < \gamma \leq 2$, the exponent $\gamma=1$ appearing in \eqref{eq:SQG:1} is not just physically motivated but also mathematically challenging. In view of the underlying scaling invariance associated with \eqref{eq:SQG:1} (see Remark~\ref{rem:box} below) and the available conservation laws (the $L^\infty$ maximum principle), it is customary to refer to the dissipation power $\gamma=1$ as {\em critical}, although the question of whether a dramatic change in the behavior of the solution occurs for $\gamma<1$ (the {\em supercritical} case) remains an open problem. 

From the mathematical point of view, the issue of {\em global regularity} {vs} {\em finite-time blowup} for the fractionally dissipative SQG equation has been extensively studied over the past two decades. The subcritical case ($\gamma>1$) has been essentially resolved in \cite{Resnick95,ConstantinWu99}.  See also~\cite{Wu01,SchonbekSchonbek03,Ju04,Marchand08a,DongLi08} and references therein for further qualitative properties of weak and strong solutions in the subcritical case.

The issue of global regularity for the critical SQG equation is more challenging since the balance between the nonlinearity and the dissipative term in \eqref{eq:SQG:1}--\eqref{eq:SQG:3} is the same no matter at which scales one zooms in. This is why treating the equation as a perturbation of the fractional heat equation fails to be useful for large initial data (and nonlinearity). In order to obtain bounds on the solution one has to appeal to finer methods.  Two different such sophisticated tools have been introduced independently in \cite{CaffarelliVasseur10,KiselevNazarovVolberg07}  to show that solutions of \eqref{eq:SQG:1}--\eqref{eq:SQG:3} do not develop singularities in finite time, for arbitrarily large initial data. Before these works, the global regularity was only known for initial data which is small in the $L^\infty$ norm~\cite{ConstantinCordobaWu01} and other scaling critical spaces~\cite{ChaeLee03,CordobaCordoba04,Wu04,Miura06,Ju07,ChenMiaoZhang07,HmidiKeraani07}.

The proof of global regularity in~\cite{KiselevNazarovVolberg07} is based on constructing a family of Lipschitz moduli of continuity with the property that if the initial data obeys such a modulus so will the solution of the critical SQG equation for all later times, and so that this family behaves nicely under rescaling. See also~\cite{DongDu08} for applications of the modulus of continuity method to the case of the whole space, \cite{FriedlanderPavlovicVicol09} for global regularity in the presence of a Lipschitz forcing term, \cite{KiselevNazarov10} in the presence of a linear dispersive force, and~\cite{SilvestreVicol12} for regularity of critical linear drift-diffusion equations.

The proof of~\cite{CaffarelliVasseur10} on the other hand employs the ideas of De Giorgi iteration to the case of a nonlocal parabolic equation, and shows that bounded weak solutions must in fact instantly become H\"older continuous (and hence classical). We also refer to \cite{CaffarelliVasseur10b,CaffarelliChanVasseur11} and references therein for applications of the De Giorgi ideas to more degenerate nonlocal parabolic equations, and \cite{ConstantinWu09,Silvestre10c} for linear fractional advection-diffusion equations (the latter results for linear equations are in fact sharp~\cite{SilvestreVicolZlatos13}).

In an attempt to find a bridge between these two proofs, in~\cite{KiselevNazarov09} a completely different proof of global regularity for critical SQG was obtained. The proof of~\cite{KiselevNazarov09} relies on keeping track of the action of a dual SQG evolution on a suitable family of test functions (Hardy molecules), in order to control the evolution of a H\"older norm.

In the three proofs of global regularity for \eqref{eq:SQG:1}--\eqref{eq:SQG:2} mentioned above, the quantitative way in which the dissipation is dominating the nonlinear term is not immediately transparent. In order to shed more light on this issue  \cite{ConstantinVicol12} introduced {\em nonlinear} lower bounds for {\em linear} nonlocal operators such as the fractional Laplacian. The {nonlinear nature} of these lower bounds appears because of the conservation laws which are available a priori for solutions. The inequalities bound the dissipative term from below and the nonlinear term from above in a way which makes it transparent why the {\em size} of the dissipation dominates the {\em size} of the nonlinearity. In \cite{ConstantinVicol12} it was shown that the nonlinear lower bound for $\Lambda$ is almost sufficient: it implies the global regularity for the critical SQG equation, but only if we additionally know that the solution has the {\em only small shocks} property (there exits $0 < \delta < 1$ and $L>0$ such that $|\theta(x,t) - \theta(y,t)| \leq \delta$ whenever $|x-y|\leq L$). It was proved that if the initial data has only small shocks, the size of these shocks can at most double as time evolves. This extra step is however quite technical as it uses once more the nonlinear lower bound for the fractional Laplacian, and it depends strongly on the nature of the nonlinearity.

The main contributions of this paper are as follows. We give a fully transparent, new and final proof of global regularity for the critical SQG equation, which relies {\em solely} on the nonlinear lower bound for the fractional Laplacian (Theorem~\ref{thm:global:regularity}). The bounds we obtain on the solution incorporate an explicit decay from the initial data in both weak and {\em strong topologies} (Theorem~\ref{thm:H:3/2:absorbing}). Using this crucial ingredient we prove the existence of a compact global attractor for the forced SQG dynamics (Theorem~\ref{thm:attractor:existence}), which has finite box-counting and hence Hausdorff dimension (Theorem~\ref{thm:finite:attractor}). To the best of our knowledge this is the first such result for systems that are both {\em quasilinear} and {\em critical}.

The main ingredient in the proof of the results mentioned above is the nonlinear lower bound for the fractional Laplacian established in \cite{ConstantinVicol12}. One of the main new ideas over \cite{ConstantinVicol12} is to replace the smallness coming from the ``only small shocks'' property (for which time dependence is difficult to control) with the smallness of the $\alpha$ at the exponent of the H\"older space $C^\alpha$. This simple idea is also implicitly present in the De Giorgi's methodology: the size of the H\"older exponent obtained in the bootstrap step depends on the size of the data in $L^\infty$, the ellipticity constant, and the size of the force. The advantage of the method outlined below is that the propagation of $C^\alpha$ regularity, for some {\em sufficiently small $\alpha$}, is achieved directly, without appealing to the complex De Giorgi iteration technique (Theorem~\ref{thm:Holder:propagation}). Once the solution is H\"older continuous, the nonlinear lower bound also may be used to immediately obtain the higher regularity of solutions. To better outline the method introduced here, we first apply it to the critical Burgers equation in Section~\ref{sec:Burgers}, and then to the critical SQG equation in Section~\ref{sec:global}.

We note that the proof of propagation of H\"older continuity also works in the presence of a force which lies merely in $L^\infty \cap H^1$ (the modulus of continuity proof given in \cite{FriedlanderPavlovicVicol09} requires Lipschitz forcing). Moreover, the argument is {\em dynamic} rather than in the spirit of a maximum principle, in the sense that the size of the H\"older norm of the solution is not just bounded in terms of the initial data and force, but as time evolves its size is estimated solely in terms of the force (Theorem~\ref{thm:H:3/2:absorbing} and Lemma~\ref{lem:C:alpha:absorbing}). In the unforced case this amounts to proving the decay of the H\"oder norm and higher Sobolev norms.

Armed with a proof of regularity that is dynamic, we can study the long time dynamics of solutions of the forced SQG. The decay in the unforced case has been addressed in~\cite{Dong10}. The behavior of the long-time averages along solutions of the critical SQG equations via viscous approximations was addressed in ~\cite{ConstantinTarfuleaVicol13}, where we have obtained the absence of anomalous dissipation. We note that the later result may also be proven using the estimates of this paper, but necessitates higher regularity on the force than in
\cite{ConstantinTarfuleaVicol13}.

In the second part of the paper we address the existence of a compact global attractor, and prove that it has finite box-counting dimension, and a forteriori finite Hausdorff dimension. The space-periodic setting is needed for this purpose. The {critical} SQG equation is {\em quasilinear}. To the best of our knowledge, until now all proofs of finite dimensionality have been done either for dissipative or damped {\em semilinear} equations, or {\em subcritical} quasilinear equations. See e.g. the works \cite{FoiasProdi67,FoiasManleyTemamTreve83,FoiasTemam84,ConstantinFoias85,ConstantinFoiasTemam85,ConstantinFoiasManleyTemam85,FoiasSellTemam85,Constantin87,ConstantinFoiasTemam88,FoaisSellTemam88,DoeringGibbon91,FoiasTiti91,JonesTiti92,Kukavica92,JonesTiti93,FoiasKukavica95,CockburnJonesTiti97,GibbonTiti97,Ziane97} for the 2D periodic Navier-Stokes equations and related systems, the books \cite{ConstantinFoias88,Hale88,Ladyzhenskaya91,BabinVishik92,Temam97,FoiasManleyRosaTemam01,Robinson01,ChepyzhovVishik02}, and references therein.  In the context of the dissipative SQG equation, the global attractor has been addressed previously addressed only for the subcritical regime: \cite{Ju05a} proved the existence of compact global attractor and \cite{WangTang13} showed it has finite dimensionality fractal dimension (see also \cite{Berselli02} regarding the notion of a weak attractor).

In order to prove finite dimensionality we establish the existence of a compact absorbing set in phase space.  We work in the phase space $H^1$, which is the largest Hilbert space in which uniqueness of weak solutions of SQG is currently available. Weak solutions are known to exist for initial data in $L^2$ but their uniqueness is not known~\cite{Resnick95}. It is known that solutions of the unforced SQG with initial data in $H^1$ exist for short time~\cite{Ju07}, as a result of weak-strong stability of the equation in $H^{1+\epsilon}$. The time of existence however depends on the initial function and {\em not only on its norm}. There is no lower bound on the time of existence based solely on the $H^1$ norm. Nevertheless, the solution is unique, and becomes instantly smooth~\cite{Miura06,Ju07,Dong10}. The same result can be proved for smooth forced SQG. We need only a limited amount of smoothness, in particular $C^{\alpha}$, for small $\alpha>0$. We use the nonlinear maximum principle~\cite{ConstantinVicol12} (estimate~\eqref{eq:D:h:lower} below) to show the global persistence of a $C^{\alpha}$ norm, with $\alpha$ small compared to the $L^{\infty}$ norm of the solution (Theorem~\ref{thm:Holder:propagation}). We use the $L^p$ Poincar\'{e} inequality for the fractional Laplacian (Proposition~\ref{prop:Poincare} below) to show that solutions become bounded in $L^{\infty}$ with a bound that depends only on the norm of the force and not on the initial data, after a time that depends on the initial 
$H^1$ data. We apply again the new proof of global existence to show that the solution becomes bounded in a $C^{\alpha}$ space with both $\alpha$ and the solution bound depending only on norms of the forces (Lemma~\ref{lem:C:alpha:absorbing}). Since we now have supercritical information, we use the nonlinear maximum principle again in its $C^{\alpha}$ variant to show that the solution becomes bounded in $H^{3/2}$. The upshot is that there exists a compact absorbing set ${\mathcal{B}}$ for the evolution of SQG in $H^{1}$ which is a bounded set in $H^{3/2}$ (Theorem~\ref{thm:H:3/2:absorbing}). This means that for any initial data $\theta_0\in H^1$ there exists a time $t(\theta_0)$ after which the unique solution $S(t)\theta_0$ with initial datum $\theta_0$ belongs to this absorbing set ${\mathcal{B}}$, which in turn depends on the forces alone and not on the initial data.  We show additional properties of the solution after this transient time: higher regularity, specifically  $S(t)\theta_0\in H^2$ for almost all $t>t(\theta_0)$, continuity in phase space (Proposition~\ref{prop:continuity}), and backward uniqueness (Proposition~\ref{prop:backwards}), i.e., the injectivity of $S(t)$ on the absorbing set. These properties are used to show that the the set ${\mathcal{A}} = \cap_{t>0}S(t){\mathcal{B}}$ is the global attractor for the evolution in $H^1$ (Theorem~\ref{thm:attractor:existence}). For any $\theta_0\in H^1$ the solution tends to the global attractor. The convergence is uniform on bounded sets in $H^{1+\epsilon}$.  We also establish compactness of the linearization of the solution map and local uniform approximation results  (Proposition~\ref{prop:unif:diff}). The finite dimensionality of the attractor (Theorem~\ref{thm:finite:attractor}) is then obtained by applying classical tools~\cite{ConstantinFoias85,ConstantinFoias88}.

At this stage we note that to date the global regularity for the supercritical SQG equation with arbitrarily large initial data remains open. The type of known results are: small data global well-posedness (cf.~\cite{ConstantinCordobaWu01,ChaeLee03,CordobaCordoba04,Wu04,Miura06,Ju07,ChenMiaoZhang07,HmidiKeraani07,Yu08} and references therein), conditional regularity (cf.~\cite{ConstantinWu08,ConstantinWu09,DongPavlovic09,DongPavlovic09b}),  eventual regularity (cf.~\cite{Silvestre10a,Dabkowski11,Kiselev11}), or global regularity for dissipative operators which are only logarithmically supercritical~\cite{DabkowskiKiselevVicol12,XueZheng12,DabkowskiKiselevSilvestreVicol12}. The nonlinear lower bound for the fractional Laplacian (cf.~\cite{ConstantinVicol12}) may be employed to recover most of these results.

\section{Preliminaries} \label{sec:preliminaries}

We abuse notation and denote in the same way spaces of vector functions and scalar functions.
We do  not write the subindex ``per'', to emphasize that we work with $\TT^2$-periodic functions, i.e. $L^2_{\rm per}$ is simply written as $L^2$. We also overload notation to denote by $\phi \colon \RR^2 \to \RR$ the periodic extension to the whole space of a $\TT^2$ periodic function $\phi$.

\begin{remark}[\bf Scaling] \label{rem:box}
Choosing to work on the periodic box $\TT^2 = [-\pi,\pi]^2$ is just a matter of convenience for the presentation, so that the group of characters is $\ZZ^2$. All the results in this paper may be translated to the case of a general periodic box $\TT_L^2 = [-L/2,L/2]^2$ as follows. The critical SQG equation has a natural scaling invariance associated to it.
If $\theta(x,t)$ is a solution of  \eqref{eq:SQG:1}--\eqref{eq:SQG:3} on $[0,T] \times \TT^2$, with force $f(x)$ and initial data $\theta_0(x)$, then 
\[
\theta_\lambda(x,t) = \theta(x/ \lambda, t/ \lambda)
\]
is also a solution of the equations, but on the space-time domain $[0,\lambda T] \times [-\lambda \pi, \lambda \pi]^2$, with force $f_\lambda(x) = (1/ \lambda) f(x/\lambda)$ and initial condition $\theta_{0\lambda}(x) = \theta_0(x/\lambda)$. The value of $\kappa$ remains unchanged. Thus, in order to work on the box $\TT_L^2$, one merely has to set   $\lambda= L/(2\pi)$ in the below argument.
Note also that $\| \theta\|_{L^\infty} = \| \theta_\lambda \|_{L^\infty}$ for any $\lambda> 0$ and that the $L^\infty$ norm is non-increasing along solution paths, which is why it is customary to refer to \eqref{eq:SQG:1}--\eqref{eq:SQG:3} as the {\em critical} SQG equations.
\end{remark}

The fractional Laplacian $\Lambda^s$, with $s\in \RR$ may be defined in this context as the Fourier multiplier with symbol $|k|^s$, i.e. $\Lambda^s \phi(x) = \sum_{k\in\ZZ^2_*} |k|^s \hat\phi_k \exp(i k \cdot x)$, where $\phi(x) = \sum_{k \in \ZZ^2_*} \hat\phi_k \exp(i k \cdot x)$. Note that the eigenvalues of $\Lambda = (-\Delta)^{1/2}$ are given by $|k|$, with $k \in \ZZ^2_* = \ZZ \setminus \{0\}$. 
We label them in increasing order (counting multiplicity) as 
\[
0< \lambda_1 \leq \ldots \leq  \lambda_n \leq \ldots.
\] 
and denote the eigenfunction associated to $\lambda_j$ by $e_j$. Then $\{e_j\}_{j\geq1}$ is an orthonormal basis of $L^2$, and the sets $\{ \lambda_j\}_{j\geq 1}$ and $\{|k|\}_{k \in \ZZ^2_*}$ are equal. In view of the choice  $\TT^2 = [-\pi,\pi]^2$, we have that $\lambda_1=1$.

As a consequence of the mean-free setting, for $s\in \RR$ we may identify the homogenous Sobolev spaces $\dot{H}^s(\TT^2)$ and the inhomogenous Sobolev spaces $H^s(\TT^2)$, and we simply denote these by $H^s$ (without ``dots''). As usual these are the closure of (mean-free) $C^\infty(\TT^2)$ under the norm
\[
\|\phi\|_{H^s} = \|\Lambda^s \phi \|_{L^2} .
\]
Moreover, for $p\in [1,\infty]$ we denote by $H^{s,p} = H^{s,p}(\TT^2)$ the space of mean-free $L^p(\TT^2)$ functions $\phi$, which can be written as $\phi = \Lambda^{-s} \psi$, with $\psi \in L^p$. This is normed by
$
\|\phi\|_{H^{s,p}} = \|\Lambda^s \phi\|_{L^p}.
$
Lastly, H\"older spaces are denoted as usual by $C^\alpha$ for $\alpha \in (0,1)$, with seminorm given by 
\[
[\phi]_{C^\alpha} = \sup_{x\neq y \in \TT^2} \frac{|\phi(x)  - \phi(y)|}{|x-y|^{\alpha}}
\]
and norm $\| \phi\|_{C^\alpha} = \|\phi\|_{L^\infty} + [\phi]_{C^\alpha}$.

Recall cf.~\cite{CalderonZygmund54,Shapiro64,SteinWeiss71} that for $\phi \in C^\infty(\TT^2)$ the periodic Riesz transforms $\RSZ_j$ may be defined in terms of their Fourier multiplier symbol
$
\hat{\RSZ_j \phi}_k = i k_j |k|^{-1} \hat{\phi}_k,
$
for all $k \in \ZZ_*^2$. Alternatively this is defined as the singular integral
\begin{align}
\RSZ_j \phi(x) = P.V. \int_{\TT^2} \phi(x+y) R_j^*(y) dy \label{eq:Riesz}
\end{align}
where the periodic Riesz transform kernel $R_j^*$ is given by
\begin{align}
R_j^*(y) = R_j(y) + \sum_{k \in \ZZ^2_*} \left( R_j(x + 2 \pi k) - R_j(2\pi k) \right)
\label{eq:periodic:Rj}
\end{align}
for $y \neq 2\pi \ZZ^2$, and $R_j$ is the whole space Riesz-transform kernel given by
\begin{align}
R_j(y) = \frac{\Omega_j(y/|y|)}{|y|^2} = \frac{y_j}{2 \pi |y|^3} \label{eq:Riesz:kernel}
\end{align}
for $y\neq 0$. The explicit form of the $\Omega_j$ is not important: it is a {\em smooth} function  which has {\em zero mean on ${\mathbb S}^1$}. Note that if we extend $\phi$ periodically to $\RR^2$, we may rewrite \eqref{eq:Riesz} as 
\begin{align}
\RSZ_j \phi(x) = P.V. \int_{\RR^2} \phi(x+y) R_j(y) dy \label{eq:Riesz:extended}
\end{align}
where the principal value is both as $|y| \to 0$ and $|y| \to \infty$. See e.g.~\cite[pp.~256--261]{CalderonZygmund54}, or \cite[Chapter VII]{SteinWeiss71} for a proof.

\begin{remark}[\bf Constants] 
We use the following convention regarding constants: 
\begin{itemize}
\item $C$ shall denote a positive, sufficiently large constant, whose value may change from line to line; $C$ is allowed to depend on the size of the box and other universal constants which are fixed throughout the paper; to emphasize the dependence of a constant on a certain quantity $Q$ we write $C_Q$ or $C(Q)$; 
\item $c,c_0,c_1,...$ shall denote fixed constants appearing in the estimates, that have to be referred to specifically;
again, to emphasize dependence on a certain quantity $Q$, we write $c_Q$ or $c(Q)$;
\end{itemize}
\end{remark}

The functional analytic characterization the fractional Laplacian $\Lambda^\alpha$ as the Fourier multiplier with symbol $|k|^\alpha$ turns out to be useful for estimates in $L^2$-based Sobolev spaces, but not for pointwise in $x$ estimates. For this purpose, we recall the kernel representation of the periodic fractional Laplacian~\cite{CordobaCordoba04}, see also~\cite{DiNeazzaPalatucciValdinoci11,RoncalStinga12}. Note that other very useful characterizations are available, see e.g.~\cite{CaffarelliSilvestre07} to obtain monotonicity formulae. 

For $\alpha \in (0,2)$ and $\phi \in C^\infty(\TT^2)$ we have the pointwise definition
\begin{align*}
\Lambda^\alpha \phi(x) = P.V. \int_{{\mathbb{T}}^2} \left(  \phi (x)- \phi (x+y) \right) K_\alpha(y) dy
\end{align*}
where kernel $K_\alpha$ is defined on $\TT^2 \setminus\{0\}$ as
\begin{align}
K_\alpha(y) = c_{\alpha} \sum_{k \in \ZZ^2} \frac{1}{|y- 2\pi k|^{2+\alpha}} \label{eq:K:alpha}
\end{align}
and the normalization constant is 
\begin{align}
c_{\alpha} = \frac{2^\alpha \Gamma( 1 + \alpha/2)}{|\Gamma(-\alpha/2)| \pi }.\label{eq:c:alpha}
\end{align}
Under this normalization one has that $\lim_{\alpha \to 2-} \Lambda^\alpha \phi = - \Delta  \phi$, and $\lim_{\alpha \to 0+} \Lambda^\alpha \phi= \phi - \bar \phi$, pointwise in $x \in \TT^2$, where $\bar \phi$ is the mean of $\phi$ over $\TT^2$. When $\alpha \in (0,1)$ the above definition is valid for $\phi \in C^{\alpha+\eps}$, while for $\alpha \in [1,2)$ we need that $\phi \in C^{1,\alpha-1+\eps}$, for some $\eps>0$.

We recall the following two statements, which we  use frequently throughout the paper.

\begin{proposition}[\bf Pointwise identity]\label{prop:pointwise}
Let $\alpha \in (0,2)$ and $\phi \in C^\infty(\TT^2)$. Then we have that 
\begin{align} 
2 \phi(x) \Lambda^\alpha \phi(x) = \Lambda^\alpha (\phi(x)^2) + D_\alpha[\phi](x)
\label{eq:pointwise:identity}
\end{align}
holds, where 
\begin{align} 
D_\alpha[\phi](x) 
&= P.V. \int_{\TT^2} (\phi(x) - \phi(x+y))^2 K_\alpha(y) dy \notag\\
&= P.V. \int_{\RR^2} ( \phi(x) - \phi(x+y))^2 \frac{c_\alpha}{|y|^{2+\alpha}} dy 
\label{eq:D:alpha:def}
\end{align}
pointwise for $x\in \TT^2$, and we denote by $ \phi$ the periodic extension of $\phi$ to all of $\RR^2$.
\end{proposition}
Identity \eqref{eq:pointwise:identity} was proven in~\cite{CordobaCordoba04} (see also~\cite{Constantin06,ConstantinVicol12}). The same identity was used in \cite[Lemma 3.1]{Toland00} for the periodic case in one dimension, in the context of Stokes waves. In \cite{CordobaCordoba04} the pointwise estimate $f'(\phi) \Lambda^\alpha \phi - \Lambda^\alpha f(\phi) \geq 0$ was established for functions $f$ which are non-decreasing and convex. The second equality in \eqref{eq:D:alpha:def} follows from Fubini's theorem, a change of variables and the Dominated Convergence theorem.

\begin{proposition}[\bf Fractional Laplacian in $L^p$]\label{prop:Poincare}
Let $p=4q$, $q\ge 1$, $0 \leq \alpha \leq 2$, and let $\phi \in C^\infty$ have zero mean on ${\mathbb T}^d$. Then 
\begin{align}
\int_{\TT^d} \theta^{p-1}(x) \Lambda^\alpha \theta(x) dx \geq \frac{1}{p} \| \Lambda^{\alpha/2}  ( \theta^{p/2}) \|_{L^2}^2 +\frac{1}{C_{\alpha,d}} \|\theta\|_{L^p}^p 
\label{eq:Lp:Poincare}
\end{align}
holds, with an explicit constant $C_{\alpha,d} \geq 1$, which is independent of $p$.
\end{proposition}
When the second term on the right of \eqref{eq:Lp:Poincare} is absent, the above statement was proven in~\cite{CordobaCordoba04}. 
Since for $p=4q$, $q\ge 1$,  $\theta^{p/2}$ is not of zero mean, it is not immediately clear that the first term on the right of \eqref{eq:Lp:Poincare} dominates the $p$th power of the $L^p$ norm. (In contradistinction with the case  $p=2$, which is the classical Poincar\'{e} inequality). The proof of Proposition~\ref{prop:Poincare} to our knowledge was first given in~\cite[Appendix A]{ConstantinGlattHoltzVicol13}. For the sake of convenience we include a sketch of the proof in Appendix~\ref{app:calculus} below.

\section{Global regularity for the forced critical Burgers equation} \label{sec:Burgers}
In order to present the main idea of the proof of global regularity for the critical forced SQG equation, we first
consider the one dimensional critical forced Burgers equation
\begin{align}
\partial_t \theta + \theta \partial_x \theta + \Lambda \theta = f \label{eq:Burgers}
\end{align}
on a periodic domain $\TT=[-\pi,\pi]$, with smooth force, and smooth initial data $\theta_0$. We assume that the data and the force have zero mean on $\TT$, so that the same holds for the solution $\theta$ at later times.

First notice that we have a global in time control on $L^p$ norms of the solution. Since 
\begin{align*}
\int_{\TT} \theta \partial_x \theta \theta^{p-1} dx = \frac{1}{p+1} \int_{\TT} \partial_x (\theta^{p+1}) dx = 0
\end{align*}
we multiply \eqref{eq:Burgers} with $\theta^{p-1}$, integrate over $\TT$, and use the the lower bound in Proposition~\ref{prop:Poincare} to obtain 
\begin{align*}
\frac{d}{dt} \|\theta\|_{L^p} + c \|\theta\|_{L^p} \leq \|f\|_{L^p}
\end{align*}
for all $p\geq 2$ even, where the constant $c$ is independent of $p$. 
Integrating in time and then passing $p\to \infty$ thus yields
\begin{align}
\|\theta(t)\|_{L^\infty} \leq \|\theta_0\|_{L^\infty}  + \frac{1}{c} \|f\|_{L^\infty}= :B_\infty 
\label{eq:Burgers:Linfty}
\end{align}
for all $t\geq 0$. See also Proposition~\ref{prop:Lp} below for similar estimates for the SQG equation.

Our goal is to give a global in (positive) time bound for the $C^\alpha$ norm of the solution, for some $\alpha \in (0,1)$. It is well-known that due to the critical power of the dissipation, such a bound would in turn imply that the solution cannot develop singularities in finite time.

Let $\alpha \in (0,\alpha_0]$, where $\alpha_0 \in (0,1)$ is to be determined later in terms of the initial data and the force. In order to study the propagation of H\"older continuity in \eqref{eq:Burgers} we study the evolution of 
\begin{align}
v(x,t;h) = \frac{|\delta_h \theta(x)|}{|h|^\alpha} \notag
\end{align}
where $\delta_h \theta(x) = \theta(x+h) - \theta(x)$ is the usual finite difference. Note that a bound on $
\sup_{x,h \in \TT^2} |v(t,x;h)|
$
is in fact equivalent to a bound on $[\theta(t)]_{C^\alpha}$. Evaluating \eqref{eq:Burgers} at $x+h$ and $x$ and taking the difference one obtains
\begin{align}
\left( \partial_t + \theta \partial_x + (\delta_h \theta) \partial_h + \Lambda_x \right) (\delta_h \theta) = (\delta_h f).
\notag
\end{align}
Multiplying by $|h|^{-2\alpha} (\delta_h \theta)$ and appealing to Proposition~\ref{prop:pointwise} we thus arrive at the pointwise inequality
\begin{align}
\left( \partial_t + \theta \partial_x + (\delta_h \theta) \partial_h + \Lambda_x \right) v^2  + \frac{D[\delta_h \theta]}{|h|^{2\alpha}}
= \frac{4\alpha}{|h|^{2\alpha+1}} (\delta_h \theta)^3 + \frac{2 (\delta_h f) }{|h|^\alpha} v  
&\leq \frac{4\alpha}{|h|^{1-\alpha}} v^3 + \frac{4 \|f\|_{L^\infty}}{|h|^\alpha} v
\label{eq:Burgers:2}
\end{align}
where 
\begin{align}
D[\delta_h \theta] = c P.V. \int(\delta_h \theta(x) - \delta_h \theta(x+y) )^2 \frac{1}{|y|^2} dy.
\notag
\end{align} 
The main idea is to combine a nonlinear lower bound for $D[\delta_h\theta]$, with the smallness of $\alpha$ to obtain an ODE for $\| v(t)\|_{L^\infty_{t,x}}$, which has global bounded solutions. 

Using the argument in \cite{ConstantinVicol12}, see also \eqref{eq:D:h:lower} below, we have that
\begin{align}
D[\delta_h\theta](x) \geq \frac{|\delta_h \theta(x)|^3}{C \|\theta\|_{L^\infty} |h|}
\notag
\end{align}
for some $C>0$. Therefore, using \eqref{eq:Burgers:Linfty} we arrive at
\begin{align}
\frac{D[\delta_h \theta]}{|h|^{2\alpha}} \geq \frac{|\delta_h \theta(x)|^3}{C \|\theta\|_{L^\infty} |h|^{1+2\alpha}} = \frac{v^3}{C B_\infty |h|^{1-\alpha}}.
\notag
\end{align}	
Inserting the above bound in \eqref{eq:Burgers:2} yields
\begin{align}
\left( \partial_t + \theta \partial_x + (\delta_h \theta) \partial_h + \Lambda_x \right) v^2  + \frac{v^3}{C_0 B_\infty |h|^{1-\alpha}}
&\leq \frac{4\alpha v^3}{|h|^{1-\alpha}}  + \frac{4 \|f\|_{L^\infty}v}{|h|^\alpha}.
\label{eq:Burgers:3}
\end{align}
Therefore if we choose $\alpha_0$ such that  
\begin{align}
\alpha_0 \leq \frac{1}{8 C_0 B_\infty}
\notag
\end{align}
the nonlinear term on the right side of \eqref{eq:Burgers:3} can be absorbed into the left side of the inequality, and moreover, once we appeal to the $\eps$-Young inequality in order to hide the forcing term in the dissipation, we arrive at 
\begin{align}
\left( \partial_t + \theta \partial_x + (\delta_h \theta) \partial_h + \Lambda_x \right) v^2 + \frac{v^3}{4 C_0 B_\infty |h|^{1-\alpha}} \leq C_{1} B_\infty^{1/2} \|f\|_{L^\infty}^{3/2} |h|^{\frac{1-4\alpha}{2}}. \label{eq:Burgers:4}
\end{align}
for some $C_{1} >0$ which depends only on $C_0$.
Formally evaluating \eqref{eq:Burgers:4} at a point $(\bar x,\bar h)$ where $v^2$ attains its maximal value, and noting that such a point must necessarily obey $| \bar h|\leq \pi$ (the latter is due to the periodicity in $h$ of $\delta_h \theta$, and the strictly decaying nature of $|h|^{-\alpha}$), we obtain
\begin{align}
(\partial_t v^2)(\bar x,\bar h) + \frac{v^3(\bar x,\bar h)}{4 C_0 B_\infty \pi^{1-\alpha}}  
&\leq \left( \partial_t + \theta \partial_x + (\delta_h \theta) \partial_h + \Lambda_x \right) v^2(\bar x,\bar h) + \frac{v^3(\bar x,\bar h)}{4 C_0 B_\infty \pi^{1-\alpha}}  \notag\\
&\leq C_{1} B_\infty^{1/2} \|f\|_{L^\infty}^{3/2} \pi^{\frac{1-4\alpha}{2}},
\label{eq:Burgers:5}
\end{align}
as long as we impose 
\begin{align}
\alpha_0 \leq \frac 14.
\notag
\end{align}
In \eqref{eq:Burgers:5} we used that at the maximum (in joint $x$ and $h$) of $v^2$ we must have $\partial_x v^2 = \partial_h v^2=0$, and $\Lambda v^2 \geq 0$.
One can then rigorously show that on for almost every $t$ in $[0,T_*)$, the maximal time of existence of a smooth solution to the initial value problem associated to \eqref{eq:Burgers}, we have
\begin{align}
\frac{d}{dt} \|v(t)\|_{L^\infty_{x,h}}^2 \leq (\partial_t v^2)(\bar x,\bar h)
\notag
\end{align}
which combined with \eqref{eq:Burgers:5} yields that for any  $\alpha \in (0,\alpha_0]$, with $\alpha_0 = \min\{1/ (8C_0 B_\infty), 1/4\}$, we have
\begin{align}
[\theta(t)]_{C^\alpha} = \|v(t)\|_{L^\infty_{x,h}} \leq \max\left\{ [\theta_0]_{C^\alpha} , (4C_0C_1)^{1/2} \pi^{\frac{1-2\alpha}{2}} B_\infty^{1/2} \|f\|_{L^\infty}^{1/2} \right\}
\notag
\end{align}
for all $ t \in [0,T_*)$, and thus a posteriori for all $t \geq 0$, thereby completing the proof of global regularity for critical Burgers. 

Notice that the key ingredients were the nonlinear lower bound on $\Lambda$, and the smallness of $\alpha$. This argument carries over to the SQG case modulo some technical issues having to do with the fact that the velocity depends linearly but in a nonlocal fashion on $\theta$. We give details  in Section~\ref{sec:global} below, where we also fully justify the arguments presented here only formally for clarity of the exposition.

\section{Global regularity for forced critical SQG} \label{sec:global}

In this section we give a new proof of global existence and uniqueness of solutions for \eqref{eq:SQG:1}--\eqref{eq:SQG:3} which has the advantage that the {bounds} require merely $f \in L^\infty \cap H^1$. The proof is based on the nonlinear lower bound for the fractional Laplacian discovered in~\cite{ConstantinVicol12}. We first recall some $L^p$ estimates for solutions of the forced critical SQG equation.

\begin{proposition}[\bf Absorbing ball in $L^p$]\label{prop:Lp}
Let $\theta$ be a smooth solution of \eqref{eq:SQG:1}--\eqref{eq:SQG:3}, and let $p\geq 2$ be even. Then we have
\begin{align} 
\| \theta(\cdot,t)\|_{L^p} \leq \|\theta_0\|_{L^p} e^{-t c_0 \kappa} + \frac{1}{c_0 \kappa} \| f\|_{L^p} (1- e^{-t c_0 \kappa})
\label{eq:Lp:decay}
\end{align}
for some universal constant $c_0$. Moreover,
\begin{align} 
\| \theta(\cdot,t)\|_{L^\infty} \leq \|\theta_0\|_{L^\infty} e^{-t c_0 \kappa} + \frac{1}{c_0 \kappa} \| f\|_{L^\infty} (1- e^{-t c_0 \kappa})
\label{eq:Linfty:decay}
\end{align}
holds with the same universal constant $c_0$. 
\end{proposition}
\begin{proof}
Multiplying \eqref{eq:SQG:1} by $\theta^{p-1}$, and using Proposition~\ref{prop:Poincare} we arrive at
\begin{align*}
\frac{d}{dt} \|\theta\|_{L^p} + c _0  \kappa \|\theta\|_{L^p} \leq \|f\|_{L^p} 
\end{align*}
for some $c_0>0$ which is independent of $p$. Therefore we  obtain \eqref{eq:Lp:decay} for all $t\geq 0$ and $p\geq 2$ even. Moreover, since the constants appearing in \eqref{eq:Lp:decay} are independent of $p$, and we are on a periodic domain $\TT^2 = [-\pi,\pi]^2$, we have that 
\begin{align*}
\|\theta(\cdot,t)\|_{L^p} \leq (2 \pi)^{2/p} \|\theta_0\|_{L^\infty} e^{-t c_0 \kappa} + \frac{ (2 \pi )^{2/p}}{c_0  \kappa} \| f\|_{L^\infty} (1- e^{-t c_0 \kappa})
\end{align*}
which yields \eqref{eq:Linfty:decay} upon passing $p\to \infty$.
\end{proof}
In particular, Proposition~\ref{prop:Lp} shows that for any $p \in [2,\infty)$ even and $p = \infty$ we have that
\begin{align} 
\|\theta(\cdot,t)\|_{L^p} \leq  \|\theta_0\|_{L^p}  + \frac{1}{c_0 \kappa} \| f\|_{L^p} := M_p(\theta_0,f) = M_p\label{eq:Mp:def}
\end{align}
for any $t\geq 0$. Next, we recall the following local existence~\cite{Miura06,Ju07} and smoothing~\cite{Dong10}  result.
\begin{proposition}[\bf Local solution]\label{prop:local}
Assume $\theta_0 \in H^1$ and $f\in H^1 \cap L^\infty$. There exists $T_* = T_*(\theta_0,f)>0$ and a unique solution $\theta$ of the initial value problem \eqref{eq:SQG:1}--\eqref{eq:SQG:3} which obeys the energy inequality and 
\begin{align}
\theta \in C([0,T_*);H^1) \cap L^2(0,T_*;H^{3/2}).\label{eq:local:sol}
\end{align}
Moreover, for any $\beta\geq 0$, if $f \in H^\beta$ we have that 
\begin{align}
\sup_{0<t<T_*} t^{\beta} \|\theta(\cdot,t)\|_{H^{1+\beta}} \leq c_1 \kappa^{-\beta} \|\theta_0\|_{H^1} + c_1 \kappa^{-1} T_*^{\beta} \|f\|_{H^{\beta}} \label{eq:local:smoothing}
\end{align}
for some positive universal constant $c_1$. Additionally,  we have 
$
\lim_{t\to 0+} t^{\beta} \|\theta(\cdot,t)\|_{H^{1+\beta}} = 0,
$
when $\beta>0$.
\end{proposition}
We emphasize that the time of existence $T_*$ doesn't depend in a locally uniform way on $\|\theta_0\|_{H^1}$.
The above result was proven in the aforementioned works in the absence of a forcing term, but it is not difficult to verify that estimate \eqref{eq:local:smoothing} holds assuming $f$ is sufficiently smooth.  We omit these details. Note that the solution may be extended uniquely past $T_*$ assuming that a priori we know e.g. that $ \sup_{t \in [T_*/2,T_*)} [\theta(t,\cdot)]_{C^\alpha} < \infty$, for some $\alpha >0$.

The first result we obtain is the propagation of H\"older continuity for smooth solutions, which may then be applied to a sequence of solutions to a regularized problem, in order to obtain in the limit the corresponding result for weak solutions (see Theorem~\ref{thm:Holder:weak} below).

\begin{theorem}[\bf Propagation of H\"older regularity]\label{thm:Holder:propagation}
Let $\theta_0$ and $f$ be sufficiently smooth, $T>0$ be arbitrary, and let $\theta \in C^{1/2}((0,T);C^{1,1/2})$ be the unique classical solution of the initial value problem \eqref{eq:SQG:1}--\eqref{eq:SQG:3}. As in \eqref{eq:Mp:def} above, define \[ M_\infty = M_\infty (\theta_0,f) = \|\theta_0\|_{L^\infty} + (c_0 \kappa)^{-1} \|f\|_{L^\infty}.\]
There exists a sufficiently small {universal} constant  $\eps_0 >0$ such that for any $\alpha $ with
\begin{align}
0 < \alpha \leq \alpha_0 = \min \left\{ \frac{\eps_0 \kappa}{M_\infty} , \frac 14 \right\} \label{eq:alpha:0:def},
\end{align}
if $\theta_0 \in C^\alpha$ we have
\begin{align}
[\theta(t)]_{C^\alpha} \leq M_\alpha(t) \label{eq:global:Holder:bound}
\end{align}
for all $t \in [0,T]$, where $M_\alpha$ is the solution of the ordinary differential equation \eqref{eq:M:alpha:def} below. In particular, we have that 
\begin{align}
M_\alpha(t) \leq \max \left\{ [\theta_0]_{C^\alpha} , \frac{M_\infty}{\eps_0} \right\} \label{eq:M:alpha:global}
\end{align}
for any $t \in [0,T]$, and also 
\begin{align}
M_\alpha(t) \leq \frac{2 M_\infty}{\eps_0} \label{eq:M:alpha:long:time}
\end{align}
for all $t \geq t_\alpha = t_\alpha(M_\infty, [\theta_0]_{C^\alpha})$,  which is defined explicitly in \eqref{eq:t:alpha:def} below.
\end{theorem}
\begin{proof}[Proof of Theorem~\ref{thm:Holder:propagation}]
For $\alpha>0$, we look at the evolution of weighted finite differences
\begin{align}
v(x,t;h) =  \frac{| \delta_h \theta(x)|}{|h|^{\alpha}}. \notag
\end{align}
Since $\theta_0 \in C^\alpha$, we have that 
\begin{align}
\| v_0 \|_{L^\infty_{x,h}} \leq [\theta_0]_{C^\alpha} \leq M_\alpha(0). \label{eq:M:alpha:1}
\end{align} 
Our goal is to find an upper bound $M_\alpha(t)$ such that 
\begin{align}
\|v(\cdot,t;\cdot)\|_{L^\infty_{x,h}} \leq M_\alpha(t) \label{eq:Holder:bound}
\end{align}
for any $t\geq 0$. In view of the periodicity in $h$ of $\delta_h \theta(x)$, we would in turn obtain from \eqref{eq:Holder:bound} that 
\begin{align}
[v(\cdot,t)]_{C^\alpha} = \sup_{x,h \in \TT^2} \frac{|\delta_h \theta(x,t)|}{|h|^\alpha} \leq \| v(\cdot,t;\cdot) \|_{L^\infty_{x,h}} \leq M_\alpha (t)
\label{eq:Holder:OSS}
\end{align}
which would then conclude the proof of the theorem.

In order to find $M_\alpha(t)$ for which \eqref{eq:Holder:bound} holds, we first write the equation obeyed by $\delta_h \theta$. If follows by taking finite differences in \eqref{eq:SQG:1} that pointwise in $x,t,h$ we have
\begin{align}
\left( \partial_t + u \cdot \nabla_x + (\delta_h u) \cdot \nabla_h + \kappa \Lambda \right) \delta_h \theta = \delta_h f.
\label{eq:Holder:0}
\end{align}
Upon multiplying \eqref{eq:Holder:0} by
\begin{align*}
 \frac{\delta_h \theta (x,t)}{|h|^{2\alpha}}
\end{align*}
and using Proposition~\ref{prop:pointwise}, we obtain that $v$ obeys the equation
\begin{align} 
\left(\partial_t + u \cdot \nabla_x + (\delta_h u)\cdot \nabla_h + \kappa \Lambda \right) v^2 +   \frac{\kappa D[\delta_h \theta]}{|h|^{2\alpha}}   
&= 4 \alpha (\delta_h u) \cdot \frac{h}{|h|}  \frac{v^2}{|h|}   +   \frac{2 (\delta_h f) v}{|h|^{\alpha}}   \notag\\ 
& \leq 4 \alpha |\delta_h u|   \frac{v^2}{|h|} +  \frac{4 \| f\|_{L^\infty} v}{|h|^{\alpha}}    \label{eq:Holder:1}
\end{align}
where $\delta_h u = \RSZ^\perp_x (\delta_h \theta)$, and 
\begin{align} 
D[\delta_h \theta](x) 
&= P.V. \int_{\TT^2}  \Bigl(\delta_h \theta(x) - \delta_h \theta(x+y) \Bigr)^2 K_1(y) dy \notag\\
&= \frac{1}{2\pi} P.V. \int_{\RR^2} \Bigl(\delta_h  \theta(x) - \delta_h  \theta(x+y) \Bigr)^2 \frac{1}{|y|^3} dy \label{eq:D:h:def}
\end{align}
with the kernel $K_1$ given explicitly by \eqref{eq:K:alpha} with $\alpha=1$, and we have used the explicit formula \eqref{eq:c:alpha} for the normalizing constant. First we give a lower bound on the dissipative term $D[\delta_h\theta]$, and then estimate the velocity increment $|\delta_h u|$ in \eqref{eq:Holder:1}.

Throughout the proof we will use a cutoff function $\chi \colon [0,\infty) \to [0,\infty)$ which is smooth, non-increasing, identically $1$ on $[0,1]$, vanishes on $[2,\infty)$, and obeys $|\chi'| \leq 2$.

In the spirit of~\cite{ConstantinVicol12}, we obtain a nonlinear lower bound for the dissipative term in \eqref{eq:Holder:1}. Pointwise in $x$ and $h$ it holds that
\begin{align} 
D[\delta_h \theta](x) \geq \frac{|\delta_h\theta(x)|^3}{c_2 \|\theta\|_{L^\infty} |h|} \label{eq:D:h:lower}
\end{align}
for some universal constant $c_2>0$.
To prove \eqref{eq:D:h:lower}, we proceed as follows. For $r \geq 4|h| $ to be determined, we have
\begin{align}
D[\delta_h\theta](x) 
&\geq \frac{1}{2\pi} \int_{\RR^2} \frac{(\delta_h  \theta(x) - \delta_h  \theta(x+y) )^2}{|y|^{3}} \left(1 - \chi\left(\frac{|y|}{r} \right) \right) dy \notag\\
&\geq \frac{1}{2\pi} (\delta_h \theta(x))^2 \int_{|y|\geq 2r} \frac{1}{|y|^{3}} dy - \frac{1}{\pi} |\delta_h \theta(x)| \left| \int_{\RR^2} \frac{\delta_h  \theta(x+y)}{|y|^3} \left(1 - \chi\left(\frac{|y|}{r} \right) \right)  dy  \right|\notag\\
&\geq \frac{1}{2r} |\delta_h \theta(x)|^2  - \frac{1}{\pi} |\delta_h \theta(x)| \int_{\RR^2} |\theta(x+y)| \left| \delta_{-h} \left( \frac{1}{|y|^3} - \frac{1}{|y|^3} \chi \left( \frac{|y|}{r} \right) \right) \right| dy \notag\\
&\geq \frac{1}{2r} |\delta_h \theta(x)|^2  - C  |\delta_h \theta(x)| \| \theta\|_{L^\infty}  |h| \int_{\RR^2} \left( \frac{{\bf 1}_{|y|\geq 3r/4}}{|y|^4} +  \frac{{\bf 1}_{3r/4 \leq |y|\leq 9r/4}}{r^4} \right) dy \notag\\
&\geq \frac{1}{2r} |\delta_h \theta(x)|^2  - C |\delta_h \theta(x)| \| \theta\|_{L^\infty}  \frac{|h|}{r^2}  
\label{eq:D:h:lower:1}
\end{align}
for some $C\geq 1$.
In the above estimate we have used estimate \eqref{eq:grad:gr} below. More precisely, the mean value theorem gives that for a smooth $g$ we have
\begin{align} 
|\delta_{h} g(y)| = \left| h \cdot \nabla g( (1-\lambda) y + \lambda (y+h)) \right| \leq |h| \max_{\lambda\in[0,1]} |\nabla g ( y + \lambda h)| \label{eq:MVT}
\end{align}
for $h,y \in \RR^2$. In particular, for 
\[
g(y) = \frac{1}{|y|^{3}} \left(1 - \chi \left(\frac{|y|}{r} \right) \right)
\] 
we have used that
\[
|\nabla g(z) | \leq C \left( \frac{{\bf 1}_{|z|\geq r}}{|z|^4} + \frac{{\bf 1}_{r \leq |z|\leq 2r}}{r |z|^3} \right) \leq C \left( \frac{{\bf 1}_{|z|\geq r}}{|z|^4} +  \frac{{\bf 1}_{r \leq |z|\leq 2r}}{r^4} \right)
\]
which read at $z = y + \lambda h$, yields
\begin{align}
\max_{\lambda \in [0,1]} |\nabla g( y + \lambda h) | \leq C \left( \frac{{\bf 1}_{|y|\geq 3r/4}}{|y|^4} +  \frac{{\bf 1}_{3r/4 \leq |y|\leq 9r/4}}{r^4} \right)\label{eq:grad:gr}
\end{align}
whenever $|h|\leq r/4$. This proves \eqref{eq:D:h:lower:1}. Setting
\[
r = \frac{4 C  \|\theta\|_{L^\infty}}{|\delta_h \theta(x)|} |h|
\]
in \eqref{eq:D:h:lower:1}, which obeys $r \geq 4 |h|$ due to $|\delta_h \theta| \leq 2 \|\theta\|_{L^\infty}$, completes the proof of \eqref{eq:D:h:lower}. 

Combining estimate \eqref{eq:D:h:lower} with the a priori bound \eqref{eq:Mp:def} (with $p=\infty$) we obtain that the positive term on the right side of \eqref{eq:Holder:1} is bounded from below as
\begin{align}
 \frac{\kappa  D[\delta_h \theta]}{|h|^{2\alpha}} \geq   \frac{\kappa}{|h|^{2\alpha}} \frac{|\delta_h \theta|^3}{c_2 M_\infty |h|} 
 = \frac{\kappa }{c_2 M_\infty  } \frac{v^3}{ |h|^{1-\alpha}} \label{eq:Holder:D:h}
\end{align}
pointwise in $x$ and $h$.

We now estimate the velocity finite difference $\delta_h u$. For this purpose fix $x$ and $h$ (ignore $t$ dependence) and let $\rho>0$ be such that 
\begin{align}
\rho \geq 4 |h|. \label{eq:rho:cond}
\end{align}
As before we let $\chi$ be a smooth cutoff function, that is $1$ on $[0,1]$, non-increasing, vanishes on $[2,\infty)$, and obeys $|\chi'|\leq 2$.
We decompose the singular integral defining $\delta_h u$ as 
\begin{align}
\delta_h u(x) = \RSZ^\perp_x (\delta_h \theta(x) ) = \frac{1}{2\pi} P.V. \int_{\RR^2} \frac{y^\perp}{|y|^3} \delta_h \theta(x+y) dy = \delta_h u_{in}(x)  + \delta_h u_{out}(x),\notag
\end{align}
where the inner piece is given by
\begin{align}
\delta_h u_{in}(x)  
&=\frac{1}{2\pi} P.V. \int_{\RR^2} \frac{y^\perp}{|y|^3} \chi\left(\frac{|y|}{\rho}\right) \delta_h \theta(x+y) dy\notag\\
&=\frac{1}{2\pi} P.V. \int_{\RR^2} \frac{y^\perp}{|y|^3} \chi\left(\frac{|y|}{\rho}\right) \Bigl( \delta_h \theta(x+y) - \delta_h \theta(x) \Bigr)dy\notag
\end{align}
by using that the kernel of $\RSZ^\perp$ has zero average on the unit sphere, and the outer piece is given by
\begin{align}
\delta_h u_{out}(x)
&=\frac{1}{2\pi} \int_{\RR^2} \frac{y^\perp}{|y|^3} \left( 1- \chi\left(\frac{|y|}{\rho}\right) \right)  \delta_h \theta(x+y) dy\notag\\
&=\frac{1}{2\pi} \int_{\RR^2} \delta_{-h} \left(  \frac{y^\perp}{|y|^3} \left( 1- \chi\left(\frac{|y|}{\rho}\right) \right)  \right)  \theta(x+y)   dy\notag
\end{align}
by using a finite-difference-by-parts.

For the inner piece, by appealing to the Cauchy-Schwartz inequality,  we obtain
\begin{align}
|\delta_h u_{in}| \leq C \left( \rho D[\delta_h \theta] \right)^{1/2}. \label{eq:Holder:u:in}
\end{align}
In order to bound the outer piece, we recall that the mean value theorem gives
\begin{align} 
|\delta_{h} g(y)|  \leq |h| \max_{\lambda\in[0,1]} |\nabla g (y + \lambda h)| \label{eq:MVT:2}
\end{align}
for smooth functions $g$. We apply \eqref{eq:MVT:2} with
\begin{align} 
g(y) = \frac{y^\perp}{|y|^3}  \left( 1- \chi\left(\frac{|y|}{\rho}\right) \right)  \notag
\end{align}
with derivative
\begin{align} 
|\nabla g(z)| 
\leq C \left( \frac{{\bf 1}_{|z|\geq \rho}}{|z|^3} +  \frac{{\bf 1}_{\rho\leq |z|\leq 2\rho} }{\rho |z|^2}  \right)
\leq C \frac{{\bf 1}_{|z|\geq \rho}}{|z|^3}. \notag
\end{align}
Evaluating the above inequality at $z = y - \lambda h$ and using that $|h| \leq \rho/4$ we obtain
\begin{align} 
|\delta_{-h} g(y)| \leq |h| \max_{\lambda \in [0,1]} |\nabla g (y - \lambda h)| \leq C |h| \frac{{\bf 1}_{|y| \geq \rho/2}}{|y|^3}. \notag
\end{align}
This in turn implies that 
\begin{align}
|\delta_h u_{out}| \leq C |h|  \int_{|y|\geq \rho/2} \frac{1}{|y|^3} |\theta(x+y)| dy \leq \frac{C M_\infty |h|}{\rho} \label{eq:Holder:u:out}.
\end{align}
Combining the  inner \eqref{eq:Holder:u:in} and outer \eqref{eq:Holder:u:out} velocity estimates, we obtain that 
\begin{align}
|\delta_h u| \leq C \left( \left( \rho D[\delta_h \theta] \right)^{1/2} + \frac{M_\infty |h|}{\rho} \right) \label{eq:delta:h:u}
\end{align}
if $\rho$ is chosen so that $\rho \geq 4|h|$. 

Using the Cauchy-Schwartz inequality we  obtain from \eqref{eq:delta:h:u} that
\begin{align*}
4 \alpha  |\delta_h u|  \frac{v^2}{|h|} 
&\leq C \alpha  \left( \rho D[\delta_h \theta] \right)^{1/2} \frac{v^2}{|h|} + C \alpha  \frac{M_\infty v^2}{\rho}  \notag   \\
&\leq  \frac{ \kappa  D[\delta_h \theta] }{2 |h|^{2\alpha}} +  \frac{c_3 \alpha^2 \rho v^4}{\kappa |h|^{2-2\alpha}} +    \frac{ c_3 \alpha M_\infty v^2}{\rho}  
\end{align*}
for some sufficiently large $c_3>0$. Therefore, letting
\begin{align}
\rho = \frac{\kappa^{1/2} M_\infty^{1/2} |h|^{1-\alpha} }{ \alpha^{1/2} v} = \frac{ \kappa^{1/2} M_\infty^{1/2}}{ \alpha^{1/2} |\delta_h \theta|} |h| \label{eq:rho:1}
\end{align}
we obtain
\begin{align}
4 \alpha  |\delta_h u|  \frac{v^2}{|h|} 
\leq  \frac{\kappa D[\delta_h \theta] }{2 |h|^{2\alpha}} + \frac{ 2 c_3 \alpha^{3/2} M_\infty^{1/2}}{\kappa^{1/2} } \frac{v^3}{ |h|^{1-\alpha}}.
\label{eq:delta:h:u:2}
\end{align}
Note that in order to define $\rho$ as in \eqref{eq:rho:1}, we need to ensure $\rho \geq 4|h|$. By the triangle inequality we have $|\delta_h \theta| \leq 2 M_\infty$ and hence indeed 
\begin{align*}
\rho \geq \left( \frac{\kappa}{2 \alpha M_\infty} \right)^{1/2} |h|  \geq 4 |h|
\end{align*}
holds, since by assumption $\alpha \in (0,\alpha_0]$ and 
\begin{align}
\alpha_0 \leq \frac{\kappa}{32  M_\infty}. \label{eq:alpha:0:cond:1}
\end{align}

Combining \eqref{eq:Holder:1}, the lower bound \eqref{eq:Holder:D:h},  estimate \eqref{eq:delta:h:u:2}, and recalling the definition of $M_\infty$ in \eqref{eq:Mp:def}, we arrive at
\begin{align}
\left(\partial_t + u \cdot \nabla_x + (\delta_h u)\cdot \nabla_h + \kappa \Lambda \right) v^2 
+ \left( \frac{\kappa }{2 c_2 M_\infty} - \alpha^{3/2} \frac{ 2 c_3 M_\infty^{1/2}}{\kappa^{1/2} } \right) \frac{v^3}{ |h|^{1-\alpha}}   \leq     \frac{4 c_0 \kappa M_\infty v}{|h|^{\alpha}}.
\label{eq:Holder:2}
\end{align}
Since $\alpha \leq \alpha_0$, with
\begin{align} 
\alpha_0 \leq \frac{\kappa}{(8 c_2 c_3)^{2/3} M_\infty} \label{eq:alpha:0:cond:2}
\end{align}
we furthermore obtain from \eqref{eq:Holder:2}
that 
\begin{align} 
\left(\partial_t + u \cdot \nabla_x + (\delta_h u)\cdot \nabla_h + \kappa \Lambda \right) v^2 
+   \frac{\kappa }{4 c_2 M_\infty}   \frac{v^3}{ |h|^{1-\alpha}}   \leq   \frac{4 c_0 \kappa M_\infty v}{|h|^{\alpha}}.
\label{eq:Holder:3}
\end{align}
pointwise in $x$ and $h$.
Upon using the $\eps$-Young inequality for the right hand side, that 
\begin{align} 
\left(\partial_t + u \cdot \nabla_x + (\delta_h u)\cdot \nabla_h + \kappa \Lambda \right) v^2 
+   \frac{\kappa }{6 c_2 M_\infty}   \frac{v^3}{ |h|^{1-\alpha}}   \leq   c_4 \kappa |h|^{\frac{1-4\alpha}{2}} M_\infty^2
\label{eq:Holder:4}
\end{align}
for a positive universal constant $c_4$ which may be computed explicitly ($c_4 = (4c_0)^{3/2} (4c_2)^{1/2}$).

We now proceed as in \cite[Section 4]{CordobaCordoba04}, and refer to Appendix~\ref{sec:sup:dt} below for details. Since $\theta$ is sufficiently smooth 
$v^2$ is a bounded continuous function in both $x$ and $h$, which is periodic in $x$. Moreover, given $x,h \in \TT^2$ and $k \in \ZZ^2_*$ we have that
\[
v(x,t;h)^2 \geq v^2(x,t;k+2\pi k),
\] 
in view of the periodicity in $h$ of $\delta_h \theta$, and the strict monotonicity of $|h|^{-\alpha}$. Therefore there exists {\em at least} one point $(\bar x, \bar h) = (\bar x(t), \bar h(t)) \in \TT^2 \times \TT^2$ where the function $v(\cdot,t;\cdot)^2$ attains its maximum. We define
\begin{align} 
g(t) = \sup_{x,h \in \TT^2} v(x,t;h)^2 = v(\bar x(t), t; \bar h(t))^2. \label{eq:g(t):def}
\end{align}
In Appendix~\ref{sec:sup:dt} below we show that $g$ is Lipschitz continuous on $[0,T]$, and that for almost every $t$ there exists $(\bar x(t);\bar h(t))$ such that 
\begin{align*} 
g'(t) = ( \partial_t v^2) (\bar x(t),t;\bar h(t))
\end{align*}
and \eqref{eq:g(t):def} holds.
Evaluating \eqref{eq:Holder:4} at the joint $x,h$-maximum $(\bar x,\bar h)$, using that at the maximum  we have
\[
\nabla_x v(\bar x,t;\bar h)^2 = 0 = \nabla_h v(\bar x,t;\bar h)^2 \quad \mbox{ and } \quad \Lambda v(\bar x,t;\bar h)^2 \geq 0,
\]
and the fact that $|\bar h|\leq 4 \pi$ we thus obtain 
\begin{align} 
g'(t) +\frac{\kappa }{6 c_2 M_\infty (4\pi)^{1-\alpha}}  g(t)^{3/2} \leq c_{4} \kappa M_\infty^{2} (4 \pi) ^{\frac{1-4\alpha}{2}}
\label{eq:Holder:5}
\end{align}
once we additionally assume that 
\begin{align}
\alpha_0 \leq \frac 14. \label{eq:alpha:0:cond:3}
\end{align}

Since we are now dealing with an ordinary differential equation, by the usual comparison principle for ODEs it follows that 
\begin{align*}
\| v(t)\|_{L^\infty_{t,x}}^2 = g(t) \leq M_\alpha(t)^2,
\end{align*}
where $M_\alpha(t)$ is the solution of the initial value problem
\begin{align}
\frac{d}{dt} M_\alpha^2 + \frac{\kappa}{c_5 M_\infty} M_\alpha^3 = c_5^2 \kappa M_\infty^2, \quad M_\alpha(0) = [\theta_0]_{C^\alpha},
\label{eq:M:alpha:def}
\end{align}
where $c_5 = c_5(c_2,c_4)$, is a fixed deterministic constant which is {\em independent} of $\kappa, M_\infty$, or $\alpha$.

In particular, there have proven that 
\begin{align*} 
[\theta(t)]_{C^\alpha} \leq M_\alpha(t) \leq \max\left\{ [\theta_0]_{C^\alpha}, c_5  M_\infty \right\}
\end{align*}
for any $t\geq 0$.
Moreover there exists 
$t_\alpha = t_\alpha(M_\infty,[\theta_0]_{C^\alpha}) \geq 0$ defined as
\begin{align} 
t_\alpha = 
\begin{cases} 
0, & \mbox{if } [\theta_0]_{C^\alpha} \leq 2 c_5 M_\infty \\ 
\frac{1}{7\kappa} \left( \frac{[\theta_0]_{C^\alpha}^2}{ 4 c_5^2 M_\infty^2} -1 \right), & \mbox{if } [\theta_0]_{C^\alpha} > 2 c_5 M_\infty 
\end{cases}
\label{eq:t:alpha:def}
\end{align}
such that
\begin{align*} 
[\theta(t)]_{C^\alpha} \leq 2 c_5  M_\infty
\end{align*}
for any $t\geq t_\alpha$. The above bound shows that the solution forgets the initial data even in the $C^\alpha$ norm.
\end{proof}

Theorem~\ref{thm:Holder:propagation} implies the propagation of  H\"older continuity for weak solutions.

\begin{theorem}[\bf H\"older propagation for weak solutions] \label{thm:Holder:weak}
Assume $f \in L^\infty \cap H^1$, $\theta_0 \in L^\infty \cap H^1$, $T>0$ is arbitrary, and let $\theta \in L^\infty(0,T;H^1) \cap L^2(0,T;H^{3/2})$ be the unique weak solution of the critical, forced SQG equation \eqref{eq:SQG:1}--\eqref{eq:SQG:3}. Let $\alpha_0 = \alpha_0 (\|\theta_0\|_{L^\infty},\|f\|_{L^\infty}) \leq 1/4$ be defined as in \eqref{eq:alpha:0:def}. For any $\alpha \in (0,\alpha_0]$, if $\theta_0 \in C^\alpha$, then $\theta \in L^\infty(0,T;C^\alpha)$. Moreover,
we have $[ \theta(t)]_{C^\alpha} \leq M_\alpha(t)$ for a function $M_\alpha(t)$ which obeys \eqref{eq:M:alpha:global} and \eqref{eq:M:alpha:long:time}.
\end{theorem}

\begin{proof}[Proof of Theorem~\ref{thm:Holder:weak}]
Let $J_\eps$ be a standard mollifier operator. For $\eps \in (0,1]$, we let $\theta_\eps$ be the solution of 
\begin{align}
\partial_t \theta^\eps + \kappa \Lambda \theta^\eps + u^\eps \cdot \nabla \theta^\eps - \eps \Delta \theta^\eps = J_\eps f, \quad u^\eps = \RSZ^\perp \theta^\eps, \quad \theta^\eps_0 = \theta_0.
\label{eq:regular:SQG}
\end{align}
As in Proposition~\ref{prop:Lp} we obtain that
\begin{align}
\|\theta^\eps(t) \|_{L^\infty} \leq \|\theta_0\|_{L^\infty} + \frac{1}{c_0 \kappa} \|J_\eps f\|_{L^\infty} \leq \|\theta_0\|_{L^\infty} + \frac{1}{c_0 \kappa} \|f\|_{L^\infty} = M_\infty(\theta_0,f).
\label{eq:regular:SQG:bdd}
\end{align}
Indeed, the addition of the regularizing term $-\eps \Delta \theta^\eps$ in the equation does not change any part of the argument, and $J_\eps$ is given by convolving with an $L^1$ kernel of mean $1$. Once we have that  $\theta^\eps \in L^\infty_{t,x}$, a supercritical information for the dissipation given by the Laplacian, the existence of a unique global smooth solution of \eqref{eq:regular:SQG} follows from classical arguments (see e.g.~\cite{ConstantinWu99} for subcritical SQG). Since $J_\eps f \in C^\infty$, in fact a  bootstrap shows that $\theta^\eps \in C^\infty( (0,T) \times \TT^2)$, but with bounds that depend on $\eps$.

At this stage we compute $\alpha_0$ as in \eqref{eq:alpha:0:def}, which is independent of $\eps$ due to \eqref{eq:regular:SQG:bdd}, and then apply Theorem~\ref{thm:Holder:propagation} to $\theta^\eps$, which we are allowed to since $\theta^\eps$ is smooth. We emphasize that the presence of the regularizing term $-\eps \Delta$ does not require {\em any} modification to the proof of Theorem~\ref{thm:Holder:propagation}. The Laplacian (in $x$) does not affect the finite differences in $h$, and the negative Laplacian evaluated at the maximum of a function is non-negative.  Therefore, for any $\eps \in (0,1]$, we have
\begin{align}
[\theta^\eps(t)]_{C^\alpha} \leq M_\alpha(t)
\label{eq:theta:eps:C:alpha}
\end{align}
for all $t\geq 0$, where $M_\alpha (t)$ (which is independent of $\eps$) is given by the solution of \eqref{eq:M:alpha:def}, and obeys the bounds \eqref{eq:M:alpha:global} for all time, and \eqref{eq:M:alpha:long:time} for long enough time. 

The sequence of solutions $\{ \theta^\eps \}_{\eps \in (0,1]}$ is thus uniformly bounded in $L^\infty(0,T;C^\alpha)$. In particular,  since $|\TT^2|<\infty$ this implies that $\theta^\eps$ is uniformly bounded in $L^\infty(0,T;H^\alpha)$, and from \eqref{eq:regular:SQG} we have that $\partial_t \theta^\eps$ is uniformly bounded in $L^\infty(0,T;H^{\alpha-2})$. Since the injection of $H^\alpha(\TT^2)$ into $L^2(\TT^2)$ is compact, and the injection of $L^2(\TT^2)$ in $H^{\alpha-2}(\TT^2)$ is continuous, the Aubin-Lions compactness lemma and the uniform in $\eps$ estimates obtained earlier, imply that there exists $\bar \theta \in L^\infty(0,T;C^\alpha)$, with bounds inherited (e.g. by duality) directly from \eqref{eq:theta:eps:C:alpha}, such that 
\[
\theta^\eps \to \bar \theta \mbox{ in } L^2(0,T;L^2).
\]
The above strong convergence in $L^2_{t,x}$ is enough in order to pass to the limit in the weak formulation of \eqref{eq:regular:SQG}, and show that $\bar \theta$ is a weak solution of the critical forced SQG equation \eqref{eq:SQG:1}--\eqref{eq:SQG:3} on $[0,T)$. This is seen by writing the nonlinear term in divergence form.

To conclude the proof we notice that in fact $\theta = \bar \theta$. This follows in the spirit of weak-strong uniqueness: one writes the equation obeyed by $\theta - \bar \theta$ and performs an $L^2$ energy estimate. The equation for the difference has zero initial data and zero force. Using that $\bar \theta \in L^\infty_t C^\alpha_x$ we have $\int  \RSZ^\perp \bar \theta \cdot \nabla (\theta -\bar \theta) (\theta - \bar \theta) dx = 0$, and since $\theta \in L^\infty_t H^1_x \cap L^2_t H^{3/2}_x$ we have $|\int \RSZ^\perp(\theta-\bar \theta) \cdot \nabla \theta (\theta - \bar \theta) dx| \leq \kappa \|\theta-\bar \theta\|_{H^{1/2}}^2 + C \|\theta -\bar \theta\|_{L^2}^2 \|\theta\|_{H^{3/2}}^2$. The proof of $\theta = \bar \theta$ is then concluded via the Gr\"onwall inequality.
\end{proof}

The results obtained in this section may be summarized as follows.

\begin{theorem}[\bf Global regularity]\label{thm:global:regularity}
Let $\theta_0 \in H^1(\TT^2)$ and $f \in L^\infty(\TT^2) \cap H^1(\TT^2)$. There exists a unique global solution $\theta \in L^\infty([0,\infty);H^1) \cap L^2_{loc}((0,\infty);H^{3/2})$ of the initial value problem \eqref{eq:SQG:1}--\eqref{eq:SQG:3}. For any $t_1>0$ we have $\theta \in L^\infty([t_1,\infty);H^{3/2}) \cap L^2_{loc}([t_1,\infty);H^2)$. 
\end{theorem}

It is clear that if we would furthermore assume $f \in C^\infty(\TT^2)$, then $\theta \in C^\infty([t_1,\infty) \times \TT^2)$, for any $t_1 >0$. 

\begin{proof}[Proof of Theorem~\ref{thm:global:regularity}]
 By the local existence result of Proposition~\ref{prop:local}, there exists a time $T_* =  T_*(\theta_0,f) > 0$, and a unique solution $\theta \in L^\infty_t H^1_x \cap L^2_t H^{3/2}_x$ of \eqref{eq:SQG:1}--\eqref{eq:SQG:3} on $[0,T_*)$. In addition, by the local smoothing estimate \eqref{eq:local:smoothing}, since $f \in H^1$ we conclude that $\theta(t_1) \in H^2 \supset C^{\alpha_0}$, where $\alpha_0 = \alpha_0(\|\theta_0\|_{L^\infty}, \|f\|_{L^\infty})$ is as in Theorem~\ref{thm:Holder:propagation}, and $t_1 \in (0,T_*)$ is arbitrary. The propagation of H\"older continuity result of Theorem~\ref{thm:Holder:weak}, with initial data $\theta(t_1)$ then yields
\[
\sup_{t \in [t_1, T_* - \tau]} [\theta(t)]_{C^{\alpha_0}} \leq C(\theta_0,f)
\]
where the constant $C(\theta_0,f)$ is {\em independent} of $\tau$. Since the H\"older bound is supercritical for the natural scaling of the equation, we use the $C^{\alpha_0}$ version of the nonlinear lower bound for the fractional Laplacian~\cite[Theorem 2.2]{ConstantinVicol12} in order to bootstrap in regularity and obtain
\[
\sup_{t \in [t_1, T_* - \tau]} \|\theta(t)\|_{H^{3/2}}^2 + \int_{t_1}^{T_* - \tau} \| \theta(t)\|_{H^2}^2 dt \leq C(\theta_0,f)
\]
independently of $\tau$. The proof of this bootstrap procedure is given as part the proof of Theorem~\ref{thm:H:3/2:absorbing} below, and we omit the details here to avoid redundancy. The above estimate in particular shows that the solution $\theta$ may be continued uniquely past $T_*$, since for data in $H^{3/2}$ the time of existence of the $L^\infty_t H^1_x \cap L^2_t H^{3/2}_x$ solution depends only on the size of the {\em norm} in $H^{3/2}$ (we do not have this fact available if the initial data merely lies in $H^1$). Having extended the solution past $T_*$, we repeat the above argument and conclude the proof of global regularity.
\end{proof}

\section{Existence of a global attractor}

In view of the global existence  established in Theorem~\ref{thm:global:regularity}, we define a {\em solution operator} $S(t)$ for the initial value problem \eqref{eq:SQG:1}--\eqref{eq:SQG:3} via
\begin{align}
S(t) \colon H^1 \to H^1, \qquad S(t) \theta_0 = \theta(\cdot,t), \label{eq:S(t):def}
\end{align}
for any $t \geq 0$. In this section we establish (cf.~Theorem~\ref{thm:attractor:existence} below) the existence of a global attractor $\AA$ for the long-time dynamics of $S(t)$ on the phase space $H^1$. 

\begin{theorem}[\bf Existence of a global attractor]
\label{thm:attractor:existence}
The solution map $S \colon [0,\infty) \times H^1 \to H^1$ associated to \eqref{eq:SQG:1}--\eqref{eq:SQG:2} with $f\in L^\infty \cap H^1$, possesses a global attractor $\AA$ which is an compact invariant connected set, with $S(t) \AA = \AA$ for all $t \in \RR$, and such that for every $\theta_0 \in H^1$ we have
\[
\lim_{t\to \infty} {\rm dist}(S(t) \theta_0, \AA) = 0.
\]
The set $\AA$ is maximal in the sense that for any bounded subset $\BB_1 \subset H^{1+\delta}$ with $\delta>0$, which is invariant under $S(t)$, obeys $\BB_1 \subset \AA$. Moreover, there exists $M_{\AA}$ which depends only on $\kappa, \|f\|_{L^\infty\cap H^1}$, and universal constants, such that 
if $\theta \in \AA$, we have that 
\begin{align}
\| \theta \|_{H^{3/2}} \leq M_{\AA} \label{eq:M:32:A}
\end{align}
and
\begin{align}
\frac{1}{T} \int_{t}^{t+T} \|S(\tau) \theta\|_{H^2}^2 d\tau \leq M_{\AA}^2 \label{eq:M:2:A}
\end{align}
for any $T>0$ and $t \in \RR$. In particular, for $\theta \in \AA$ we have $\| S(t) \theta \|_{H^2} \leq M_{\AA}$ for almost every $t$.
\end{theorem}

The proof of Theorem~\ref{thm:attractor:existence}, given at the end of this section, follows closely the steps outlined in \cite{ConstantinFoias88}, and relies on the following main ingredients: 
\begin{itemize}
\item[(i)] There exists a compact absorbing set $\BB$ (which is a ball around the origin in $H^{3/2}$) for the dynamics induced by $S(t)$ on the phase space $H^1$ (cf.~Theorem~\ref{thm:H:3/2:absorbing} below).
\item[(ii)] The solution map $S(t) \colon H^1 \to H^1$ is injective on $\BB$ (cf.~Proposition~\ref{prop:backwards}).
\item[(iii)] For each $\theta_0 \in H^1$ the solution $S(t)\theta_0 \colon [0,\infty)  \to H^1$ is a continuous function of $t$, and for fixed $t>0$, we have that $S(t) \colon \BB \to H^1$ is a Lipschitz continuous function of $\theta_0$ (cf.~Proposition~\ref{prop:continuity}).
\end{itemize}
Establishing (i), the existence of a compact absorbing ball, turns out to be the most important step. For this we need to use the global regularity {\em twice}. From the local existence of solutions we pick up a time when a $C^{\alpha}$ norm of the solution is finite, with $\alpha$ small. Then we guarantee first that the solution satisfies strong bounds for all time, but the bounds depend of the initial data. However, after long enough time the $L^{\infty}$ norm of the solution obeys a bound that no longer depends on initial data (its size depends solely on force). At that time, because we have guaranteed that the solution remained smooth enough in the meantime, we apply again the $C^\alpha$ persistence result, but this time the size of $\theta$ in $L^{\infty}$ is given by $f$, which permits a calculation with an $\alpha$ that depends only on $f$. After an additional time, we obtain a bound of this $C^{\alpha}$ norm that depends only on $f$. At this stage, because the bounds make the situation subcritical, with constants which depend on $f$ only, we bootstrap in regularity and obtain that the size of the $H^{3/2}$ norm is determined by $f$ alone.

\begin{theorem}[\bf Absorbing ball in $H^{3/2}$]
\label{thm:H:3/2:absorbing}
Let $\theta_0 \in H^1$ and $f \in L^\infty \cap H^1$.
There exists a  time $t_{H^{3/2}} = t_{H^{3/2}}(\theta_0,f)$ and an $M_{3/2,f} = M_{3/2,f} (\|f\|_{L^\infty \cap H^1})$ such that  for all $t\geq t_{H^{3/2}}$ we have
\begin{align}
\|S(t) \theta_0\|_{H^{3/2}} \leq M_{3/2,f}.
\label{eq:H32:absorbing:ball}
\end{align}
That is, 
\begin{align}
\BB = \{ \theta \in H^{3/2} \colon \|\theta\|_{H^{3/2}} \leq M_{3/2,f} \}
\end{align}
is an absorbing set. Moreover, there exists an $M_{2,f} = M_{2,f} (\|f\|_{L^\infty \cap H^1})$, such that 
\begin{align}
\frac{1}{T} \int_{t}^{t+T} \|S(\tau) \theta_0\|_{H^2}^{2} d\tau \leq M_{2,f}^2
\label{eq:M:2:f}
\end{align}
for any $t\geq t_{H^{3/2}}$ and any $T > 0$.
\end{theorem}

As described in the above outline, in order to prove Theorem~\ref{thm:H:3/2:absorbing}, we first need to show that after waiting long enough time, the solution belongs to a H\"older space, with both the H\"older exponent and the H\"older norm, independent of the initial data. We achieve this in the following lemma, by combining the estimates established in Proposition~\ref{prop:Lp} and Theorem~\ref{thm:Holder:propagation} (respectively Theorem~\ref{thm:Holder:weak}).

\begin{lemma}[\bf Absorbing ball in $C^\alpha$]\label{lem:C:alpha:absorbing}
Let $\theta_0 \in H^1$, $f \in L^\infty \cap H^1$, and define the H\"older exponent
\begin{align}
\alpha_* = \alpha_*(\|f\|_{L^\infty}) := \min \left\{ \frac{\eps_1 \kappa^2}{\|f\|_{L^\infty}}, \frac 14 \right\}
\label{eq:alpha:*}
\end{align}
where $\eps_1 >0$ is a universal constant.
There exists a time
$t_{\alpha_*}=  t_{\alpha_*} ( \theta_0, f)$, such that 
\begin{align}
\|S(t) \theta_0 \|_{C^{\alpha_*}} \leq M_{\infty,f} := \frac{2 \|f\|_{L^\infty}}{ \eps_1 \kappa} 
\label{eq:Holder:absorbing:ball}
\end{align}
for all $t \geq t_{\alpha_*}$.
\end{lemma}

\begin{proof}[Proof of Lemma~\ref{lem:C:alpha:absorbing}]
By Proposition~\ref{prop:local} there exists 
$
t_0 = t_0 (\theta_0,f) > 0
$
such that $S(t_0) \theta_0 \in H^2$ and moreover by \eqref{eq:local:smoothing} we have
\begin{align}
\| S(t_0) \theta_0 \|_{H^2} \leq C \kappa^{-1} t_0^{-1} \|\theta_0\|_{H^1} + C  \kappa^{-1} \|f\|_{H^1} . 
\label{eq:S:t0:H2}
\end{align}
By the Sobolev embedding it follows from \eqref{eq:S:t0:H2} that
\begin{align}
\| S(t_0)\theta_0\|_{C^{1/4}} 
&\leq C \| S(t_0) \theta_0 \|_{H^2} \notag\\
&\leq C  \kappa^{-1} t_0^{-1} \|\theta_0\|_{H^1} + C  \kappa^{-1} \|f\|_{H^1}  = C(\kappa, \theta_0,f).
\label{eq:S:t0:C14}
\end{align}
In particular, $\| S(t_0) \theta_0\|_{L^\infty} \leq C(\kappa, \theta_0,f)$ holds.

Throughout this proof the value of $C(\kappa, \theta_0,f)$ may change from line to line, since we just want to emphasize the dependence of this bound solely on data and force.

We now apply Theorem~\ref{thm:Holder:propagation} with initial data  $S(t_0)\theta_0$.
Let $\eps_0$ be the constant in Theorem~\ref{thm:Holder:propagation}, and define 
\begin{align*}
\alpha_1 = \min \left\{ \frac{\eps_0\kappa}{M_\infty(S(t_0)\theta_0,f)}, \frac 14 \right\}
\end{align*}
where
\begin{align*}
M_\infty(S(t_0)\theta_0,f) = \|S(t_0)\theta_0\|_{L^\infty} + \frac{\|f\|_{L^\infty}}{c_0 \kappa}
\end{align*}
is as defined as in \eqref{eq:Mp:def}, with corresponding constant $c_0$. 
Since  $S(t_0)\theta_0 \in C^{1/4} \subseteq C^{\alpha_1}$ and \eqref{eq:S:t0:C14} holds, 
it follows from  estimate \eqref{eq:M:alpha:global}, that
\begin{align*}
[S(t) \theta_0]_{C^{\alpha_1}} \leq C(\kappa, \theta_0,f)
\end{align*}
for all $t \geq t_0$. Since in fact we know $S(t_0) \theta_0 \in H^{3/2}$, we bootstrap the above global in time estimate for the $C^{\alpha_1}$ norm, which is a subcritical quantity, to obtain
\begin{align}
\| S(t)\theta_0\|_{H^{3/2}} \leq C(\kappa, \theta_0,f)
\label{eq:S:t0:H32}
\end{align}
for all $t\geq t_0$. We refer to the Proof of Theorem~\ref{thm:H:3/2:absorbing} for the main idea in this bootstrap argument.

We now apply Proposition~\ref{prop:Lp}, with initial data $S(t_0)\theta_0 \in L^\infty$ bounded as in \eqref{eq:S:t0:C14}, to conclude that there exists 
\begin{align*}
t_1  = t_1 (\theta_0,f) = \frac{1}{c_0 \kappa} \log \left( 1 + \frac{c_0 \kappa \|S(t_0) \theta_0\|_{L^\infty}}{\|f\|_{L^\infty}} \right)
\end{align*}
such that 
\begin{align*}
\|S(t) \theta_0\|_{L^\infty} \leq \frac{2 \|f\|_{L^\infty}}{c_0\kappa}  
\end{align*}
for all $t \geq t_1 + t_0$. 

Now finally define
\begin{align*}
\alpha_* = \min \left\{ \frac{\eps_0 c_0 \kappa^2}{2 \|f\|_{L^\infty}}, \frac 14 \right\}
\end{align*}
and apply the argument in Theorem~\ref{thm:Holder:propagation}, with initial data taken to be $S(t_0 + t_1)\theta_0 \in H^{3/2} \subseteq  C^{\alpha_*}$.
We conclude from \eqref{eq:M:alpha:long:time} and \eqref{eq:S:t0:H32} that there exists 
$
t_{\alpha_*} = t_{\alpha_*}(\theta_0,f) 
$
with $t_1 + t_0 \leq t_{\alpha_*} < \infty$, such that 
\begin{align*}
\| S(t)\theta_0 \|_{C^{\alpha_*}} \leq \frac{4 \|f\|_{L^\infty}}{c_0\eps_0 \kappa}  
\end{align*}
holds for all $t \geq t_{\alpha_*}$, which concludes the proof the theorem.
\end{proof}

The proof of Theorem~\ref{thm:H:3/2:absorbing} now follows from Lemma~\ref{lem:C:alpha:absorbing} and a bootstrap procedure, which is based on the sub-criticality of the H\"older norm and the nonlinear lower bound on the fractional Laplacian~ \cite{ConstantinVicol12}.

\begin{proof}[Proof of Theorem~\ref{thm:H:3/2:absorbing}] Let $\alpha_* = \alpha_*(\|f\|_{L^\infty})$, and $t_{\alpha_*} = t_{\alpha_*}(\theta_0,f)$ be as defined in Lemma~\ref{lem:C:alpha:absorbing}.
Using estimate \eqref{eq:Holder:absorbing:ball}, we have that 
\begin{align}
\| S(t) \theta_0\|_{C^{\alpha_*}} \leq M_{\infty,f} \quad \left( := \frac{2 \|f\|_{L^\infty}}{\eps_1 \kappa}  \right)
\label{eq:C:alpha:global:bnd}
\end{align}
for all $t \geq t_{\alpha_*}$, and 
moreover, by \eqref{eq:S:t0:H32} we know that 
\begin{align}
\|S(t) \theta_0\|_{H^{3/2}}^2 \leq C(\kappa, \theta_0, f) < \infty
\label{eq:H32:bad:global:bound}
\end{align}
for all $t \geq t_{\alpha_*}$.

For the rest of the proof, denote
\begin{align}
\theta_0^* = S(t_{\alpha_*}) \theta_0.
\label{eq:theta:0:*}
\end{align}
The first step is to obtain a bound on time averages of the $H^{3/2}$ norm of the solution. We apply $\nabla$ to \eqref{eq:SQG:1} and pointwise in $x$ take inner product with $\nabla \theta$, and apply Proposition~\ref{prop:pointwise} to obtain
\begin{align}
\left( \partial_t + u \cdot \nabla + \kappa \Lambda \right) |\nabla \theta|^2 + \kappa D[\nabla \theta] = - 2  \partial_k u_j \partial_j \theta \partial_k \theta +2  \nabla f \cdot \nabla \theta
\label{eq:grad:PDE:1}
\end{align}
where
\begin{align*}
D[\nabla \theta](x) = \frac{1}{2\pi} P.V. \int_{\RR^2} |\nabla \theta(x)  - \nabla \theta(x+y) |^2 \frac{1}{|y|^3} dy
\end{align*}
and we use the same notation for the $\TT^2$-periodic function $\nabla \theta$, and its extension to all of $\RR^2$ by periodicity.
The main observation here is that since we already have from \eqref{eq:C:alpha:global:bnd} a bound for $\sup_{t \geq t_{\alpha_*}} \|S(t) \theta_0\|_{C^{\alpha_*}}$, 
we have an {\em improved nonlinear lower bound} on $D[\nabla \theta]$. Indeed, from~\cite[Theorem 2.2]{ConstantinVicol12} we have 
\begin{align}
D[\nabla \theta](x) \geq \frac{|\nabla \theta(x)|^{\frac{3-\alpha_*}{1-\alpha_*}}}{C [\theta]_{C^{\alpha_*}}^{\frac{1}{1-\alpha_*}}}
\label{eq:D:grad:nonlinear}
\end{align}
where the constant $C$ depends only on $\alpha_*$, and is uniformly bounded for $\alpha_* \in (0,1/2]$. Combining \eqref{eq:D:grad:nonlinear} with \eqref{eq:C:alpha:global:bnd}, we arrive at
\begin{align}
D[\nabla \theta](x,t) 
\geq \frac{ |\nabla \theta(x,t)|^{\frac{3-\alpha_*}{1-\alpha_*}}}{c_7 M_{\infty,f}^{\frac{1}{1-\alpha_*}}} 
\label{eq:D:grad:alpha}
\end{align}
for all $t\geq t_{\alpha_*}$, where $c_7$ is a universal constant which is independent of $\alpha_*$, for $\alpha_* \in (0,1/2]$.

Next, we estimate the nonlinear term on the right side of \eqref{eq:grad:PDE:1}. Let $\chi$ be a smooth cutoff function, that is $1$ on $[0,1]$, non-increasing, vanishes on $[2,\infty)$, and obeys $|\chi'|\leq 2$. For $\rho > 0$ to be determined, using the same argument which led to \eqref{eq:delta:h:u}, we obtain
\begin{align*}
| \nabla u(x) | 
&\leq \frac{1}{2\pi} \left| P.V. \int_{\RR^2} \frac{y^\perp}{|y|^3} \chi\left( \frac{|y|}{\rho} \right) \nabla \theta(x+y) dy \right| + \frac{1}{2\pi} \left|  \int_{\RR^2} \frac{y^\perp}{|y|^3}\left(1- \chi\left( \frac{|y|}{\rho} \right) \right)\nabla \theta(x+y) dy \right| \notag\\
&\leq C \Bigl( \rho D[\nabla \theta](x) \Bigr)^{1/2} + C \frac{\|\theta\|_{L^\infty}}{\rho}
\end{align*}
for some universal constant $C>0$.
Therefore, we have
\begin{align}
2 |\nabla u(x)| |\nabla \theta(x)|^2 
&\leq C \Bigl( \rho D[\nabla \theta](x) \Bigr)^{1/2} |\nabla \theta(x)|^2 + C \frac{\|\theta\|_{L^\infty}}{\rho} |\nabla \theta(x)|^2 \notag\\
&\leq \frac{\kappa}{2} D[\nabla \theta](x) + C \left( \frac{\rho}{\kappa} |\nabla \theta (x)|^4 + \frac{\|\theta\|_{L^\infty}}{\rho} |\nabla \theta(x)|^2 \right) \notag\\
&\leq \frac{\kappa}{2} D[\nabla \theta](x)  + \frac{C}{\kappa^{1/2}} \|\theta\|_{L^\infty}^{1/2} |\nabla \theta(x)|^3
\label{eq:grad:nonlinear}
\end{align}
by letting $\rho = \kappa^{1/2} \|\theta\|_{L^\infty}^{1/2} |\nabla \theta(x)|^{-1}$.

Since $t \geq t_{\alpha_*}$, we combine \eqref{eq:C:alpha:global:bnd} with \eqref{eq:grad:PDE:1}, \eqref{eq:D:grad:alpha}, and \eqref{eq:grad:nonlinear}, to arrive at
\begin{align}
\left( \partial_t + u \cdot \nabla + \kappa \Lambda \right) |\nabla \theta|^2 + \frac{\kappa}{4} D[\nabla \theta] +  \frac{\kappa  |\nabla \theta(x,t)|^{\frac{3-\alpha_*}{1-\alpha_*}}}{4 c_7 M_{\infty,f}^{\frac{1}{1-\alpha_*}}} \leq 2 |\nabla f| |\nabla \theta| + \frac{c_8 M_{\infty,f}^{1/2} |\nabla \theta|^3}{\kappa^{1/2}}
\label{eq:grad:PDE:2}
\end{align}
for some universal constant $c_8$.
Using the $\eps$-Young inequality, since $(3-\alpha_*)/(1-\alpha_*) > 3$ we furthermore infer from \eqref{eq:grad:PDE:2} that
\begin{align}
\left( \partial_t + u \cdot \nabla +\kappa \Lambda \right) |\nabla \theta|^2 + \frac{\kappa}{4} D[\nabla \theta] +  \frac{\kappa |\nabla \theta(x,t)|^{\frac{3-\alpha_*}{1-\alpha_*}}}{8 c_7 M_{\infty,f}^{\frac{1}{1-\alpha_*}}} \leq 2 |\nabla f| |\nabla \theta| + \frac{c_8 (8c_7)^{\frac{3-3\alpha_*}{2\alpha_*}}}{\kappa^{\frac{9-7 \alpha_*}{4\alpha_*}}} M_{\infty,f}^{\frac{9 - \alpha_*}{4 \alpha_*}}.
\label{eq:grad:PDE:3}
\end{align}
Next we integrate \eqref{eq:grad:PDE:3} over $\TT^2$, and use
\begin{align*}
\frac 12 \int_{\TT^2} D[\nabla \theta](x) dx = \int_{\TT^2} \nabla \theta \cdot \Lambda \nabla \theta dx = \|\theta\|_{H^{3/2}}^2 \geq \|\theta\|_{H^1}^2
\end{align*}
to obtain
\begin{align}
\frac{d}{dt} \|\theta\|_{H^1}^2 + \frac{\kappa}{6} \|\theta\|_{H^1}^2 + \frac{\kappa}{6} \|\theta\|_{H^{3/2}}^2 \leq \frac{6}{\kappa} \|f\|_{H^1}^2 + \frac{c_8 (8c_7)^{\frac{3-3\alpha_*}{2\alpha_*}}}{\kappa^{\frac{9-7 \alpha_*}{4\alpha_*}}} M_{\infty,f}^{\frac{9 - \alpha_*}{4 \alpha_*}}
\label{eq:grad:ODE:1}
\end{align}
for times $t \geq t_{\alpha_*}$. Using the Gr\"onwall inequality we obtain from \eqref{eq:grad:ODE:1} that 
\begin{align}
\|S(t+t_{\alpha_*}) \theta_0\|_{H^1}^2 
&= \|S(t) \theta_0^*\|_{H^1}^2 \notag\\
&\leq \|\theta_0^*\|_{H^1}^2 e^{-\frac{t \kappa}{6}} +   \left(   \frac{36}{\kappa^2} \|f\|_{H^1}^2 + \frac{c_8 (8c_7)^{\frac{3-3\alpha_*}{2\alpha_*}}}{6 \kappa^{\frac{9-3 \alpha_*}{4\alpha_*}}} M_{\infty,f}^{\frac{9 - \alpha_*}{4 \alpha_*}}  \right) (1- e^{-\frac{t \kappa}{6}}).
\label{eq:grad:ODE:2}
\end{align}
Recall cf.~\eqref{eq:H32:bad:global:bound} and \eqref{eq:theta:0:*} that $\|\theta_0^*\|_{H^1}^2 = \|S(t_{\alpha_*})\theta_0\|_{H^1}^2 \leq C(\kappa, \theta_0,f)$. We conclude from \eqref{eq:grad:ODE:2} that there exists 
\begin{align*}
t_{H^1} = t_{H^1}( \theta_0,f) \geq t_{\alpha_*}
\end{align*}
such that for all $t \geq t_{H^1}$ we have
\begin{align}
\|S(t) \theta_0\|_{H^1}^2 \leq  \frac{72}{\kappa^2} \|f\|_{H^1}^2 + \frac{c_8 (8c_7)^{\frac{3-3\alpha_*}{2\alpha_*}}}{3 \kappa^{\frac{9-3 \alpha_*}{4\alpha_*}}} M_{\infty,f}^{\frac{9 - \alpha_*}{4 \alpha_*}} =: M_{1,f}^2.
\label{eq:H1:absorbing:ball}
\end{align}
Note that cf.~\eqref{eq:alpha:*} we have $\alpha_* = \alpha_* (\|f\|_{L^\infty})$ and cf.~\eqref{eq:Holder:absorbing:ball} we have $M_{\infty,f} = M_{\infty,f} (\|f\|_{L^\infty})$, so that $M_{1,f} = M_{1,f} (\kappa, \|f\|_{L^\infty \cap H^1})$. The dependence on $\kappa$ and $f$ may be computed explicitly from  \eqref{eq:alpha:*}--\eqref{eq:Holder:absorbing:ball}  and \eqref{eq:H1:absorbing:ball}.

Inequality \eqref{eq:H1:absorbing:ball} not only gives an absorbing ball in $H^1$, but combined with \eqref{eq:grad:ODE:1}, integrated between $t$ and $t+1$, it also gives the bound
\begin{align}
\int_{t}^{t+1} \|\theta(s)\|_{H^{3/2}}^2 ds \leq \frac{6+\kappa}{\kappa} M_{1,f}^2 
\label{eq:H32:time:average}
\end{align}
for all $t \geq t_{H^1}$.

Estimate \eqref{eq:H32:time:average} now directly implies the existence of an absorbing ball for $S(t)$ in $H^{3/2}$. To see this, we take the $L^2$ inner product of \eqref{eq:SQG:1} with $\Lambda^3 \theta$ and write
\begin{align}
\frac{d}{dt} \|\theta\|_{H^{3/2}}^2 + \kappa \|\theta\|_{H^2}^2 
&\leq \frac{1}{\kappa} \|f\|_{H^1}^2 + 2 \left| \int_{\TT^2} \left( \Lambda^{3/2} (u\cdot \nabla \theta) - u \cdot \nabla \Lambda^{3/2} \right) \Lambda^{3/2} \theta dx \right| \notag\\
&\leq \frac{1}{\kappa} \|f\|_{H^1}^2 + C \|\theta\|_{H^{3/2}} \|\Lambda^{3/2} \theta\|_{L^4} \|\Lambda \theta \|_{L^4} \notag\\
&\leq \frac{1}{\kappa} \|f\|_{H^1}^2 + C \|\theta\|_{H^{3/2}}^2  \| \theta\|_{H^2}  \notag\\
&\leq \frac{1}{\kappa} \|f\|_{H^1}^2 + \frac{\kappa}{2} \|\theta\|_{H^2}^2 + \frac{c_9}{\kappa} \|\theta\|_{H^{3/2}}^4
\label{eq:H32:ODE:1}
\end{align}
for some universal constant $c_9 > 0$.
In the above estimate we have appealed to the commutator estimate \eqref{eq:Hs:commutator} of Lemma~\ref{lem:fractional:calculus}, and we used the the Sobolev embedding $H^{1/2} \subset L^4$.
We obtain from \eqref{eq:H32:ODE:1} that 
\begin{align}
\frac{d}{dt} \|\theta\|_{H^{3/2}}^2 + \frac{\kappa}{2} \|\theta\|_{H^2}^2 \leq \frac{1}{\kappa} \|f\|_{H^1}^2 + \left( \frac{c_9}{\kappa} \|\theta\|_{H^{3/2}}^2 \right) \|\theta\|_{H^{3/2}}^2
\label{eq:H32:ODE:2}
\end{align}
for $t\geq 0$.

At this stage we apply the Uniform Gr\"onwall Lemma~\ref{lem:Gronwall}, with the functions
\begin{align*}
x= \|\theta\|_{H^{3/2}}^2, \quad a(t) = \frac{c_9}{\kappa} \|\theta\|_{H^{3/2}}^2, \quad b =  \frac{1}{\kappa} \|f\|_{H^1}^2
\end{align*} 
that by \eqref{eq:H32:time:average} obey the bounds
\begin{align*}
\int_{t}^{t+1} x(s) ds \leq  \frac{6+\kappa}{\kappa} M_{1,f}^2, \quad \int_{t}^{t+1} a(s) ds \leq \frac{c_9(6+\kappa)}{\kappa^2} M_{1,f}^2, \quad \int_t^{t+1} b(s)ds =  \frac{1}{\kappa} \|f\|_{H^1}^2
\end{align*}
for any $t \geq t_{H^1}$. We conclude from \eqref{eq:unif:Gronwall} that
\begin{align}
\|S(t)\theta_0\|_{H^{3/2}}^2 
\leq \left( \frac{6+\kappa}{\kappa} M_{1,f}^2 + \frac{1}{\kappa} \|f\|_{H^1}^2\right) \exp\left( \frac{c_9(6+\kappa)}{\kappa^2} M_{1,f}^2 \right) =: M_{3/2,f}^2
\label{eq:M:32:f:def}
\end{align}
for any  $t \geq t_{H^{3/2}}$, where  
\begin{align*}
t_{H^{3/2}} = t_{H^{3/2}}(\theta_0, f) =:   t_{H^1} + 1.
\end{align*}
We also note that  since both $\alpha^*$ and $M_{1,f}$ depend only on $\kappa$ and $\|f\|_{L^\infty \cap H^1}$ (and universal constants), we have $M_{3/2,f} = M_{3/2,f}(\kappa, \|f\|_{L^\infty \cap H^1})$. 

Lastly, we notice that $L^2_t H^2_x$ bounds are also available from the above argument. By combining \eqref{eq:H32:ODE:2} with \eqref{eq:M:32:f:def} we obtain
\begin{align}
\frac{1}{T} \int_{t}^{t+T} \|S(\tau) \theta_0\|_{H^2}^{2} d\tau \leq \frac{2}{\kappa^2} \|f\|_{H^1}^2 + \frac{2c_9}{\kappa^2} M_{3/2,f}^4 =: M_{2,f}^2
\label{eq:M:2:f:def}
\end{align}
for any $t\geq t_{H^{3/2}}$ and $T > 0$. This concludes the proof of the theorem.
\end{proof}

\begin{remark}[\bf Uniform attraction] 
\label{rem:uniform}
Theorem~\ref{thm:H:3/2:absorbing} guarantees that for $\theta_0 \in H^1$, there exists a time $t_{H^{3/2}}$ which depends on $\theta_0$ (and $f$) so that $S(t)\theta_0 \in \BB$ for $t \geq t_{H^{3/2}}$. Note however that $t_{H^{3/2}}(\theta_0,f)$ does not depend solely on $\|\theta_0\|_{H^1}$, and it is not a priori  locally uniform with respect to initial data. The sole reason for this is that the time of local existence of the solution arising from initial data $\theta_0 \in H^1$ is not guaranteed to depend only on $\|\theta_0\|_{H^1}$. On the other hand, we would like  to emphasize that  if $\theta_0 \in H^{1 + \delta}$ with $\delta >0$, it can be shown that$t_{H^{3/2}} = t_{H^{3/2}}(\|\theta_0\|_{H^{1 + \delta}},\|f\|_{L^\infty \cap H^1})$ and the time of entering the absorbing ball is a non-decreasing function of its arguments. {\em In particular, this implies that given a ball $\BB_R = \{ \theta \in H^{3/2} \colon \| \theta \|_{H^{3/2}} \leq R\}$, there exists a time $t_R = t_R (R,\|f\|_{L^\infty \cap H^1})$ such that $S(t) \BB_R \subset \BB$ for all $t \geq t_R$}. The reason for this fact is the following. For this smoother initial data $\theta_0 \in H^{1 + \delta}$, we find a local time of existence of a unique $L^\infty_t H^1_x \cap L^2_t H^{3/2}_x$ solution,  that depends only on $\| \theta_0 \|_{H^{1 + \delta}}$ and norms of $f$. Then going line-by-line through the proofs of Lemma~\ref{lem:C:alpha:absorbing} and Theorem~\ref{thm:H:3/2:absorbing} above shows that by waiting long enough, depending only on the $H^{1+\delta}$ norm of $\theta_0$ and on norms of $f$, the $C^{\alpha_*}$ norm and then the $H^{3/2}$ norm of $S(t)\theta_0$ are under control. To avoid redundancy we omit further details.
\end{remark}

As stated in the outline below Theorem~\ref{thm:attractor:existence}, besides having a compact absorbing set $\BB$ we need the injectivity of $S(t)$ on $\BB$ and continuity properties of $S(t)$ on $\BB$. The following lemma shows that the solution operator in {injective}.
\begin{proposition}[\bf Backwards uniqueness] \label{prop:backwards}
Let $\theta^{(i)}_0, \theta^{(2)}_0 \in H^1$ be two initial data, and let 
\[
\theta^{(i)}(t) = S(t) \theta^{(i)}_0 \in C([0,\infty);H^1) \cap L^2(0,\infty;H^{3/2})
\]
be the corresponding solutions of the initial value problem \eqref{eq:SQG:1}--\eqref{eq:SQG:3} for $i \in \{1,2\}$. If there exists $T>0$ such that $\theta^{(1)}(T) = \theta^{(2)}(T)$, then $\theta^{(1)}_0 = \theta^{(2)}_0$ holds.
\end{proposition}
The proof uses the classical log-convexity method of Agmon and Nirenberg~\cite{AgmonNirenberg67}, and is given in Appendix~\ref{app:attractor:existence} below.

The usual attractor theory requires that the solution map $S$ is continuous with respect to initial data for fixed time, and continuous with respect to time for fixed initial data, in the topology of $H^1$. The next lemma in particular proves the Lipschitz continuity of $S(t) \colon \BB \to H^1$ for fixed $t>0$. The proof is given in Appendix~\ref{app:attractor:existence} below. 

\begin{proposition}[\bf Continuity]\label{prop:continuity}
For {fixed} $\theta_0 \in H^1$ we have that $S(\cdot)\theta_0 \colon [0,\infty) \to H^1$ is continuous. Fix a ball $\BB_0 \in H^{3/2}$. We have that for $\theta_0, \tilde \theta_0 \in \BB_0$ with $\|\theta_0 - \tilde \theta_0\|_{H^1} \leq \eps \kappa$, for some universal $0 < \eps \ll 1$, we have 
\[
\|S(t) \theta_0 - S(t) \tilde \theta_0\|_{H^1} \leq e(t) \|\theta_0 - \tilde \theta_0\|_{H^1}
\]
for some non-decreasing continuous function of time $e(t)$. In particular, if $\{ t_n \}_{n\geq 1}$ is a sequence of times that diverge to $\infty$ as $n\to \infty$, and $\{ \theta_{0,n} \}_{n \geq 1} \subset \BB_0$ are a sequence of initial data such that 
\begin{align*}
\| S(t_n) \theta_{0,n} - \theta_0 \|_{H^1} \to 0
\end{align*}
as $n \to \infty$ for some $\theta_0 \in H^1$, then for any fixed $t>0$ we have
\begin{align*}
\| S(t+t_n) \theta_{0,n} - S(t) \theta_0\|_{H^1} = \| S(t) S(t_n) \theta_{0,n} - S(t) \theta_0\|_{H^1} \to 0
\end{align*}
as $n\to \infty$. 
\end{proposition}

We conclude this section with the proof of the existence of the global attractor on the phase space $H^1$. 

\begin{proof}[Proof of Theorem~\ref{thm:attractor:existence}]
The proof follows using the same argument given in \cite[pp.~133--136]{ConstantinFoias88}.
By Theorem~\ref{thm:H:3/2:absorbing} we have a compact absorbing set $\BB = \{ \theta \in H^{3/2} \colon \|\theta\|_{H^{3/2}} \leq M_{3/2,f} \}$, where $M_{3/2,f}$ can be computed in terms of  $\kappa$, $\|f\|_{L^\infty\cap H^1}$, and universal constants. The idea is that for any bounded $\BB_1 \subset H^{1+\delta}$ with $\delta>0$ by Remark~\ref{rem:uniform} we have that the omega limit set obeys $\omega (\BB_1) \subset \BB$ and thus, by Proposition~\ref{prop:continuity} we have 
\[
S(t)(\omega(\BB_1)) = \omega (\BB_1)
\] 
for all $t\geq 0$,
where the omega limit sets are in the $H^1$ topology. The solution map is continuous with respect to initial data in $\BB$, and the global attractor is just 
\[
\AA = \bigcap_{t>0}S(t) \BB.
\]
That $\lim_{t\to \infty} {\rm dist}(S(t) \theta_0, \AA) = 0$ for any $\theta_0 \in H^1$ follows since $S(t) \omega(\theta_0) = \omega(\theta_0) \subset \BB$ and the definition of $\AA$.
The invariance of $\AA$ for all time follows from the backward uniqueness on $\BB$ established in Proposition~\ref{prop:backwards}. The bounds \eqref{eq:M:32:A}--\eqref{eq:M:2:A} follow by taking $M_\AA = \max\{ M_{3/2,f} , M_{2,f}\}$, where $M_{3/2,f}$ and $M_{2,f}$ are as defined by \eqref{eq:M:32:f:def} and \eqref{eq:M:2:f:def} above.
\end{proof}

\begin{remark}[\bf Higher regularity]
If $f \in C^\infty(\TT^2)$, it can be shown that in fact $\AA \subset C^\infty(\TT^2)$ with bounds that depend only on $\kappa$ and $f$. Similar statements hold in the real analytic or Sobolev categories.
\end{remark}

\section{Finite dimensionality of the attractor}

In this section we establish a bound on the fractal (and a forteriori  Hausdorff) dimension of the global attractor $\AA$ for $S(t)$ evolving on $H^1$. The physical meaning of this bound is that the long-time behavior of solutions to the forced critical SQG equations can be fully described by a finite number of independent degrees of freedom. 

We recall that the fractal dimension $d_f(Z)$ of a compact set $Z$ is given by 
\begin{align*}
d_f(Z) = \limsup_{r\to 0} \frac{\log n_{Z}(r)}{\log(1/r)}
\end{align*}
where $n_Z(r)$ is the minimal number of balls in $H^1$ of radii $r$ needed to cover $Z$. Note that fractal dimension gives an upper bound (which may be strict) for the Hausdorff dimension of a compact set $Z$. 

The proof closely follows the outline given in~\cite{ConstantinFoias85, ConstantinFoias88}, where the connection with global Lyapunov exponents and the Kaplan-Yorke formula is established. The main idea is as follows. Assume $\AA$ is covered by a finite number of balls of radius $r$. Let the flow $S(t)$ transport a ball of radius $r$ centered at $\theta_0$. Then up to an $o(r)$ error the image of the ball is an ellipsoid centered at $S(t) \theta_0$, with semi-axes on the directions given by the eigenvalues of $M(t,\theta_0)$, of lengths given by $r$ multiplied by the eigenvalues of $M(t,\theta_0)$, where $M(t,\theta_0) = \left( S'(t,\theta_0)^\ast S'(t,\theta_0) \right)^{1/2}$, and $S'(t,\theta_0)$ is the Fr\'echet derivative of $S(t) \theta_0$. A control on the volume of this ellipsoid (given in terms of the product of the eigenvalues of $M(t,\theta_0)$) then gives a bound on the number of balls of radius $r$ needed to (re-)cover the ellipsoid. It then turns out that in order to estimate the fractal dimension of $\AA$, it is sufficient to find an integer $N$ with the property that $n$-dimensional volume elements carried by the flow decay exponentially, for all $n\geq N+1$. We now make these ideas more precise.

\begin{definition}[\bf Continuous differentiability of $S(t)$]
\label{def:conts:diff}
The solution map $S(t)$ is continuously differentiable on $\AA$ if for every $\theta_0 \in \AA$ there exists a linear operator
\begin{align*}
S'(t,\theta_0) \colon H^1 \to H^1
\end{align*}
and a positive function $e(r,t)$, which is continuos with respect to both variables, such that 
\begin{align}
\sup_{\theta_0,\phi_0 \in \AA, 0 < \|\theta_0 - \phi_0\|_{H^1} \leq r} \frac{\| S(t) \phi_0 - S(t)\theta_0 - S'(t,\theta_0)[\phi_0 - \theta_0] \|_{H^1}^2}{\|\phi_0 - \theta_0\|_{H^1}^2} \leq e(r,t)
\label{eq:Frechet:differentiability}
\end{align}
with
\begin{align}
\lim_{r\to 0^+} e(r,t) = 0
\end{align}
and moreover
\begin{align}
\sup_{\theta_0 \in \AA, \|\xi_0\|_{H^1}=1} \|S'(t,\theta_0) [\xi_0]\|_{H^1} < \infty
\label{eq:Frechet:boundedness}
\end{align}
for every $t\geq 0$. 
\end{definition}

The solution map $S(t)$ induced by the critical SQG equation is indeed continuously differentiable on $\AA$. 
For $\theta_0\in\AA$, write $\theta = \theta(t) = S(t) \theta_0$ and for $\xi \in H^1$ let us introduce the elliptic operator
\begin{align} 
A_{\theta_0}(t)[ \xi] = A_\theta [\xi] = - \kappa \Lambda \xi - \RSZ^\perp \theta \cdot \nabla \xi - \RSZ^\perp \xi \cdot \nabla \theta.
\label{eq:A:xi:def}
\end{align}
We express $S'(t,\theta_0)$ using  $A_{\theta_0}(t)$.
\begin{proposition}[\bf Linearization about a trajectory on the attractor] \label{prop:unif:diff}
The solution map $S(t)$ associated to \eqref{eq:SQG:1}--\eqref{eq:SQG:2} is continuously differentiable on $\AA$. Moreover, the linear operator $S'(t,\theta_0)$, when acting on an element $\xi_0 \in H^1$ is given by
\begin{align*}
S'(t,\theta_0)[\xi_0]  = \xi(t)
\end{align*}
where $\xi(t)$ is the solution of 
\begin{align}
\partial_t \xi = A_{\theta_0}(t) [\xi] := - \kappa \Lambda \xi - \RSZ^\perp \theta \cdot \nabla \xi - \RSZ^\perp \xi \cdot \nabla \theta, \quad \xi(0) = \xi_0.
\label{eq:L:xi:def}
\end{align}
Also, for any $t>0$ the operator $S'(t,\theta_0)$ is compact.
\end{proposition}

It follows from the proof that the function $e(r,t)$ in Definition~\ref{def:conts:diff} may be taken $\approx r^{2-a} \exp(C t)$ for some $a \in (0,1)$ and some $C>0$, which depends only on $\kappa$ and $\|f\|_{L^\infty\cap H^1}$. The proof of Proposition~\ref{prop:unif:diff} is quite technical, and we defer it to Appendix~\ref{app:attractor:existence}.

We next show that there is an $N$ such that volume elements which are carried by the flow of $S(t) \theta_0$, with $\theta_0 \in \AA$, decay exponentially for dimensions larger than $N$. Consider $\theta_0 \in \AA$, and an initial orthogonal set of infinitesimal displacements $\{\xi_{1,0},\ldots, \xi_{n,0}\}$ for some $n \geq 1$. The volume of the parallelepiped they span is given by 
\begin{align*}
V_n(0) = \| \xi_{1,0} \wedge \ldots \wedge \xi_{n,0} \|_{H^1}.
\end{align*}
The reason we have introduced in Proposition~\ref{prop:unif:diff} the linearization $S'(t,\theta_0)$ of the flow near $S(t) \theta_0$ is that these displacements $\xi_{i}$ evolve exactly under this linearization, that is, we define 
\begin{align*}
 \xi_i(t) = S'(t,\theta_0) [\xi_{i,0}] \quad \mbox{for all } i \in \{1,\ldots,n\}, \mbox{ and } t\geq 0,
\end{align*}
or equivalently the $\xi_i$ obey the equation
\begin{align*}
\partial_t \xi_i = A_{\theta_0}(t) [\xi_i], \quad \xi_i(0) = \xi_{i,0},
\end{align*}
where $A_{\theta_0}(t)$ is defined in \eqref{eq:A:xi:def} above. Then it follows cf.~\cite{ConstantinFoias85, ConstantinFoias88} that the volume elements
\begin{align*}
V_n(t) = \| \xi_{1}(t) \wedge \ldots \wedge \xi_{n}(t) \|_{H^1}
\end{align*}
satisfy 
\begin{align*}
V_n(t) = V_n(0) \exp\left( \int_0^t {\rm Tr} (P_n(s) A_{\theta_0}(s) ) ds \right)
\end{align*} 
where the orthogonal projection $P_n(s)$ is onto the linear span of $\{ \xi_1(s), \ldots , \xi_n(s)\}$ in the Hilbert space $H^1$, and ${\rm Tr}(P_n(s) A_\theta)$ is defined by
\begin{align} 
{\rm Tr} (P_n(s) A_\theta) = \sum_{j=1}^n \int_{\TT^2} (-\Delta \phi_j(s)) A_\theta[\phi_j(s)] dx
\label{eq:trace:Pn}
\end{align}
for $n\geq 1$, with $\{ \phi_{1}(s), \ldots , \phi_n(s)\}$ an orthornormal set spanning the linear span of  $\{ \xi_1(s), \ldots , \xi_n(s)\}$. The value of ${\rm Tr} (P_n(s) A_\theta) $ does not depend on the choice of this orthonormalization.
Therefore, letting
\begin{align}
\la P_n A_{\theta_0} \ra := \limsup_{T\to \infty} \frac{1}{T} \int_0^T {\rm Tr}(P_n(t) A_{\theta_0}(t) ) dt
\label{eq:limsup:trace:Pn}
\end{align}
we obtain
\begin{align}
V_n(t) \leq V_n(0) \exp\left( t \sup_{\theta_0 \in \AA} \sup_{P_n(0)} \la P_n A_{\theta_0} \ra \right) 
\label{eq:volume:decay}
\end{align}
for all $t \geq 0$, where the supremum over $P_n(0)$ is a supremum over all choices of initial $n$ orthogonal set of infinitesimal displacements that we take around $\theta_0$.

Next, we show that $n$-dimensional volume elements decay exponentially in time (at a rate that is bounded from below), whenever $n$ is sufficiently large, {\em independently} of the choice of $\theta_0$ in $\AA$, and {\em independently} of initial set of orthogonal displacements $\{\xi_{i,0}\}_{i=1}^n$ which define $P_n(0)$. The key is to to show that the symmetric part of the operator $A_{\theta_0}(t)$ obeys good quadratic form bounds in the $H^1$ topology.

\begin{proposition}[\bf Contractivity of large dimensional volume elements] 
\label{prop:n:contraction}
There exists $N = N( \kappa,M_{\AA} )$ such that for any $\theta_0 \in \AA$ and any set of initial orthogonal displacements $\{\xi_{i,0}\}_{i=1}^n$, we have 
\begin{align} 
\la P_n A_{\theta_0} \ra  < 0
\label{eq:neg:Lyapunov}
\end{align}
whenever $n \geq N$. In particular, $V_n(t)$ decays exponentially in $t$ for any $n\geq N$.
\end{proposition}

\begin{proof}[Proof of Proposition~\ref{prop:n:contraction}]

Let $\xi \in H^1$ be arbitrary. By the definition of $A_\theta[\xi]$ in \eqref{eq:A:xi:def}, and the fact that $\RSZ^\perp \theta$ is divergence-free, we have 
\begin{align*} 
\int_{\TT^2} \Lambda^2 \xi A_\theta[\xi] dx 
&= - \kappa \|\xi\|_{H^{3/2}}^2 + \left| \int_{\TT^2} \partial_k \RSZ^\perp \theta \cdot \nabla \xi \partial_k \xi dx \right| + \left| \int_{\TT^2} \partial_k ( \RSZ^\perp \xi \cdot \nabla \theta) \partial_k \xi dx \right| \notag\\
&\leq  - \kappa \|\xi\|_{H^{3/2}}^2 +  C \|\theta\|_{H^{3/2}} \|\xi\|_{H^{3/2}} \| \xi\|_{H^1} + C \| \xi\|_{H^{1/2}} \| \theta\|_{H^{2}} \| \xi\|_{H^{3/2}}.
\end{align*}
Here we have appealed to the Sobolev embedding $H^{1/2} \subset L^4$. Using the Poincar\'e inequality it follows that
\begin{align}
\int_{\TT^2} \Lambda^2 \xi A_\theta[\xi] dx 
&\leq  - \frac{\kappa}{2} \|\xi\|_{H^{3/2}}^2 +  \frac{c_{10}}{\kappa} \|\theta\|_{H^{2}}^2 \| \xi\|_{H^1}^2.
\label{eq:A:H1}
\end{align}
for some universal constant $c_10>0$.
Now, for any $\theta_0$ and any $t\geq 0$ the definition \eqref{eq:trace:Pn}, the inequality \eqref{eq:A:H1}, the normalization of the $\phi_j$'s (recall that $\{ \phi_{1}(t), \ldots , \phi_n(t)\}$ an orthornormal set spanning the linear span of  $\{ \xi_1(t), \ldots , \xi_n(t)\}$), and estimate \eqref{eq:M:2:A} yield
\begin{align} 
\frac{1}{T} \int_0^T {\rm Tr} (P_n(t) A_{\theta_0}(t) ) dt
&= \frac{1}{T} \int_0^T \sum_{j=1}^n \int_{\TT^2} (-\Delta \phi_j(t)) A_\theta[\phi_j(t)] dx dt \notag\\
&\leq - \frac{\kappa}{4} \frac{1}{T} \int_0^T \sum_{j=1}^n \|\phi_j(t) \|_{H^{3/2}}^2 dt + \frac{c_{10}}{\kappa} \frac{1}{T} \int_0^T \|\theta(t)\|_{H^2}^2 \sum_{j=1}^n \|\phi_j(t) \|_{H^1}^2 dt \notag\\
&\leq - \frac{\kappa}{4} \frac{1}{T} \int_0^T {\rm Tr} (P_n(t) \Lambda) + n \frac{c_{10}}{\kappa} \frac{1}{T} \int_0^T \|\theta(t)\|_{H^2}^2 dt \notag\\
&\leq - \frac{\kappa}{c_{11}} n^{3/2} + n \frac{c_{10}}{\kappa}  M_\AA^2 \label{eq:negative:trace}
\end{align}
where in the last inequality we have used that the eigenvalues $\{\lambda_j\}_{j\geq 1}$ of $\Lambda^{1/2}$ obey
\begin{align*} 
\lambda_j \geq \frac{1}{c_{11}} j^{1/2}
\end{align*}
for a sufficiently large universal constant $c_{11}>0$ (see e.g.~\cite{ConstantinFoias88}).
Choosing
\begin{align} 
N = N(\kappa,M_{\AA} ) = \left \lceil \left(  c_{10}c_{11} \kappa^{-2} M_\AA^2     \right)^2 \right\rceil
\label{eq:N:def}
\end{align}
the lemma now follows directly from \eqref{eq:negative:trace} and the definition \eqref{eq:limsup:trace:Pn}.
\end{proof}

The upshot of Proposition~\ref{prop:n:contraction} is that  $N$-dimensional volume elements decay exponentially in time. This also implies that the fractal (box-counting) dimension of $\AA$ is finite, and is bounded by this $N$.

\begin{theorem}[\bf Finite dimensionality of the attractor]
\label{thm:finite:attractor}
Let $N = N(\kappa^{-1} M_\AA )$ be as defined in \eqref{eq:N:def} above. Then the fractal dimension of $\AA$ is finite, and we have
${\rm dim}_f (\AA) \leq N$.
\end{theorem}

\begin{proof}[Proof of Theorem~\ref{thm:finite:attractor}]
We follow precisely the lines of the argument in~\cite[pp.~115--130, and Chapter~14]{ConstantinFoias88}. The main ingredients are the continuous differentiability of $S(t)$ on $\AA$, the compactness of the linearization, and the exponential decay of large-dimensional volume elements which follows from \eqref{eq:volume:decay} and \eqref{eq:neg:Lyapunov}. We omit further details and refer to~\cite{ConstantinFoias88}.
\end{proof}

\appendix

\section{Fractional inequalities}
\label{app:calculus}

We recall the following fractional product (Kato-Ponce), commutator (Kenig-Ponce-Vega), and Sobolev estimates, cf.~\cite{KatoPonce88,KenigPonceVega91,Taylor91,SchonbekSchonbek03,Ju04} and references therein.

\begin{lemma}[\bf Fractional calculus]\label{lem:fractional:calculus}
Lef $f,g \in C^\infty(\TT^2)$, $s >0$, and $p \in (1,\infty)$. Then we have that 
\begin{align}
\|\Lambda^s (fg)\|_{L^{p}} \leq C  \|g\|_{L^{p_1}} \|\Lambda^s f\|_{L^{p_2}}+ C \|\Lambda^s g\|_{L^{p_3}}\|f\|_{L^{p_4}},
\label{eq:Hs:product}
\end{align}
where $1/p = 1/p_1 + 1/p_2 = 1/p_3 + 1/p_4$, and $p_2,p_3 \in (1,\infty)$, for a sufficiently large constant $C$ that depends only on $s,p,p_i$. Moreover,
\begin{align}
\| \Lambda^s (f  g) - f  \Lambda^s g \|_{L^p} 
\leq C  \|\nabla f \|_{L^{p_1}} \|\Lambda^{s-1} g\|_{L^{p_2}} + C \|\Lambda^s f\|_{L^{p_3}} \|g\|_{L^{p_4}}  \label{eq:Hs:commutator}
\end{align}
where $p_i$ are as above. 
For $q \in [p,\infty)$  and $f$ of zero mean we also have
\begin{align}
\| f\|_{L^q} \leq C \|\Lambda^{\frac{2}{p} -\frac{2}{q}} f\|_{L^p} \label{eq:Hs:Sobolev}
\end{align}
for a sufficiently large constant $C$ that depends only on $p$ and $q$.
\end{lemma}

We conclude this appendix by giving a sketch of the proof of Proposition~\ref{prop:Poincare}.
The detailed proof can be found in~\cite{ConstantinGlattHoltzVicol13}, and we give below only the main ideas.

\begin{proof}[Proof of Proposition~\ref{prop:Poincare}]
The case $\alpha=0$ trivially holds, while in the case $\alpha=2$ estimate \eqref{eq:Lp:Poincare} follows upon integration by parts. Therefore, henceforth consider $\alpha \in (0,2)$. For $p=2$ inequality \eqref{eq:Lp:Poincare} holds due to the Parseval's identity, and the rest of the proof we let $p \geq 4$ be even. 
For $0<\alpha<2$ we have 
\begin{align}
& \int \theta^{p-1}(x) \Lambda^\alpha \theta(x) dx \notag\\
&= \frac 12 P.V. \intint \left(\theta^{p-1}(x) -\theta^{p-1}(y) \right) \left( \theta(x) - \theta(y) \right) K_\alpha(x-y) dy dx \notag\\
&= \frac{1}{2p} P.V. \intint \left( p\left(\theta^{p-1}(x) -\theta^{p-1}(y) \right) \left( \theta(x) - \theta(y) \right)  - 2 \left( \theta^{p/2}(x) -\theta^{p/2}(y) \right)^2  \right) K_\alpha(x-y) dy dx \notag\\
&\qquad + \frac 1p P.V. \intint \left( \theta^{p/2}(x) -\theta^{p/2}(y) \right)^2   K_\alpha(x-y) dy dx \notag\\
&= \frac{1}{2p} P.V. \intint f_p(\theta(x),\theta(y)) K_\alpha(x-y) dy dx + \frac{1}{p} \| \Lambda^{\alpha/2}  ( \theta^{p/2}) \|_{L^2}^2  \notag\\
&=: \frac{1}{2p} {\mathcal T} +  \frac{1}{p} \| \Lambda^{\alpha/2}  ( \theta^{p/2}) \|_{L^2}^2 
\label{eq:identity}
\end{align}
where the double integral is over ${\mathbb{T}}^{2d}$, and we have defined
\begin{align*}
f_p(a,b) = p(a^{p-1} - b^{p-1}) (a-b) -2 (a^{p/2} - b^{p/2})^2.
\end{align*}
it can be easily seen that $f_p(a,b) \geq 0$ on $\RR^2$ when $p$ is even, and so the term ${\mathcal T}$ is positive. The {main idea} is that exactly ${\mathcal T}$ gives the $\|\theta\|_{L^p}^p$ term in the lower bound \eqref{eq:Lp:Poincare}.

We next claim that for $p\geq 4$ even, and $a,b \in \RR$ we have
\begin{align}
f_p(a,b) \geq \frac{p-2}{2} (a-b)^2 a^{p-2}. \label{eq:trick}
\end{align}
This fact may be checked directly using calculus. 
Using \eqref{eq:trick} we now prove \eqref{eq:Lp:Poincare}. Since $K_\alpha$ is positive, letting $e_1 = (1,0, \ldots,0)$, we have
\begin{align}
{\mathcal T} &\geq \frac{p-2}{2} P.V. \intint (\theta(x) - \theta(y))^2 \theta(x)^{p-2} K_\alpha(x-y) dy dx \notag\\
&\geq \frac{p-2}{2}   c_{d,\alpha} \intint (\theta(x) - \theta(y))^2 \theta(x)^{p-2} \frac{1}{|x-y-2\pi e_1|^{d+\alpha}} dy dx\notag\\
&\geq \frac{(p-2) c_{d,\alpha} }{2 (2 \pi + |{\rm diam}({\mathbb{T}}^d)|)^{d+\alpha}} \intint (\theta(x) - \theta(y))^2 \theta(x)^{p-2} dy dx \notag\\
&=  \frac{(p-2) c_{d,\alpha} }{2(2 \pi + |{\rm diam}({\mathbb{T}}^d)|)^{d+\alpha}} \intint \left( \theta^p(x) - 2 \theta^{p-1}(x) \theta(y)+ \theta^{p-2}(x) \theta^2(y) \right) dy dx\notag\\
&\geq  \frac{(p-2) c_{d,\alpha} }{2(2 \pi + |{\rm diam}({\mathbb{T}}^d)|)^{d+\alpha}} \int_{{\mathbb{T}}^d} \left( \theta^p(x) |{\mathbb{T}}^d| - 2 \theta^{p-1}(x)  \int_{{\mathbb{T}}^d} \theta(y) dy \right)  dx.
\label{eq:lower}
\end{align}
At this point we use that $\theta$ has zero mean. It then follows from \eqref{eq:lower},  that 
\begin{align}
{\mathcal T} \geq \frac{(p-2) c_{d,\alpha}  |{\mathbb{T}}^d|}{2(2 \pi + |{\rm diam}({\mathbb{T}}^d)|)^{d+\alpha}}  \| \theta\|_{L^p}^p.
\label{eq:Lp:Poincare:constant}
\end{align}
This proves \eqref{eq:Lp:Poincare} with the constant 
\begin{align*}
\frac{(p-2) 2^\alpha \Gamma( (n+\alpha)/2) |{\mathbb{T}}^d|}{4 p (2 \pi + |{\rm diam}({\mathbb{T}}^d)|)^{d+\alpha}|\Gamma(-\alpha/2)| \pi^{d/2}} 
\geq \frac{2^\alpha \Gamma( (n+\alpha)/2) |{\mathbb{T}}^d|}{8 (2 \pi + |{\rm diam}({\mathbb{T}}^d)|)^{d+\alpha}|\Gamma(-\alpha/2)| \pi^{d/2}} 
= \frac{1}{C_{d,\alpha}}
\end{align*}
for any $p\geq 4$. When $d=2$ and $\alpha =1$ the above constant $C_{d,\alpha}$ may be taken to equal $2^9 \pi^2$.
\end{proof}

\section{Interchanging the spatial supremum with the time derivative}
\label{sec:sup:dt}

\begin{lemma}[\bf Switching $d/dt$ and $\sup$]
\label{lem:switch:dt:sup}
Let $\KK \subset \RR^d$ be compact, and let $T>0$.  Consider a function
\[
f \colon (0,T) \times \KK \to [0,\infty)
\]
and assume that for every $\lambda \in \KK$ the functions
\[
f_\lambda(\cdot) = f(\cdot,\lambda) \colon (0,T) \to [0,\infty) 
\qquad \mbox{and} \qquad 
\dot{f}_\lambda(\cdot) = (\partial_t f)(\cdot,\lambda) \colon (0,T) \to \RR
\]
are continuous.
Additionally, assume that the following properties hold:
\begin{itemize}
 \item[(i)] The families $\{f_\lambda\}_{\lambda \in \KK}$ and $\{\dot{f}_\lambda\}_{\lambda \in \KK}$ are uniformly equicontinuous with respect to $t$.
 \item[(ii)] For every  $t\in (0,T)$, the functions $f(t,\cdot) \colon \KK \to [0,\infty)$ and $(\partial_t f)(t,\cdot) \colon \KK \to \RR$ are continuous.
\end{itemize}
Lastly, define
\begin{align*}
F(t) = \sup_{\lambda \in \KK} f_\lambda(t)
\end{align*}
Then, for almost every $t\in (0,T)$ the function $F$ is differentiable at $t$, and there exists $ \lambda_* =  \lambda_*(t) \in \KK$ such that simultaneously   
\begin{align}
\dot F(t) = \dot{f}_{\lambda_*} (t) \quad \mbox{and} \quad F(t) = f_{\lambda_*}(t)
\label{eq:F:at:max}
\end{align}
hold.
\end{lemma}

\begin{proof}[Proof of Lemma~\ref{lem:switch:dt:sup}] The proof follows along the lines of  \cite[Theorem 4.1]{CordobaCordoba04} and \cite[Lemma A.3]{KukavicaVicol11b}, but for the sake of completeness we present here the full argument.

The uniform equicontinuity of $\{\dot f_\lambda(t)\}_{\lambda \in \KK}$ implies that there exists $\delta>0$ such that 
\[
\sup_{\lambda \in \KK} |\dot f_\lambda (t) - \dot f_\lambda(s)| \leq 1 \quad \mbox{whenever} \quad |t-s|<\delta.
\]
Since $(\partial_t f)(T/2,\lambda)$ is a continuous function of $\lambda$, it attains its maximum  over the compact $\KK$ at some $\lambda_0$, and we obtain from the above that 
\[
\sup_{t \in (0,T)} \sup_{\lambda \in \KK} |\dot f_\lambda(t)| \leq |\dot f_{\lambda_0}(T/2)| + \frac{T}{\delta} =: M < \infty
\]
Therefore, for $t \in (0,T)$, and $ |\Delta t| >0 $ sufficiently small so that $t + \Delta t \in (0,T)$,  we have
\begin{align*}
|F(t) - F(t+\Delta t)| 
= \left| \sup_{\lambda \in \KK} f_\lambda(t) - \sup_{\lambda \in \KK} f_\lambda(t+\Delta t) \right|
&\leq \sup_{\lambda \in \KK} |f_\lambda(t) - f_\lambda(t+\Delta t)| \notag\\
&= |\Delta t| \sup_{\lambda \in \KK} |\dot{f}_{\lambda}(\tau_\lambda)| \qquad \mbox{(for some }\tau_\lambda \in (t,t+\Delta t)\mbox{)}\notag\\
&\leq |\Delta t| M.
\end{align*}
Therefore, $F$ is Lipschitz continuous on $(0,T)$, and Rademacher's theorem implies that 
$F$ is differentiable almost everywhere.

Fix $t \in (0,T)$ such that $F$ is differentiable at $t$, and let $\Delta t >0$. In view of the continuity in $\lambda$ of $f(t+\Delta t,\lambda)$ and the compactness of $\KK$, there exists $\lambda(t+\Delta t) \in \KK$ such that 
\begin{align*}
F(t + \Delta t) = f_{\lambda(t+\Delta t)}(t).
\end{align*}
Since $\KK$ is compact, we can find a sequence $\Delta t_n \to 0^+$ and a point $\lambda_* \in \KK$ such that 
\begin{align}
\lambda(t+\Delta t_n) \to \lambda_* \quad \mbox{as} \quad \Delta t_n \to 0^+.
\label{eq:lambda*}
\end{align}
We now check that 
\[ 
F(t) = f_{\lambda_*}(t),\]
i.e., the second statement in \eqref{eq:F:at:max}.
We write
\begin{align}
|F(t)  - f_{\lambda_*}(t) |
&\leq \lim_{\Delta t_n \to 0^+}  \left| F(t) - F(t+ \Delta t_n) \right| \notag\\
&\quad + \lim_{\Delta t_n \to 0^+}  \left|  f_{\lambda(t+\Delta t_n)}(t+\Delta t_n) - f_{\lambda(t+\Delta t_n)}(t) \right| \notag\\
&\quad + \lim_{\Delta t_n \to 0^+}   \left|  f_{\lambda(t+\Delta t_n)}(t) - f_{\lambda_*}(t) \right|.
\label{eq:F:lsc}
\end{align}
Note that 
\begin{align*}
|F(t) - F(t+\Delta t_n)| 
   \leq \sup_{\lambda \in \KK} |f_\lambda(t) - f_\lambda(t+\Delta t_n)|
\end{align*}
and 
\begin{align*}
\left| f_{\lambda(t+\Delta t_n)}(t + \Delta t_n) -  f_{\lambda(t+\Delta t_n)}(t) \right| \leq \sup_{\lambda \in \KK} |f_\lambda(t) - f_\lambda(t+\Delta t_n)|.
\end{align*}
In view of the equicontinuity at $t$ of the family $\{f_\lambda\}_{\lambda \in \KK}$, we have
\begin{align*}
\lim_{\Delta t_n \to 0^+} \sup_{\lambda \in \KK} |f_\lambda(t) - f_\lambda(t+\Delta t_n)| = 0
\end{align*}
and thus the first two limits on the right side of \eqref{eq:F:lsc} vanish. The third limit vanishes in view of \eqref{eq:lambda*} and the assumption of continuity with respect to $\lambda$ of $f(t,\lambda)$, at any given fixed $t$. This proves the second part of \eqref{eq:F:at:max}.

Let $\eps>0$ and fix $t \in(0,T)$. In view of the uniform equicontinuity of the family $\{\dot f_\lambda\}_{\lambda \in \KK}$, we have that there exists $\delta_1 = \delta_1(\eps)>0$, such that 
\begin{align}
\sup_{\lambda \in \KK} |\dot{f}_\lambda(t) - \dot{f}_\lambda(\tau)| < \eps \quad \mbox{whenever} \quad |t-\tau| < \delta_1.
\label{eq:delta:1}
\end{align}
Also, in view of the continuity with respect to $\lambda$ of $(\partial_t f)(t,\lambda)$ at a fixed $t$, there exists $\delta_2 =\delta_2(\eps,t) >0$ such that 
\begin{align}
|\dot{f}_{\lambda} (t) - \dot{f}_{\lambda_*}(t) | < \eps \quad \mbox{whenever} \quad |\lambda - \lambda_*| < \delta_2.
\label{eq:delta:2}
\end{align}
Let $n$ be sufficiently large, so that $0 < \Delta t_n < \delta_1$ and $|\lambda(t+\Delta t_n) - \lambda_*| < \delta_2$, which is possible in view of \eqref{eq:lambda*}. Using the fundamental theorem of calculus, \eqref{eq:delta:1}, \eqref{eq:delta:2}, and the fact that $\Delta t_n > 0$, we obtain\begin{align*}
&\frac{F(t+\Delta t_n) - F(t)}{\Delta t_n} \notag\\
&= \frac{f_{\lambda(t+\Delta t_n)}(t+\Delta t_n) - f_{\lambda_*}(t)}{\Delta t_n} \notag\\
&= \frac{f_{\lambda(t+\Delta t_n)}(t+\Delta t_n) - f_{\lambda(t+\Delta t_n)}(t)}{\Delta t_n} + \frac{f_{\lambda(t+\Delta t_n)}(t) - f_{\lambda_*}(t)}{\Delta t_n} \notag\\
&\leq \frac{1}{\Delta t_n} \int_t^{t+\Delta t_n} \dot{f}_{\lambda(t+\Delta t_n)}(s) ds  \notag\\
&= \frac{1}{\Delta t_n} \int_t^{t+\Delta t_n} \dot{f}_{\lambda_*}(s) ds +  \frac{1}{\Delta t_n} \int_t^{t+\Delta t_n} \left( \dot{f}_{\lambda(t+\Delta t_n)}(s) - \dot{f}_{\lambda_*}(s) \right) ds \notag\\
&= \frac{1}{\Delta t_n} \int_t^{t+\Delta t_n} \dot{f}_{\lambda_*}(t) + \left( \dot{f}_{\lambda_*}(s) - \dot{f}_{\lambda_*}(t) \right) ds \notag\\
&\qquad + \frac{1}{\Delta t_n} \int_t^{t+\Delta t_n} \left( \dot{f}_{\lambda(t+\Delta t_n)}(s) - \dot{f}_{\lambda(t+\Delta t_n)}(t)  \right) + \left(\dot{f}_{\lambda_*}(t)- \dot{f}_{\lambda_*}(s) \right) + \left( \dot{f}_{\lambda(t+\Delta t_n)}(t) - \dot{f}_{\lambda_*}(t) \right) ds  \notag\\
&\leq \dot{f}_{\lambda_*}(t) + 4 \eps,
\end{align*}
which shows that 
\begin{align}
\dot F(t) = \lim_{\Delta t_n \to 0+} \frac{F(t+\Delta t_n) - F(t)}{\Delta t_n} \leq \dot{f}_{\lambda_*}(t) \label{eq:F:d+}
\end{align}
since $\eps$ was arbitrary, and we chose $t$ so that $\dot F(t)$ exists.
Conversely, for $\Delta t_n >0$ we have
\begin{align*}
\frac{F(t+\Delta t_n) - F(t)}{\Delta t_n} &= \frac{f_{\lambda(t+\Delta t_n)}(t+\Delta t_n) - f_{\lambda_*}(t)}{\Delta t_n} \notag\\
&= \frac{f_{\lambda(t+\Delta t_n)}(t+\Delta t_n) - f_{\lambda_*}(t+\Delta t_n)}{\Delta t_n} + \frac{f_{\lambda_*}(t+\Delta t_n) - f_{\lambda_*}(t)}{\Delta t_n} \notag\\
&\geq  \frac{f_{\lambda_*}(t+\Delta t_n) - f_{\lambda_*}(t)}{\Delta t_n} 
\end{align*}
which shows that
\begin{align}
\dot F(t) = \lim_{\Delta t_n \to 0+} \frac{F(t+\Delta t_n) - F(t)}{\Delta t_n} \geq \dot{f}_{\lambda_*}(t)\label{eq:F:d-}
\end{align}
Estimates \eqref{eq:F:d+} and \eqref{eq:F:d-} prove the first part of \eqref{eq:F:at:max}, and hence of the lemma.
\end{proof}

\begin{corollary}
\label{cor:switch:dt:sup}
Assume that $\theta \in C^\beta((0,T);C^{1,\beta}(\TT^2))$ is a classical solution of \eqref{eq:SQG:1}--\eqref{eq:SQG:2} on $(0,T)$, for some $\beta \in (0,1)$,  with force $f \in C^\beta$. For $ 0 < \alpha <  \beta/2 $ we define 
 \begin{align*}
v(t,x;h) = \frac{|\delta_h\theta(x,t)|}{|h|^\alpha}  = \frac{|\theta(x+h,t) - \theta(x,t)|}{|h|^\alpha} \colon \TT^2 \times (0,T) \times \TT^2 \to \RR
\end{align*}
with the convention that $v(t,x;0) = 0$.
Then for almost every $t \in (0,T)$, there exists a pair 
\[
(\bar x,\bar h) = (\bar x(t),\bar h(t)) \in \TT^2 \times \TT^2
\] 
such that 
\begin{align*}
v(t,\bar x;\bar h)^2 = \sup_{(x,h) \in \TT^2 \times \TT^2} v(t,\cdot;\cdot)^2
\end{align*}
and moreover 
\begin{align*}
\frac{d}{dt} \left( \sup_{(x,h) \in \TT^2 \times \TT^2} v(t,\cdot;\cdot)^2 \right) = (\partial_t v^2)(t,\bar x;\bar h)
\end{align*}
holds.
\end{corollary}
\begin{proof}[Proof of Corollary~\ref{cor:switch:dt:sup}]
The proof follows by applying Lemma~\ref{lem:switch:dt:sup} to 
the function
\[
f(t,\lambda) = v(t,x;h)^2
\]
with $\lambda = (x;h) \in \TT^2 \times \TT^2 = \KK$, which is clearly compact. It is clear that $f$ is a non-negative function.

By assumption, for fixed $\lambda = (x;h)$, the function 
\[ 
f_\lambda(t) =f(\lambda ,t) = v(t,x;h)^2 = \frac{(\theta(x+h,t) - \theta(x,t))^2}{|h|^{2\alpha}}
\]
is continuous with respect to $t$. In fact, since $2\alpha < \beta <1$, for $t,s \in [0,T]$ and $\lambda = (x;h)$ we have that 
\begin{align*} 
| f_\lambda(t) - f_\lambda(s) | 
&= \frac{1}{|h|^{2\alpha}} \left| \delta_h \theta(x,t) + \delta_h \theta(x,s) \right| \left| \delta_h \theta(x,t) - \delta_h \theta(x,s) \right| \notag\\
&\leq \left( [\theta(t)]_{C^{2\alpha}} + [\theta(s)]_{C^{2\alpha}} \right) \left( | \theta(x+h,t) - \theta(x+h,s) | + |\theta(x,t) - \theta(x,s)| \right) \notag\\
&\leq 4 \|\theta\|_{L^\infty(0,T;C^\beta)} |t-s|^\beta \|\theta\|_{C^\beta(0,T;L^\infty)} \notag\\
&\leq 4 |t-s|^{\beta} \|\theta\|_{C^\beta(0,T;C^{\beta})}^2
\end{align*}
which shows that the family $\{f_\lambda\}_{\lambda \in \KK}$ is uniformly equicontinuous on $(0,T)$.

Using the equation \eqref{eq:SQG:1} obeyed by $\theta$, we moreover have that 
\begin{align*}
\dot f_\lambda(t) &= (\partial_t f)(t,\lambda)  = \frac{\partial}{\partial t} v(t,x;h)^2 \notag\\
&= \frac{2 (\theta(x+h,t)-\theta(x,t)) }{|h|^{2\alpha}} (\partial_t \theta(x+h,t) - \partial_t \theta(x,t)) \notag\\
&=\frac{2 \delta_h \theta(x,t) }{|h|^{2\alpha}} \left( \delta_h f(x) - \kappa \delta_h \Lambda \theta(x,t) - u(x,t)  \cdot \delta_h \nabla \theta(x,t) -  \delta_h u(x,t) \cdot \theta(x+h,t) \right)
\end{align*}
is also continuous with respect to $t$, since by assumption $\nabla \theta, \Lambda \theta \in C^\beta(0,T;C^\beta)$, and $u \in C^\beta(0,T;C^{1,\beta})$. 
To verify the equicontinuity of the family $\{\dot f_\lambda\}_{\lambda \in \KK}$, we note that for $t,s\in [0,T]$ and $(x;h) \in \TT^2 \times \TT^2$ it holds that
\begin{align*} 
&| \dot f_\lambda(t) - \dot f_\lambda(s) | \notag\\
&\qquad \leq C |t-s|^\beta \|\theta\|_{C^\beta(0,T;L^\infty)} \left( \|f\|_{C^{2\alpha}} + \kappa \|\theta\|_{L^\infty(0,T;C^{1,2\alpha})} + \|\theta\|_{L^\infty(0,T;C^{2\alpha})} \|\theta\|_{L^\infty(0,T;C^{1,2\alpha})} \right) \notag\\
&\qquad\qquad  + C |t-s|^\beta \|\theta\|_{L^\infty(0,T;C^{2\alpha})} \left( \kappa \|\theta\|_{C^\beta(0,T;C^{1,\beta})} + \|\theta\|_{C^\beta(0,T;C^{1,\beta})}^2 \right) \notag\\
&\qquad\leq C |t-s|^\beta \left( \|f\|_{C^{\beta}} + \kappa \|\theta\|_{C^\beta(0,T;C^{1,\beta})}  + \|\theta\|_{C^\beta(0,T;C^{1,\beta})}^2 \right)
\end{align*}
which by assumption is a finite number times $|t-s|^\beta$.

It is left to check that for fixed $t \in (0,T)$, the quantities $f(t,\lambda)$ and $\partial_t f(t,\lambda)$ vary continuously with respect to $\lambda = (x;h)$. This can be verified similarly to the equicontinuity of $f_\lambda$ and $\dot f_\lambda$. Note that there is no problem at $h=0$ since by assumption $\beta> 2\alpha$. We omit further details.
\end{proof}

\section{Technical details about the existence and size of the attractor}
\label{app:attractor:existence}

In this appendix, we present the present a number of technical lemmas which are needed in order to establish the existence of the global attractor for the solution map $S(t) \colon H^1 \to H^1$ associated to the critical SQG equation, and to give an estimate on its fractal dimension.

The following variant of the classical Gr\"onwall lemma is used in the proof of Theorem~\ref{thm:H:3/2:absorbing}, in order to bootstrap information about the time average of the $H^{3/2}$ norm, to information about the pointwise in time behavior of the $H^{3/2}$ norm. The lemma is due to~Foias and Prodi~\cite{FoiasProdi67}. See also~\cite{ConstantinFoias88,Temam97,Robinson01}.

\begin{lemma}[\bf Uniform Gr\"onwall Lemma]
\label{lem:Gronwall}
Assume $x,a,b \colon [0,\infty) \to [0,\infty)$  are functions such that 
\begin{align*}
\frac{dx}{dt} \leq a x + b
\end{align*}
and in addition assume that there exists $r>0, t_0 >0$ such that 
\begin{align*}
\int_{t}^{t+r} x(s) ds \leq X, \quad \int_{t}^{t+r} a(s) ds \leq A, \quad \int_{t}^{t+r} b(s) ds \leq B
\end{align*}
for all $t \geq t_0$. 
Then we have
\begin{align}
x(t) \leq (X r^{-1} + B) e^{A} \label{eq:unif:Gronwall}
\end{align}
for all $t \geq t_0 + r$.
\end{lemma}

Next, we give the proof of the backwards uniqueness property for $S(t)$. The proof uses the classical log-convexity method of Agmon and Nirenberg~\cite{AgmonNirenberg67}, see also~\cite{Temam97,Robinson01,Kukavica07}.

\begin{proof}[Proof of Proposition~\ref{prop:backwards}]
 Let $\theta(t) = \theta^{(1)} - \theta^{(2)}$  and $\bar \theta (t) = ( \theta^{(1)}(t) + \theta^{(2)}(t))/2$ be the difference, respectively the average of the two solutions. The equation obeyed by $\theta$ is 
\begin{align*} 
\partial_t \theta + \kappa \Lambda \theta + \bar u \cdot \nabla \theta + u \cdot \nabla \bar \theta = 0, \quad \theta_0 = \theta^{(1)}_0 - \theta^{(2)}_0.
\end{align*}
By contradiction, assume that $\theta_0 \neq 0$. Then, by continuity in time, we have that $\|\theta(t)\|_{L^2} > 0$ for sufficiently small $t$, and let $\tau \in (0,T]$ be defined as the minimal time such that $\lim_{t \to \tau^{-}} \|\theta(t)\|_{L^2} = 0$. The inequality $\tau \leq T$ follows by assumption. By continuity in time, we can define $m = \max_{t\in [0,\tau]} \|\theta(t)\|_{L^2}$. Then, by the minimality of $\tau$, the function 
\[
w(t) = \log \frac{2 m}{\|\theta(t)\|_{L^2}}
\] 
is well-defined and positive on $[0,\tau)$, with $w(0) < \infty$. We compute
\begin{align*} 
\frac{d}{dt} w = - \frac{1}{\|\theta\|_{L^2}^2} \int\theta \partial_t \theta  dx 
&\leq  \frac{1}{\|\theta\|_{L^2}^2} \left(- \kappa \|\Lambda^{1/2} \theta\|_{L^2}^2 + \|u\|_{L^2} \|\nabla \bar \theta\|_{L^4} \|\theta\|_{L^4} \right) \notag\\
&\leq  \frac{1}{\|\theta\|_{L^2}^2} \left(- \kappa \|\Lambda^{1/2} \theta\|_{L^2}^2 + C \|\theta\|_{L^2} \|\Lambda^{1/2}\theta\|_{L^2} \|\bar \theta\|_{H^{3/2}} \right)
\end{align*}
by using the Sobolev embedding $H^{1/2} \subset L^4$. Since by assumption $\theta^{(i)} \in L^2(0,T;H^{3/2})$, we obtain from the Cauchy-Schwartz inequality and integrating in time that 
\begin{align*} 
w(t) \leq w(0) + C \int_0^\tau \|\bar \theta(s) \|_{H^{3/2}}^2 ds < \infty
\end{align*}
for all $t \in [0,\tau)$. This contradicts the assumption that as $t \to \tau^{-}$ we have $\|\theta(t)\|_{L^2} \to 0$, which is equivalent to $w(t) \to \infty$.
\end{proof}

We now give the proof of the continuity property of the solution map with respect to time and with respect to perturbations in the initial data, in the $H^1$ topology.

\begin{proof}[Proof of Proposition~\ref{prop:continuity}]
The continuity in time for fixed initial data was already given by Proposition~\ref{prop:local}, so it remains to check continuity with respect to the initial data which originates from $\BB_0$. 

Due to Theorem~\ref{thm:H:3/2:absorbing}, we know there exists an absorbing ball $\BB \subset H^{3/2}$ for the dynamics induced by $S(t)$ on $H^1$. In particular, by Remark~\ref{rem:uniform} there exist a time $t_{H^{3/2}}(\BB_0)$ such that $S(t) \BB_0 \subset \BB$ for all $t\geq t_{H^{3/2}}(\BB_0)$. Since the sequence $t_n$ in the statement of the proposition diverges as $n\to \infty$, we may assume without loss of generality that $\tilde \theta_{0,n}= S(t_n) \theta_{0,n} \in \BB$ for all $n\geq 1$, and even that $S(t) \tilde \theta_{0,n} \in \BB$ for all $t\geq 0, n\geq 1$.

Let $\theta_0 \in H^1$ be arbitrary. Fix some $ \tilde \theta_0 \in \BB$ such that $\tilde \theta(t) = S(t) \tilde \theta_0 \in \BB$ for all $t\geq 0$, and such that 
\begin{align*}
\|\theta_0 - \tilde \theta_0\|_{H^1} \leq \eps \kappa
\end{align*}  
for $\eps>0$ to be determined later.
Denote $\bar \theta(t)  = \tilde \theta(t) - \theta(t) = S(t) \tilde \theta_0 - S(t) \theta_0$, for all $t\geq 0$, and let $\bar u$ be the corresponding velocity difference. The equation obeyed by $\bar \theta$ is 
\begin{align*}
\partial_t \bar \theta + \kappa \Lambda \bar \theta - \bar u \cdot \nabla \bar \theta+  \tilde u \cdot \nabla \bar \theta + \bar u \cdot \nabla \tilde \theta = 0.
\end{align*}
Multiplying the above by $-\Delta \theta$, integrating over $\TT^2$, and integrating by parts, yields
\begin{align}
 \frac{1}{2} \| \bar \theta\|_{H^1}^2 + \kappa \| \bar \theta\|_{H^{3/2}}^2  
&\leq \left| \int (\nabla  \bar u \cdot \nabla \bar \theta ) \cdot \nabla \bar \theta \right|  + \left| \int (\nabla  \tilde u \cdot \nabla \bar \theta ) \cdot \nabla \bar \theta \right| + \left| \int \Lambda^{1/2} (\bar u \cdot \nabla \tilde \theta ) \Lambda^{3/2} \bar \theta \right| \notag\\
&\leq C(  \|\nabla  \bar u\|_{L^4} +  \|\nabla  \tilde u\|_{L^4}) \| \bar \theta\|_{H^1} \|\nabla \bar \theta\|_{L^4} + C  \| \Lambda^{1/2}( \bar u \cdot \nabla \tilde \theta)\|_{L^2} \| \bar \theta\|_{H^{3/2}} \notag\\
&\leq C ( \|\tilde  \theta\|_{H^{3/2}} +   \|\bar  \theta\|_{H^{3/2}} )  \| \bar \theta\|_{H^1} \|\bar \theta\|_{H^{3/2}} + C \|\bar u\|_{L^\infty} \| \tilde \theta \|_{H^{3/2}} \| \bar \theta\|_{H^{3/2}} \label{eq:cont:data:1}
\end{align}
Here we have also appealed to the fractional Sobolev embedding $H^{1/2} \subset L^4$, and the product estimate from Lemma~\ref{lem:fractional:calculus}. We may appeal to Brezis-Gallou\"et the inequality
\begin{align}
\|\bar u\|_{L^\infty} = \| \RSZ^\perp \bar \theta \|_{L^\infty} \leq C \|\bar \theta\|_{H^1}\left( 1+ \log \frac{\|\bar \theta\|_{H^{3/2}}^2}{ \|\bar \theta\|_{H^{1}}^2} \right)^{1/2} \label{eq:Brezis:Gallouet}
\end{align}
which combined with the inequality 
\begin{align}
a \mu \left(1+\log \frac{\mu^2}{b^2} \right)^{1/2} \leq \eps \mu^2 + \frac{a^2}{\eps} \log{\frac{2a}{\eps b}}
\label{eq:FMTK}
\end{align}
which holds for any $a,\eps>0$ and $\mu\geq b$ (see~\cite{FoiasManleyTemam88,Kukavica96}), and the estimate \eqref{eq:cont:data:1}, yields
\begin{align}
&\frac{d}{dt} \|\bar \theta\|_{H^1}^2 + \frac{\kappa}{2} \|\bar \theta\|_{H^{3/2}}^2 
\leq C_0 \|\bar \theta\|_{H^1} \|\bar \theta \|_{H^{3/2}}^2 + \frac{C_0}{\kappa} \|\bar \theta\|_{H^1}^2 \|\tilde \theta \|_{H^{3/2}}^2   \left( 1+  \log \frac{C_0 \| \tilde \theta\|_{H^{3/2}}}{\kappa} \right)
\label{eq:cont:data:2}
\end{align}
for some universal constant $C_0>0$.
Note that the initial data $\bar \theta_0$  obeys
\begin{align*}
\| \bar \theta_0\|_{H^1} \leq \eps \kappa
\end{align*}
where  $\eps$ is chosen so that  $4 \eps C_0 \leq 1$.  Due to continuity in time, we therefore conclude that there exists $T > 0$ such that on $[0,T]$ we have $ \|\bar \theta(t) \|_{H^1} \leq 2 \eps \kappa$, and on this time interval from \eqref{eq:cont:data:2} we conclude that 
\begin{align}
\frac{d}{dt} \|\bar \theta\|_{H^1}^2  \leq \|\bar \theta\|_{H^1}^2 \frac{C_0 M^2}{\kappa}     \left( 1+  \log \frac{C_0 M}{\kappa} \right)
\label{eq:cont:data:3}.
\end{align}
Here we used the assumption that $S(t) \tilde \theta_0 \in \BB$ for all $t \geq 0$ and denoted by $M$ the radius of the ball $\BB$. Thus, a posteriori we conclude that we could have chosen
\begin{align*}
T = \frac{ \kappa \log  2}{ C_0 M^2      \left( 1+  \log \frac{C_0 M}{\kappa} \right)} = T(\kappa, \BB)
\end{align*}
so that 
\begin{align*}
\|\bar \theta(t) \|_{H^1} \leq 2 \eps \kappa
\end{align*}
for $t\in [0,T]$. It is important that this $T$ is independent of $\eps$.

The proof of the proposition may now be concluded. The above estimates shows that as $\|\tilde \theta_{0,n} - \theta_0\|_{H^1} \to 0$, we have
\begin{align*}
\|S(t) \tilde \theta_{0,n} - S(t) \theta_0\|_{H^1}^2 \to 0 \quad \mbox{as} \quad n \to \infty
\end{align*}
for all $t \in [0,T(\kappa,\BB)]$. We then re-iterate this argument for $t \in [i T(\kappa,\BB), (i+1) T(\kappa,\BB)]$ for all $i \geq 1$, which proves the Proposition.
\end{proof}

We conclude the appendix by giving the proof of continuous differentiability of the solution solution map around trajectories on the global attractor.

\begin{proof}[Proof of Proposition~\ref{prop:unif:diff}]
For $\theta_0, \phi_0 \in \AA$, denote $\xi_0 = \phi_0 - \theta_0 $ and define $\xi(t) = S'(t,\theta_0) [\xi_0]$ via \eqref{eq:L:xi:def}. We let
\begin{align*}
\eta(t) = \phi(t) - \theta(t) - \xi(t) = S(t) \phi_0 - S(t) \theta_0 - S'(t,\theta_0)[\xi_0]
\end{align*}
and observe that $\eta$ obeys the equation
\begin{align}
\partial_t \eta + \kappa \Lambda \eta + \RSZ^\perp \eta \cdot \nabla \theta + \RSZ^\perp \theta \cdot \nabla \eta = - \RSZ^\perp w  \cdot \nabla w, \qquad \eta(0) = 0.
\label{eq:eta:PDE}
\end{align}
where 
\begin{align}
w(t) = \phi(t) - \theta(t) = S(t) \phi_0 - S(t) \theta_0 = S(t) \xi_0. \label{eq:w:t:def}
\end{align}
In order to estimate $\|\eta(t)\|_{H^1}$, take an $L^2$ inner product of \eqref{eq:eta:PDE} with $-\Delta \eta$ and use Lemma~\ref{lem:fractional:calculus} to obtain
\begin{align}
\frac 12 &\frac{d}{dt} \|\eta\|_{H^1}^2 + \kappa \|\eta\|_{H^{3/2}}^2 \notag\\
&=  \int_{\TT^2} \RSZ^\perp \eta \cdot \nabla \theta \Delta \eta dx - \int_{\TT^2} \partial_k \RSZ^\perp \theta \cdot \nabla \eta \partial_k \eta dx + \int_{\TT^2} \RSZ^\perp w \cdot \nabla w \Delta \eta dx \notag\\
&\leq \|\Lambda^{3/2}\eta\|_{L^2} \|\Lambda^{1/2}(\RSZ^\perp \eta \cdot \nabla \theta) \|_{L^2} + \|\nabla \RSZ^\perp \theta\|_{L^4} \|\nabla \eta\|_{L^4} \| \nabla \eta \|_{L^2} +   \|\Lambda^{3/2}\eta\|_{L^2} \|\Lambda^{1/2}(\RSZ^\perp w \cdot \nabla w) \|_{L^2}\notag\\
&\leq C  \|\eta\|_{H^{3/2}}  \left( \|\eta\|_{H^1} \|\theta\|_{H^{3/2}} + \|\eta\|_{H^{1/2}} \|\theta\|_{H^2} \right) +  C \|\eta\|_{H^{3/2}}  \left( \|w\|_{H^{1}} \|w\|_{H^{3/2}} + \|\RSZ^\perp w\|_{L^\infty} \|w\|_{H^{3/2}} \right)\notag\\
&\leq \frac{\kappa}{2} \|\eta\|_{H^{3/2}}^2 + \frac{C}{\kappa} \|\eta\|_{H^1}^2 \|\theta\|_{H^2}^2 + \frac{C}{\kappa} \|w\|_{H^1}^2 \|w\|_{H^{3/2}}^2 \left(1 + \log\frac{\|w\|_{H^{3/2}}^2}{\|w\|_{H^1}^2} \right)
\label{eq:eta:ODE:1}
\end{align}
for some universal constant $C>0$.
In the last inequality we have also appealed to \eqref{eq:Brezis:Gallouet}.

Next we estimate $w$, as defined in \eqref{eq:w:t:def}. It obeys the equation
\begin{align}
\partial_t w + \kappa \Lambda w  + \RSZ^\perp \phi \cdot \nabla w + \RSZ^\perp w \cdot \nabla \theta = 0, \quad w(0) = \xi_0.
\label{eq:w:PDE}
\end{align}
We note that in view of Theorem~\ref{thm:attractor:existence}, we a priori have estimates on the $H^{3/2}$ and even $H^2$ norms of $\theta$ and $\phi$, since they are elements of $\AA$. Multiplying \eqref{eq:w:PDE} with $-\Delta w$ and integrating, similarly to \eqref{eq:eta:ODE:1} we obtain
\begin{align}
&\frac 12 \frac{d}{dt} \|w\|_{H^1}^2 + \kappa \|w\|_{H^{3/2}}^2 \notag\\
&\qquad \leq\|w\|_{H^{3/2}} \left( \|\nabla \RSZ^\perp \phi\|_{L^4} \|\nabla w\|_{L^2} + C \|\Lambda^{1/2} \RSZ^\perp w\|_{L^4} \|\nabla \theta\|_{L^4} + C \|\RSZ^\perp w\|_{L^4} \|\Lambda^{1/2} \nabla \theta \|_{L^4} \right) \notag\\
&\qquad \leq \frac{\kappa}{2} \|w\|_{H^{3/2}}^2 + \frac{C}{\kappa} \|w\|_{H^1}^2 \left( \|\phi\|_{H^{3/2}}^2 + \|\theta\|_{H^{2}}^2  \right).
\label{eq:w:ODE:1}
\end{align}
The Gr\"onwall inequality and the bounds \eqref{eq:M:32:A}--\eqref{eq:M:2:A} for $\theta,\phi \in \AA$ then yield
\begin{align}
\|w(t) \|_{H^1}^2 \leq \|\xi_0\|_{H^1}^2 \exp \left( C \kappa^{-1} t M_{\AA}^2\right) \leq \|\xi_0\|_{H^1}^2 K(t,M_\AA)
\label{eq:w:H1}
\end{align}
for all $t\geq 0$.  Here and throughout the proof $K(\cdot,\cdot)$ is an increasing continuous function in each variable. This function {\em may change from line to line}. 

Inserting the estimate \eqref{eq:w:H1} back into \eqref{eq:w:ODE:1} gives
\begin{align}
\int_0^t \|w(s)\|_{H^{3/2}}^2 ds \leq \|\xi_0\|_{H^1}^2C \kappa^{-1}  \left( 1 +  \kappa^{-1} t M_{\AA}^2 \exp( C \kappa^{-1} t M_{\AA}^2) \right) \leq  \|\xi_0\|_{H^1}^2 K(t,M_\AA)
\label{eq:w:H32:weak}
\end{align}
for $t\geq 0$. Estimate \eqref{eq:w:H32:weak} can be upgraded to a pointwise in time bound, at the cost of losing the dependence of $\xi_0$ in $H^1$. Since 
\begin{align}
\| w(0)\|_{H^{3/2}} = \| \xi_0 \|_{H^{3/2}} = \| \phi_0 - \theta_0 \|_{H^{3/2}} \leq 2 M_\AA
\label{eq:w0:H32}
\end{align}
we are justified to study the  time evolution of $\|w(t)\|_{H^{3/2}}$. Taking an $L^2$ inner product of \eqref{eq:w:PDE} with $\Lambda^3 w$, we obtain
\begin{align*}
&\frac 12 \frac{d}{dt} \|w\|_{H^{3/2}}^2  + \kappa \|w\|_{H^2}^2 \notag\\
&\qquad \leq \|\Lambda^{3/2} w\|_{L^2} \| [ \Lambda^{3/2}, \RSZ^\perp \phi \cdot \nabla] w \|_{L^2} + \|\Lambda^2 w\|_{L^2} \|\Lambda (\RSZ^\perp w \cdot \nabla \theta )\|_{L^2}\notag\\
&\qquad \leq \|w\|_{H^{3/2}} \left( \|\Lambda^{3/2} \RSZ^\perp \phi\|_{L^4} \|\nabla w\|_{L^4}  + \|\nabla \RSZ^\perp \phi\|_{L^4} \|\Lambda^{1/2} \nabla w\|_{L^4} \right) \notag\\
&\qquad  \qquad \qquad + \|w\|_{H^2} \left( \|\nabla^\perp w\|_{L^4} \| \nabla \theta\|_{L^4}  +   \|\RSZ^\perp w\|_{L^\infty} \| \Lambda \nabla \theta\|_{L^2}\right)\notag\\
&\qquad \leq \frac{\kappa}{2} \|w\|_{H^2}^2 + \|w\|_{H^{3/2}}^2 \left( \|\phi\|_{H^2} + \frac{C}{\kappa} \|\theta\|_{H^2}^2  \right)
\end{align*} 
by using Lemma~\ref{lem:fractional:calculus} and the embedding $H^{1/2} \subset L^4$. It thus follows from the above estimate, the Gr\"onwall inequality, \eqref{eq:M:2:A}, and \eqref{eq:w0:H32}, that
\begin{align}
\|w(t)\|_{H^{3/2}}^2 \leq 4 M_\AA^2 \exp\left( C \kappa^{-1} t (1+ M_\AA^2) \right) \leq   K(t,M_\AA)
\label{eq:w:H32}
\end{align}
for some universal $C>0$.

We now combine \eqref{eq:eta:ODE:1} with  \eqref{eq:w:H1}   and \eqref{eq:w:H32} to obtain
\begin{align}
\frac{d}{dt} \|\eta\|_{H^1}^2 + \kappa \|\eta\|_{H^{3/2}}^2 
&\leq \frac{C}{\kappa} \|\eta\|_{H^1}^2 \|\theta\|_{H^2}^2 +  \frac{C}{\kappa} \|w\|_{H^1}^2 \|w\|_{H^{3/2}}^2 \left(1 + \log\frac{\|w\|_{H^{3/2}}^2}{\|w\|_{H^1}^2} \right) \notag\\
&\leq \frac{C}{\kappa} \|\eta\|_{H^1}^2 \|\theta\|_{H^2}^2 +  \frac{C}{\kappa} \|w\|_{H^1}^{2-a} \|w\|_{H^{3/2}}^{2+a} \notag \\
&\leq \frac{C}{\kappa} \|\eta\|_{H^1}^2 \|\theta\|_{H^2}^2 +    \|\xi_0\|_{H^1}^{2-a} K(t,M_\AA) \|w\|_{H^{3/2}}^{2}
\label{eq:w:ODE:2}
\end{align}
where $a \in (0,1)$ is arbitrary. Using the Gr\"onwall inequality combined with  \eqref{eq:M:2:A},  \eqref{eq:w:H32:weak}, and the fact that $\eta(0) =0$ we conclude from \eqref{eq:w:ODE:2} that
\begin{align} 
\|\eta(t)\|_{H^1}^2 
&\leq \exp\left( \frac{C}{\kappa} \int_0^t \|\theta(s)\|_{H^2}^2 ds \right)  \|\xi_0\|_{H^1}^{2-a} K(t,M_\AA) \int_0^t \|w(s)\|_{H^{3/2}}^{2} ds \notag\\
&\leq \|\xi_0\|_{H^1}^{4-a} K(t,M_\AA)
\label{eq:eta:est:final}
\end{align}
for a suitable function $K$ as described above, and $a \in (0,1)$. This proves that 
\begin{align*} 
\lim_{r\to 0+} \left( \sup_{\theta_0,\phi_0 \in \AA, 0 < \|\xi_0\|_{H^1} \leq r} \frac{\|\eta(t)\|_{H^1}^2}{\|\xi_0\|_{H^1}^2} \right) 
\leq \lim_{r\to 0+} r^{2-a} K(t ,M_\AA) =  \lim_{r\to 0+} e(r,t) = 0
\end{align*}
with $e(r,t) = r^{2-a} K(t ,M_\AA)$, 
and thus \eqref{eq:Frechet:differentiability} holds.

In order to prove \eqref{eq:Frechet:boundedness}, consider $\xi_0$ normalized so that $\|\xi_0\|_{H^1}=1$, and let $\theta_0 \in \AA$ be arbitrary.
Then, using similar estimates as above we have
\begin{align} 
\frac 12 \frac{d}{dt} \|\xi\|_{H^1}^2 + \kappa \|\xi\|_{H^{3/2}}^2 
&\leq C \|\xi\|_{H^{3/2}} \| \xi\|_{H^1} \|\theta\|_{H^{3/2}} + \|\RSZ^\perp \xi\|_{L^\infty} \|\xi\|_{H^1} \|\theta\|_{H^2}\notag \\
&\leq \frac{\kappa}{2} \|\xi\|_{H^{3/2}}^2 + \frac{C}{\kappa} \|\xi\|_{H^1}^2 \|\theta\|_{H^2}^2
\label{eq:xi:ODE:1}
\end{align}
which combined with \eqref{eq:M:2:A} yields
\begin{align} 
\|\xi(t) \|_{H^1}^2 \leq \exp\left(C \kappa^{-1} t M_\AA^2 \right) 
\label{eq:xi:H1}
\end{align}
which indeed proves \eqref{eq:Frechet:boundedness}.

It remains to prove that for any $t>0$ and $\theta_0 \in \AA$, the operator $S'(t,\theta_0)$ is compact. Without loss of generality we may look at the image under $S'(t,\theta_0)$ of the unit ball in $H^1$, and show it is precompact. More precisely, we show that the image of this ball is included in ball in $H^{3/2}$. Combine \eqref{eq:xi:ODE:1} and \eqref{eq:xi:H1} to obtain that
\begin{align*} 
\int_0^t \|\xi(s)\|_{H^{3/2}}^2 ds \leq K(t,M_\AA)
\end{align*}
for any $t>0$. By the mean value theorem, there exits $\tau \in (0,t/2)$ such that 
\begin{align} 
\|\xi(\tau)\|_{H^{3/2}}^2 \leq \frac{1}{t} K(t,M_\AA)
\label{eq:xi:H32:tau}.
\end{align}
Taking the inner product of \eqref{eq:L:xi:def} with $\Lambda^3 \xi$, using the usual commutator, Sobolev, and Poincar\'e inequalities, we obtain 
\begin{align} 
\frac 12 \frac{d}{dt} \|\xi\|_{H^{3/2}}^2 + \kappa \|\xi\|_{H^2}^2
&\leq \| [\Lambda^{3/2}, \RSZ^\perp \theta \cdot \nabla] \xi\|_{L^2} \|\Lambda^{3/2} \xi\|_{L^2} + \|\Lambda( \RSZ^\perp \xi \cdot \nabla \theta )\|_{L^2} \|\Lambda^{3/2} \xi\|_{L^2} \notag\\
&\leq C \|\theta\|_{H^2} \|\xi\|_{H^{3/2}}^2 + C \|\theta\|_{H^{3/2}} \|\xi\|_{H^2} \|\xi\|_{H^{3/2}} + C \|\RSZ^\perp \xi\|_{L^\infty} \|\theta\|_{H^2} \|\xi\|_{H^2} \notag\\
&\leq \frac{\kappa}{2} \|\xi\|_{H^2}^2 + \frac{C}{\kappa} \|\xi\|_{H^{3/2}}^2 \|\theta\|_{H^2}^2
\label{eq:xi:ODE:2}
\end{align}
at times larger than the $\tau$ in \eqref{eq:xi:H32:tau}. Integrating \eqref{eq:xi:ODE:2} between $\tau$ and $t$, and using the bound
\eqref{eq:M:2:A} and \eqref{eq:xi:H32:tau}, we thus obtain
\begin{align*} 
\|\xi(t)\|_{H^{3/2}}^2 
&\leq \|\xi(\tau)\|_{H^{3/2}}^2 \exp\left( \frac{C}{\kappa} \int_{\tau}^t \|\theta(s)\|_{H^2}^2 ds \right) \leq \frac{1}{t} K(t,M_\AA)
\end{align*}
for a suitable function $K$ which is continuous and increasing in all its parameters. This concludes the proof of the Proposition.
\end{proof}

\bigskip
{\bf Acknowledgements.}

The work of PC was supported in part by NSF
grants DMS-1209394, DMS-1265132, and DMS-1240743. The work of AT
was supported in part by the NSF GRFP grant. The work of VV was
supported in part by the NSF grant DMS-1211828.

\newcommand{\etalchar}[1]{$^{#1}$}


\begin{thebibliography}{FMRT01}

\bibitem[AN67]{AgmonNirenberg67}
S.~Agmon and L.~Nirenberg.
\newblock Lower bounds and uniqueness theorems for solutions of differential
  equations in a {H}ilbert space.
\newblock {\em Comm. Pure Appl. Math.}, 20:207--229, 1967.

\bibitem[Ber02]{Berselli02}
L.C. Berselli.
\newblock Vanishing viscosity limit and long-time behavior for 2{D}
  quasi-geostrophic equations.
\newblock {\em Indiana Univ. Math. J.}, 51(4):905--930, 2002.

\bibitem[BV92]{BabinVishik92}
A.V. Babin and M.I. Vishik.
\newblock {\em Attractors of evolution equations}, volume~25 of {\em Studies in
  Mathematics and its Applications}.
\newblock North-Holland Publishing Co., Amsterdam, 1992.
\newblock Translated and revised from the 1989 Russian original by Babin.

\bibitem[CC04]{CordobaCordoba04}
A.~C{\'o}rdoba and D.~C{\'o}rdoba.
\newblock A maximum principle applied to quasi-geostrophic equations.
\newblock {\em Comm. Math. Phys.}, 249(3):511--528, 2004.

\bibitem[CCV11]{CaffarelliChanVasseur11}
L.~Caffarelli, C.H. Chan, and A.~Vasseur.
\newblock Regularity theory for parabolic nonlinear integral operators.
\newblock {\em J. Amer. Math. Soc.}, 24(3):849--869, 2011.

\bibitem[CCW01]{ConstantinCordobaWu01}
P.~Constantin, D.~C{\'o}rdoba, and J.~Wu.
\newblock On the critical dissipative quasi-geostrophic equation.
\newblock {\em Indiana Univ. Math. J.}, 50(Special Issue):97--107, 2001.
\newblock Dedicated to Professors Ciprian Foias and Roger Temam (Bloomington,
  IN, 2000).

\bibitem[CF85]{ConstantinFoias85}
P.~Constantin and C.~Foias.
\newblock Global {L}yapunov exponents, {K}aplan-{Y}orke formulas and the
  dimension of the attractors for {$2$}{D} {N}avier-{S}tokes equations.
\newblock {\em Comm. Pure Appl. Math.}, 38(1):1--27, 1985.

\bibitem[CF88]{ConstantinFoias88}
P.~Constantin and C.~Foias.
\newblock {\em Navier-{S}tokes equations}.
\newblock Chicago Lectures in Mathematics. University of Chicago Press,
  Chicago, IL, 1988.

\bibitem[CF01]{CordobaFefferman01}
D.~C{\'o}rdoba and C.~Fefferman.
\newblock {Behavior of several two-dimensional fluid equations in singular
  scenarios.}
\newblock {\em Proc. Natl. Acad. Sci. USA}, 98(8):4311--4312, 2001.

\bibitem[CF02]{CordobaFefferman02}
D.~C{\'o}rdoba and C.~Fefferman.
\newblock Growth of solutions for {QG} and 2{D} {E}uler equations.
\newblock {\em J. Amer. Math. Soc.}, 15(3):665--670, 2002.

\bibitem[CFMR05]{CordobaFontelosManchoRodrigo05}
D.~C{{\'o}}rdoba, M.A. Fontelos, A.M. Mancho, and J.L. Rodrigo.
\newblock Evidence of singularities for a family of contour dynamics equations.
\newblock {\em Proc. Natl. Acad. Sci. USA}, 102(17):5949--5952, 2005.

\bibitem[CFMT85]{ConstantinFoiasManleyTemam85}
P.~Constantin, C.~Foias, O.~P. Manley, and R.~Temam.
\newblock Determining modes and fractal dimension of turbulent flows.
\newblock {\em J. Fluid Mech.}, 150:427--440, 1985.

\bibitem[CFT85]{ConstantinFoiasTemam85}
P.~Constantin, C.~Foias, and R.~Temam.
\newblock Attractors representing turbulent flows.
\newblock {\em Mem. Amer. Math. Soc.}, 53(314):vii+67, 1985.

\bibitem[CFT88]{ConstantinFoiasTemam88}
P.~Constantin, C.~Foias, and R.~Temam.
\newblock On the dimension of the attractors in two-dimensional turbulence.
\newblock {\em Phys. D}, 30(3):284--296, 1988.

\bibitem[CGHV13]{ConstantinGlattHoltzVicol13}
P.~Constantin, N.~Glatt-Holtz, and V.~Vicol.
\newblock Unique ergodicity for fractionally dissipated, stochastically forced
  2d Euler equations.
\newblock {\em arXiv preprint arXiv:1304.2022}, 2013.

\bibitem[Cha08]{Chae08}
D.~Chae.
\newblock The geometric approaches to the possible singularities in the
  inviscid fluid flows.
\newblock {\em J. Phys. A}, 41(36):365501, 11, 2008.

\bibitem[CJT97]{CockburnJonesTiti97}
B.~Cockburn, D.A. Jones, and E.S. Titi.
\newblock Estimating the number of asymptotic degrees of freedom for nonlinear
  dissipative systems.
\newblock {\em Math. Comp.}, 66(219):1073--1087, 1997.

\bibitem[CL03]{ChaeLee03}
D.~Chae and J.~Lee.
\newblock Global well-posedness in the super-critical dissipative
  quasi-geostrophic equations.
\newblock {\em Comm. Math. Phys.}, 233(2):297--311, 2003.

\bibitem[CLS{\etalchar{+}}12]{ConstantinLaiSharmaTsengWu12}
P.~Constantin, M.-C. Lai, R.~Sharma, Y.-H. Tseng, and J.~Wu.
\newblock New numerical results for the surface quasi-geostrophic equation.
\newblock {\em J. Sci. Comput.}, 50(1):1--28, 2012.

\bibitem[CMT94]{ConstantinMajdaTabak94}
P.~Constantin, A.J. Majda, and E.~Tabak.
\newblock Formation of strong fronts in the {$2$}-{D} quasigeostrophic thermal
  active scalar.
\newblock {\em Nonlinearity}, 7(6):1495--1533, 1994.

\bibitem[CMZ07]{ChenMiaoZhang07}
Q.~Chen, C.~Miao, and Z.~Zhang.
\newblock A new {B}ernstein's inequality and the 2{D} dissipative
  quasi-geostrophic equation.
\newblock {\em Comm. Math. Phys.}, 271(3):821--838, 2007.

\bibitem[CNS98]{ConstantinNieSchorghofer98}
P.~Constantin, Q.~Nie, and N.~Sch{{\"o}}rghofer.
\newblock Nonsingular surface quasi-geostrophic flow.
\newblock {\em Phys. Lett. A}, 241(3):168--172, 1998.

\bibitem[Con87]{Constantin87}
P.~Constantin.
\newblock Collective {$L^\infty$} estimates for families of functions with
  orthonormal derivatives.
\newblock {\em Indiana Univ. Math. J.}, 36(3):603--616, 1987.

\bibitem[Con02]{Constantin02}
P.~Constantin.
\newblock Energy spectrum of quasigeostrophic turbulence.
\newblock {\em Physical Review Letters}, 89(18):184501, 2002.

\bibitem[Con06]{Constantin06}
P.~Constantin.
\newblock Euler equations, {N}avier-{S}tokes equations and turbulence.
\newblock In {\em Mathematical foundation of turbulent viscous flows}, volume
  1871 of {\em Lecture Notes in Math.}, pages 1--43. Springer, Berlin, 2006.

\bibitem[Cor98]{Cordoba98}
D.~C{\'o}rdoba.
\newblock Nonexistence of simple hyperbolic blow-up for the quasi-geostrophic
  equation.
\newblock {\em Ann. of Math. (2)}, 148(3):1135--1152, 1998.

\bibitem[CS07]{CaffarelliSilvestre07}
L.~Caffarelli and L.~Silvestre.
\newblock An extension problem related to the fractional {L}aplacian.
\newblock {\em Comm. Partial Differential Equations}, 32(7-9):1245--1260, 2007.

\bibitem[CTV13]{ConstantinTarfuleaVicol13}
P.~Constantin, A.~Tarfulea, and V.~Vicol.
\newblock Absence of anomalous dissipation of energy in forced two dimensional
  fluid equations.
\newblock {\em arXiv preprint arXiv:1305.7089}, 2013.

\bibitem[CV02]{ChepyzhovVishik02}
V.V. Chepyzhov and M.I. Vishik.
\newblock {\em Attractors for equations of mathematical physics}, volume~49.
\newblock American Mathematical Society Providence, RI, 2002.

\bibitem[CV10a]{CaffarelliVasseur10}
L.A. Caffarelli and A.~Vasseur.
\newblock Drift diffusion equations with fractional diffusion and the
  quasi-geostrophic equation.
\newblock {\em Ann. of Math. (2)}, 171(3):1903--1930, 2010.

\bibitem[CV10b]{CaffarelliVasseur10b}
L.A. Caffarelli and A.F. Vasseur.
\newblock The {D}e {G}iorgi method for regularity of solutions of elliptic
  equations and its applications to fluid dynamics.
\newblock {\em Discrete Contin. Dyn. Syst. Ser. S}, 3(3):409--427, 2010.

\bibitem[CV12]{ConstantinVicol12}
P.~Constantin and V.~Vicol.
\newblock Nonlinear maximum principles for dissipative linear nonlocal
  operators and applications.
\newblock {\em Geometric And Functional Analysis}, 22(5):1289--1321, 2012.

\bibitem[CW99]{ConstantinWu99}
P.~Constantin and J.~Wu.
\newblock Behavior of solutions of 2{D} quasi-geostrophic equations.
\newblock {\em SIAM J. Math. Anal.}, 30(5):937--948, 1999.

\bibitem[CW08]{ConstantinWu08}
P.~Constantin and J.~Wu.
\newblock Regularity of {H}{\"o}lder continuous solutions of the supercritical
  quasi-geostrophic equation.
\newblock {\em Ann. Inst. H. Poincar{\'e} Anal. Non Lin{\'e}aire},
  25(6):1103--1110, 2008.

\bibitem[CW09]{ConstantinWu09}
P.~Constantin and J.~Wu.
\newblock H{\"o}lder continuity of solutions of supercritical dissipative
  hydrodynamic transport equations.
\newblock {\em Ann. Inst. H. Poincar{\'e} Anal. Non Lin{\'e}aire},
  26(1):159--180, 2009.

\bibitem[CZ54]{CalderonZygmund54}
A.P. Calder{{\'o}}n and A.~Zygmund.
\newblock Singular integrals and periodic functions.
\newblock {\em Studia Math.}, 14:249--271, 1954.

\bibitem[Dab11]{Dabkowski11}
M.~Dabkowski.
\newblock Eventual regularity of the solutions to the supercritical dissipative
  quasi-geostrophic equation.
\newblock {\em Geom. Funct. Anal.}, 21(1):1--13, 2011.

\bibitem[DD08]{DongDu08}
H.~Dong and D.~Du.
\newblock Global well-posedness and a decay estimate for the critical
  dissipative quasi-geostrophic equation in the whole space.
\newblock {\em Discrete Contin. Dyn. Syst.}, 21(4):1095--1101, 2008.

\bibitem[DG91]{DoeringGibbon91}
C.R. Doering and J.D. Gibbon.
\newblock Note on the {C}onstantin-{F}oias-{T}emam attractor dimension estimate
  for two-dimensional turbulence.
\newblock {\em Phys. D}, 48(2-3):471--480, 1991.

\bibitem[DHLY06]{DengHouLiYu06}
J.~Deng, T.Y. Hou, R.~Li, and X.~Yu.
\newblock Level set dynamics and the non-blowup of the 2{D} quasi-geostrophic
  equation.
\newblock {\em Methods Appl. Anal.}, 13(2):157--180, 2006.

\bibitem[DKSV12]{DabkowskiKiselevSilvestreVicol12}
M.~Dabkowski, A.~Kiselev, L.~Silvestre, and V.~Vicol.
\newblock Global well-posedness of slightly supercritical active scalar
  equations. Analysis and PDE, to appear.
\newblock {\em arXiv:1203.6302v1 [math.AP]}, 03 2012.

\bibitem[DKV12]{DabkowskiKiselevVicol12}
M.~Dabkowski, A.~Kiselev, and V.~Vicol.
\newblock Global well-posedness for a slightly supercritical surface
  quasi-geostrophic equation.
\newblock {\em Nonlinearity}, 25(5):1525--1535, 2012.

\bibitem[DL08]{DongLi08}
H.~Dong and D.~Li.
\newblock Spatial analyticity of the solutions to the subcritical dissipative
  quasi-geostrophic equations.
\newblock {\em Arch. Ration. Mech. Anal.}, 189(1):131--158, 2008.

\bibitem[DNPV11]{DiNeazzaPalatucciValdinoci11}
E.~Di~Nezza, G.~Palatucci, and E.~Valdinoci.
\newblock Hitchhiker's guide to the fractional sobolev spaces.
\newblock {\em Bull. Sci. Math. to appear. arXiv:1104.4345 [math.FA]}, 2011.

\bibitem[Don10]{Dong10}
H.~Dong.
\newblock Dissipative quasi-geostrophic equations in critical {S}obolev spaces:
  smoothing effect and global well-posedness.
\newblock {\em Discrete Contin. Dyn. Syst.}, 26(4):1197--1211, 2010.

\bibitem[DP09a]{DongPavlovic09}
H.~Dong and N.~Pavlovi{{\'c}}.
\newblock Regularity criteria for the dissipative quasi-geostrophic equations
  in {H}{\"o}lder spaces.
\newblock {\em Comm. Math. Phys.}, 290(3):801--812, 2009.

\bibitem[DP09b]{DongPavlovic09b}
H.~Dong and N.~Pavlovi{{\'c}}.
\newblock A regularity criterion for the dissipative quasi-geostrophic
  equations.
\newblock {\em Ann. Inst. H. Poincar{\'e} Anal. Non Lin{\'e}aire},
  26(5):1607--1619, 2009.

\bibitem[FK95]{FoiasKukavica95}
C.~Foias and I.~Kukavica.
\newblock Determining nodes for the {K}uramoto-{S}ivashinsky equation.
\newblock {\em J. Dynam. Differential Equations}, 7(2):365--373, 1995.

\bibitem[FMRT01]{FoiasManleyRosaTemam01}
C.~Foias, O.~Manley, R.~Rosa, and R.~Temam.
\newblock {\em Navier-{S}tokes equations and turbulence}, volume~83 of {\em
  Encyclopedia of Mathematics and its Applications}.
\newblock Cambridge University Press, Cambridge, 2001.

\bibitem[FMT88]{FoiasManleyTemam88}
C.~Foias, O.~Manley, and R.~Temam.
\newblock Modelling of the interaction of small and large eddies in
  two-dimensional turbulent flows.
\newblock {\em RAIRO Mod{\'e}l. Math. Anal. Num{\'e}r.}, 22(1):93--118, 1988.

\bibitem[FMTT83]{FoiasManleyTemamTreve83}
C.~Foias, O.P. Manley, R.~Temam, and Y.M. Tr{{\`e}}ve.
\newblock Asymptotic analysis of the {N}avier-{S}tokes equations.
\newblock {\em Phys. D}, 9(1-2):157--188, 1983.

\bibitem[FP67]{FoiasProdi67}
C.~Foias and G.~Prodi.
\newblock Sur le comportement global des solutions non-stationnaires des
  {\'e}quations de {N}avier-{S}tokes en dimension {$2$}.
\newblock {\em Rend. Sem. Mat. Univ. Padova}, 39:1--34, 1967.

\bibitem[FPV09]{FriedlanderPavlovicVicol09}
S.~Friedlander, N.~Pavlovi{{\'c}}, and V.~Vicol.
\newblock Nonlinear instability for the critically dissipative
  quasi-geostrophic equation.
\newblock {\em Comm. Math. Phys.}, 292(3):797--810, 2009.

\bibitem[FR11]{FeffermanRodrigo11}
C.~Fefferman and J.L. Rodrigo.
\newblock Analytic sharp fronts for the surface quasi-geostrophic equation.
\newblock {\em Comm. Math. Phys.}, 303(1):261--288, 2011.

\bibitem[FST85]{FoiasSellTemam85}
C.~Foias, G.R. Sell, and R.~Temam.
\newblock Vari{\'e}t{\'e}s inertielles des {\'e}quations diff{\'e}rentielles
  dissipatives.
\newblock {\em C. R. Acad. Sci. Paris S{\'e}r. I Math.}, 301(5):139--141, 1985.

\bibitem[FST88]{FoaisSellTemam88}
C.~Foias, G.R. Sell, and R.~Temam.
\newblock Inertial manifolds for nonlinear evolutionary equations.
\newblock {\em J. Differential Equations}, 73(2):309--353, 1988.

\bibitem[FT84]{FoiasTemam84}
C.~Foias and R.~Temam.
\newblock Determination of the solutions of the {N}avier-{S}tokes equations by
  a set of nodal values.
\newblock {\em Math. Comp.}, 43(167):117--133, 1984.

\bibitem[FT91]{FoiasTiti91}
C.~Foias and E.S. Titi.
\newblock Determining nodes, finite difference schemes and inertial manifolds.
\newblock {\em Nonlinearity}, 4(1):135--153, 1991.

\bibitem[GT97]{GibbonTiti97}
J.D. Gibbon and E.S. Titi.
\newblock Attractor dimension and small length scale estimates for the
  three-dimensional {N}avier-{S}tokes equations.
\newblock {\em Nonlinearity}, 10(1):109--119, 1997.

\bibitem[Hal88]{Hale88}
J.K. Hale.
\newblock {\em Asymptotic behavior of dissipative systems}, volume~25 of {\em
  Mathematical Surveys and Monographs}.
\newblock American Mathematical Society, Providence, RI, 1988.

\bibitem[HK07]{HmidiKeraani07}
T.~Hmidi and S.~Keraani.
\newblock Global solutions of the super-critical 2{D} quasi-geostrophic
  equation in {B}esov spaces.
\newblock {\em Adv. Math.}, 214(2):618--638, 2007.

\bibitem[HPGS95]{HeldPierrehumbertGarnerSwanson95}
I.M. Held, R.T. Pierrehumbert, S.T. Garner, and K.L. Swanson.
\newblock Surface quasi-geostrophic dynamics.
\newblock {\em J. Fluid Mech.}, 282:1--20, 1995.

\bibitem[JT92]{JonesTiti92}
D.A. Jones and E.S. Titi.
\newblock On the number of determining nodes for the {$2$}{D} {N}avier-{S}tokes
  equations.
\newblock {\em J. Math. Anal. Appl.}, 168(1):72--88, 1992.

\bibitem[JT93]{JonesTiti93}
D.A. Jones and E.S. Titi.
\newblock Upper bounds on the number of determining modes, nodes, and volume
  elements for the {N}avier-{S}tokes equations.
\newblock {\em Indiana Univ. Math. J.}, 42(3):875--887, 1993.

\bibitem[Ju04]{Ju04}
N.~Ju.
\newblock Existence and uniqueness of the solution to the dissipative 2{D}
  quasi-geostrophic equations in the {S}obolev space.
\newblock {\em Comm. Math. Phys.}, 251(2):365--376, 2004.

\bibitem[Ju05]{Ju05a}
N.~Ju.
\newblock The maximum principle and the global attractor for the dissipative
  2{D} quasi-geostrophic equations.
\newblock {\em Comm. Math. Phys.}, 255(1):161--181, 2005.

\bibitem[Ju07]{Ju07}
N.~Ju.
\newblock Dissipative 2{D} quasi-geostrophic equation: local well-posedness,
  global regularity and similarity solutions.
\newblock {\em Indiana Univ. Math. J.}, 56(1):187--206, 2007.

\bibitem[Kis11]{Kiselev11}
A.~Kiselev.
\newblock Nonlocal maximum principles for active scalars.
\newblock {\em Advances in Mathematics}, 227(5):1806--1826, 2011.

\bibitem[KN09]{KiselevNazarov09}
A.~Kiselev and F.~Nazarov.
\newblock A variation on a theme of {C}affarelli and {V}asseur.
\newblock {\em Zap. Nauchn. Sem. S.-Peterburg. Otdel. Mat. Inst. Steklov.
  (POMI)}, 370(Kraevye Zadachi Matematicheskoi Fiziki i Smezhnye Voprosy Teorii
  Funktsii. 40):58--72, 220, 2009.

\bibitem[KN10]{KiselevNazarov10}
A.~Kiselev and F.~Nazarov.
\newblock Global regularity for the critical dispersive dissipative surface
  quasi-geostrophic equation.
\newblock {\em Nonlinearity}, 23(3):549--554, 2010.

\bibitem[KNV07]{KiselevNazarovVolberg07}
A.~Kiselev, F.~Nazarov, and A.~Volberg.
\newblock Global well-posedness for the critical 2{D} dissipative
  quasi-geostrophic equation.
\newblock {\em Invent. Math.}, 167(3):445--453, 2007.

\bibitem[KP88]{KatoPonce88}
T.~Kato and G.~Ponce.
\newblock Commutator estimates and the {E}uler and {N}avier-{S}tokes equations.
\newblock {\em Comm. Pure Appl. Math.}, 41(7):891--907, 1988.

\bibitem[KPV91]{KenigPonceVega91}
C.E. Kenig, G.~Ponce, and L.~Vega.
\newblock Well-posedness of the initial value problem for the {K}orteweg-de
  {V}ries equation.
\newblock {\em J. Amer. Math. Soc.}, 4(2):323--347, 1991.

\bibitem[Kuk92]{Kukavica92}
I.~Kukavica.
\newblock On the number of determining nodes for the {G}inzburg-{L}andau
  equation.
\newblock {\em Nonlinearity}, 5(5):997--1006, 1992.

\bibitem[Kuk96]{Kukavica96}
I.~Kukavica.
\newblock Level sets of the vorticity and the stream function for the {$2$}{D}
  periodic {N}avier-{S}tokes equations with potential forces.
\newblock {\em J. Differential Equations}, 126(2):374--388, 1996.

\bibitem[Kuk07]{Kukavica07}
I.~Kukavica.
\newblock Log-log convexity and backward uniqueness.
\newblock {\em Proc. Amer. Math. Soc.}, 135(8):2415--2421 (electronic), 2007.

\bibitem[KV11]{KukavicaVicol11b}
I.~Kukavica and V.~Vicol.
\newblock On the analyticity and {G}evrey-class regularity up to the boundary
  for the {E}uler equations.
\newblock {\em Nonlinearity}, 24(3):765--796, 2011.

\bibitem[Lad91]{Ladyzhenskaya91}
O.~Ladyzhenskaya.
\newblock {\em Attractors for solution maps and evolution equations}.
\newblock Lezioni Lincee. [Lincei Lectures]. Cambridge University Press,
  Cambridge, 1991.

\bibitem[Mar08]{Marchand08a}
F.~Marchand.
\newblock Existence and regularity of weak solutions to the quasi-geostrophic
  equations in the spaces {$L^p$} or {$\dot H^{-1/2}$}.
\newblock {\em Comm. Math. Phys.}, 277(1):45--67, 2008.

\bibitem[Miu06]{Miura06}
H.~Miura.
\newblock Dissipative quasi-geostrophic equation for large initial data in the
  critical {S}obolev space.
\newblock {\em Comm. Math. Phys.}, 267(1):141--157, 2006.

\bibitem[OS10]{OhkitaniSakajo10}
K.~Ohkitani and T.~Sakajo.
\newblock Oscillatory damping in long-time evolution of the surface
  quasi-geostrophic equations with generalized viscosity: a numerical study.
\newblock {\em Nonlinearity}, 23(12):3029--3051, 2010.

\bibitem[OY97]{OhkitaniYamada97}
K.~Ohkitani and M.~Yamada.
\newblock Inviscid and inviscid-limit behavior of a surface quasigeostrophic
  flow.
\newblock {\em Phys. Fluids}, 9(4):876--882, 1997.

\bibitem[Ped82]{Pedlosky82}
J.~Pedlosky.
\newblock {\em Geophysical Fluid Dynamics}.
\newblock Springer Verlag, 1982.

\bibitem[Res95]{Resnick95}
S.G. Resnick.
\newblock {\em Dynamical problems in non-linear advective partial differential
  equations}.
\newblock ProQuest LLC, Ann Arbor, MI, 1995.
\newblock Thesis (Ph.D.)--The University of Chicago.

\bibitem[Rob01]{Robinson01}
J.~Robinson.
\newblock {\em Infinite-dimensional dynamical systems}.
\newblock Cambridge Texts in Applied Mathematics. Cambridge University Press,
  Cambridge, 2001.
\newblock An introduction to dissipative parabolic PDEs and the theory of
  global attractors.

\bibitem[RS12]{RoncalStinga12}
L.~Roncal and P.R. Stinga.
\newblock Fractional laplacian on the torus.
\newblock {\em arXiv preprint arXiv:1209.6104}, 2012.

\bibitem[Sha64]{Shapiro64}
V.L. Shapiro.
\newblock Fourier series in several variables.
\newblock {\em Bull. Amer. Math. Soc.}, 70:48--93, 1964.

\bibitem[Sil10a]{Silvestre10a}
L.~Silvestre.
\newblock Eventual regularization for the slightly supercritical
  quasi-geostrophic equation.
\newblock {\em Ann. Inst. H. Poincar{\'e} Anal. Non Lin{\'e}aire},
  27(2):693--704, 2010.

\bibitem[Sil10b]{Silvestre10c}
L.~Silvestre.
\newblock Holder estimates for advection fractional-diffusion equations.
\newblock {\em arXiv preprint arXiv:1009.5723}, 2010.


\bibitem[SS03]{SchonbekSchonbek03}
M.E. Schonbek and T.P. Schonbek.
\newblock Asymptotic behavior to dissipative quasi-geostrophic flows.
\newblock {\em SIAM J. Math. Anal.}, 35(2):357--375, 2003.

\bibitem[SV12]{SilvestreVicol12}
L.~Silvestre and V.~Vicol.
\newblock H{\"o}lder continuity for a drift-diffusion equation with pressure.
\newblock {\em Annales de l'Institut Henri Poincare (C) Non Linear Analysis},
  20(4):637--652, 2012.

\bibitem[SVZ13]{SilvestreVicolZlatos13}
L.~Silvestre, V.~Vicol, and A.~Zlato{\v{s}}.
\newblock On the {L}oss of {C}ontinuity for {S}uper-{C}ritical
  {D}rift-{D}iffusion {E}quations.
\newblock {\em Arch. Ration. Mech. Anal.}, 207(3):845--877, 2013.

\bibitem[SW71]{SteinWeiss71}
E.M. Stein and G.~Weiss.
\newblock {\em Introduction to {F}ourier analysis on {E}uclidean spaces}.
\newblock Princeton University Press, Princeton, N.J., 1971.
\newblock Princeton Mathematical Series, No. 32.

\bibitem[Tay91]{Taylor91}
M.E. Taylor.
\newblock {\em Pseudodifferential operators and nonlinear {PDE}}, volume 100 of
  {\em Progress in Mathematics}.
\newblock Birkh{\"a}user Boston Inc., Boston, MA, 1991.

\bibitem[Tem97]{Temam97}
R.~Temam.
\newblock {\em Infinite-dimensional dynamical systems in mechanics and
  physics}, volume~68 of {\em Applied Mathematical Sciences}.
\newblock Springer-Verlag, New York, second edition, 1997.

\bibitem[Tol00]{Toland00}
J.F. Toland.
\newblock Stokes waves in hardy spaces and as distributions.
\newblock {\em Journal de math{\'e}matiques pures et appliqu{\'e}es},
  79(9):901--917, 2000.

\bibitem[WT13]{WangTang13}
M.~Wang and Y.~Tang.
\newblock On dimension of the global attractor for 2d quasi-geostrophic
  equations.
\newblock {\em Nonlinear Analysis: Real World Applications}, 14(4):1887--1895,
  2013.

\bibitem[Wu05]{Wu04}
J.~Wu.
\newblock Global solutions of the 2{D} dissipative quasi-geostrophic equation
  in {B}esov spaces.
\newblock {\em SIAM J. Math. Anal.}, 36(3):1014--1030 (electronic), 2004/05.

\bibitem[Wu01]{Wu01}
J.~Wu.
\newblock Dissipative quasi-geostrophic equations with {$L^p$} data.
\newblock {\em Electron. J. Differential Equations}, pages No. 56, 13, 2001.

\bibitem[XZ12]{XueZheng12}
L.~Xue and X.~Zheng.
\newblock Note on the well-posedness of a slightly supercritical surface
  quasi-geostrophic equation.
\newblock {\em J. Differential Equations}, 253(2):795--813, 2012.

\bibitem[Yu08]{Yu08}
X.~Yu.
\newblock Remarks on the global regularity for the super-critical 2{D}
  dissipative quasi-geostrophic equation.
\newblock {\em J. Math. Anal. Appl.}, 339(1):359--371, 2008.

\bibitem[Zia97]{Ziane97}
M.~Ziane.
\newblock Optimal bounds on the dimension of the attractor of the navier-stokes
  equations.
\newblock {\em Physica D: Nonlinear Phenomena}, 105(1):1--19, 1997.

\end{thebibliography}
\end{document}